\documentclass{article}
\usepackage{amsmath}
  \usepackage{paralist}
  \usepackage{graphics} %% add this and next lines if pictures should be in esp format
  \usepackage{epsfig} %For pictures: screened artwork should be set up with an 85 or 100 line screen
\usepackage{graphicx}  \usepackage{epstopdf}%This is to transfer .eps figure to .pdf figure; please compile your paper using PDFLeTex or PDFTeXify.
 \usepackage[colorlinks=true]{hyperref}
\hypersetup{urlcolor=blue, citecolor=red}

\usepackage{amsthm}
\newtheorem{theorem}{Theorem}[section]
\newtheorem{corollary}{Corollary}
\newtheorem{lemma}[theorem]{Lemma}
\newtheorem{proposition}{Proposition}

\theoremstyle{definition}
\newtheorem{definition}[theorem]{Definition}
\newtheorem{remark}{Remark}

\newtheorem{cond}{Condition}

\usepackage{lipsum}
\usepackage{amsfonts}
\usepackage{graphicx}
\usepackage{epstopdf}
\usepackage{amssymb}
\usepackage{dsfont}
\usepackage{amsopn}

\newcommand{\Supp}{\operatorname{supp}} 
\newcommand{\supp}{\operatorname{supp}}

\usepackage{cases}
\usepackage{algpseudocode}
\usepackage{caption}
\usepackage{algorithm}
\DeclareCaptionFormat{algor}{%
  \hrulefill\par\offinterlineskip\vskip1pt%
    \textbf{#1#2}#3\offinterlineskip\hrulefill}
\DeclareCaptionStyle{algori}{singlelinecheck=off,format=algor,labelsep=space}
\usepackage{tikz}
\usetikzlibrary{patterns}
\usetikzlibrary{arrows}
\tikzset{cross/.style={path picture={
  \draw[black]
    (path picture bounding box.south east)--(path picture bounding box.north west)
    (path picture bounding box.south west)--(path picture bounding box.north east);
}}}

\newcommand{\storto}[5]{
\pgfmoveto{\pgfxy(#1)}
\pgfcurveto{\pgfxy(#2)}{\pgfxy(#2)}{\pgfxy(#3)}
\pgfmoveto{\pgfxy(#3)}
\pgfcurveto{\pgfxy(#4)}{\pgfxy(#4)}{\pgfxy(#5)}
\pgfstroke}

\newcommand{\mb}[1]{\mathbb{#1}}
\newcommand{\mc}[1]{\mathcal{#1}}

\newcommand{\mr}[1]{\mathrm{#1}}

\title%[Minimal time problem for  crowd models] %Use the shortened version of the full title
      {Minimal time for the continuity equation controlled by a localized perturbation of the velocity vector field}

\author{
  Michel Duprez\thanks{CEREMADE,  Universit\'e Paris-Dauphine \& CNRS UMR 7534, Universit\'e PSL, 75016 Paris, France.  (\texttt{mduprez@math.cnrs.fr}).}%\thanks{Corresponding author.}
  \and
  Morgan Morancey\thanks{Aix Marseille Universit\'e, CNRS, Centrale Marseille, I2M, Marseille, France. (\texttt{morgan.morancey@univ-amu.fr}).}
  \and
  Francesco Rossi\thanks{Dipartimento di Matematica ``Tullio Levi-Civita", Universit\`{a} degli Studi di Padova, Via Trieste 63, 35121 Padova, Italy. (\texttt{francesco.rossi@math.unipd.it}). }
}

\begin{document}
\maketitle

%
%
%
%% Enter the first author's name and address:
%\centerline{\scshape Michel Duprez$^*$}
%\medskip
%{\footnotesize
%% please put the address of the first author
% \centerline{Sorbonne Universit\'e, Laboratoire Jacques-Louis Lions}
%   \centerline{4 Place Jussieu}
%   \centerline{75005 Paris, France}
%} % Do not forget to end the {\footnotesize by the sign }
%
%\medskip
%
%\centerline{\scshape Morgan Morancey}
%\medskip
%{\footnotesize
% % please put the address of the second  and third author
% \centerline{Aix Marseille Universit\'e, CNRS, Centrale Marseille, I2M}
%   \centerline{39, rue F. Joliot Curie}
%   \centerline{13453 Marseille Cedex 13, France}
%}
%
%\medskip
%
%\centerline{\scshape Francesco Rossi}
%\medskip
%{\footnotesize
% % please put the address of the second  and third author
% \centerline{Universit\`{a} degli Studi di Padova, Dipartimento di Matematica ``Tullio Levi-Civita"}
%   \centerline{Via Trieste 63}
%   \centerline{35121 Padova, Italy}
%}

%\bigskip
%
%% The name of the associate editor will be entered by an editorial staff
%% "Communicated by the associate editor name" is not needed for special issue.
% \centerline{(Communicated by the associate editor name)}

%The abstract of your paper
\begin{abstract}
In this work, we study the  minimal time to steer a given crowd to a desired
configuration. The control is a vector field, 
representing a perturbation of the crowd velocity, 
localized on a fixed control set.
We will assume that there is no interaction between the agents.

We give a characterization of the minimal time both for microscopic and macroscopic descriptions of a crowd. 
We show that the minimal time to steer one initial configuration to another is related to the condition of having enough mass in the control region 
to feed the desired final configuration.

The construction of the control is explicit, providing
a numerical algorithm for computing it. We finally give some numerical simulations.
\end{abstract}

\section{Introduction and main results}

In recent years, the study of systems describing a crowd of interacting autonomous agents 
has drawn a great interest from the control community.
A better understanding of such interaction phenomena can have a strong impact in several key applications, such as road traffic and egress problems for pedestrians. For a few reviews about this topic, see \textit{e.g.} \cite{axelrod,active1,camazine,CPTbook,helbing,jackson,MT14,SepulchreReview}.

Beside the description of interactions, it is now relevant to study problems of control of crowds, \textit{i.e.} of controlling such systems by acting on few agents, or on a small subset of the configuration space. 
Roughly speaking, basic problems for such models are controllability (\textit{i.e.} reachability of a desired configuration) and optimal control (\textit{i.e.} the minimization of a given functional). We already addressed the controllability problem in \cite{DMR17}, thus identifying reachable configurations for crowd models. The present article deals with the subsequent step, that is the study of a classical optimal control problem: the minimal time to reach a desired configuration.

The nature of the control problem relies on the model used to describe the crowd. Two main classes are widely used. In  microscopic  models, the position of each agent is clearly identified; the crowd dynamics is described by an ordinary differential equation of large dimension, in which coupling terms represent interactions between agents. For control of such models, a large literature is available, see \textit{e.g.} reviews \cite{bullo,kumar,lin}, as well as applications, both to pedestrian crowds \cite{ferscha,luh} and to road traffic \cite{canudas,hegyi}.

In macroscopic models, instead, the idea is to represent the crowd by the spatial density of agents; in this setting, the evolution in time of the density solves a partial differential equation, usually of transport type. 
Nonlocal terms (such as convolutions) model  interactions between  agents. To our knowledge, there exist few studies of control of this family of equations.
In \cite{PRT15}, the authors provide approximate alignment of a crowd described by the macroscopic Cucker-Smale model \cite{CS07}. The control is the acceleration, and it is localized in a control region $\omega$ which moves in time. In a similar situation, a stabilization strategy has been established in  \cite{CPRT17,CPRT17b}, by generalizing the Jurdjevic-Quinn method to partial differential equations. A different approach is given by mean-field type control, \textit{i.e.} control of mean-field equations, and by mean-field games modelling crowds, see \textit{e.g.} \cite{achdou1,achdou2,benoit1,benoit2,carmona,FS}. In this case, problems are often of optimization nature, \textit{i.e.} the goal is to find a control minimizing a given cost, with no final constraint. A notable exception is the so-called planning problem, both in the pure transport setting \cite{OPS,tonon} and with an additional diffusion term \cite{achodu-camilli,porretta2,porretta1}. In this article, there is no individual optimization, thus these results do not apply to our setting.

This article deals with the problem of steering one initial configuration to a final one in minimal time. As stated above, we recently discussed in \cite{DMR17} the problem of controllability for the systems described here, which main results are recalled in Section \ref{section 2.3}.
 We proved that one can approximately steer an initial to a final configuration if they satisfy the Geometric Condition \ref{cond1} recalled below. Roughly speaking, it requires that the whole initial configuration crosses the control set $\omega$ forward in time, and the final one crosses it backward in time. From now on, we will always assume that this condition is satisfied, so to ensure that the final configuration is approximately reachable.

When the controllability property is satisfied, it is then interesting to study minimal time problems. Indeed, from the theoretical point of view, it is the first problem in which optimality conditions can be naturally defined. More related to applications described above, minimal time problems play a crucial role: egress problems can be described in this setting, while traffic control is often described in terms of minimization of (maximal or average) total travel time.
%\textcolor{red}{Je mettrais ici le Remark 1}

 Notice that the minimal time issue in a control problem can be a consequence of various unrelated phenomena. For instance, imposing constraints on the control or on the state can lead to a positive minimal time even if the considered evolution equation enjoys an infinite propagation speed (see for instance \cite{loheac}). On the contrary, considering evolution equations with a finite speed of propagation, such as transport or wave equations, naturally imply a minimal null control time when the control is localized.
Our contribution fits in the second setting.

For microscopic models, the dynamics can be written in terms of finite-dimen\-sional control systems. For this reason, minimal time problems can sometimes be addressed with classical (linear or non-linear) control theory, see \textit{e.g.} \cite{agrabook,BP07,jurdjevic,sontag}. Instead, very few results are known for macroscopic models, that can be recasted in terms of control of the transport equation. The linear case is classical, see \textit{e.g.} \cite{C09}. Instead, more recent developments in the non-linear case (based on generalization of differential inclusions) have been recently described in \cite{cavagnari1,cavagnari3,cavagnari2}.

The originality of our research lies in the constraint given on the control: it is a perturbation of the velocity vector field localized in a given region $\omega$ of the space. Such constraint is highly non-trivial, since the control problem is clearly non-linear even though the uncontrolled dynamics is. To the best of our knowledge, minimal time problems with this constraints have not been studied, neither for microscopic nor for macroscopic models.
We first study  a microscopic  model,
where the crowd is described by a vector with $nd$ components ($n,d\in\mb{N}^*$) representing the positions of $n$ agents in $\mb{R}^d$. The natural (uncontrolled) vector field is denoted by  
$v: \mb{R}^d\rightarrow\mb{R}^d$, assumed Lipschitz and uniformly bounded.
We act on the vector field in a fixed subdomain  $\omega$ of the space, 
which will be a  bounded open connected subset of $\mb{R}^d$. 
The admissible controls are thus functions of the form $\mathds{1}_{\omega}u:\mb{R}^d\times\mb{R}^+\rightarrow\mb{R}^d$.
The control models an external action (\textit{e.g.} dynamic signs for road traffic, or traffic agents in pedestrian crowds) located in a fixed area, that interacts with all agents in such given area. Since $\omega$ is fixed, it can act on predefined locations only.

We then consider the following ordinary differential equation
(microscopic model)
\begin{equation}\label{eq ODE}
\left\{\begin{array}{l}
\dot x_i(t) =v(x_i(t)) + \mathds{1}_{\omega}(x_i(t)) u(x_i(t),t),
\\\noalign{\smallskip}
x_i(0)=x^0_i
\end{array} ~~~i=1,\ldots,n,\right.
\end{equation}
where $X^0:=\{x^0_1,...,x^0_n\}$ is the initial configuration of the crowd. 

We also study a macroscopic model,
where the crowd is represented by its density, that is a time-evolving probability measure $\mu(t)$ defined on the space $\mb{R}^d$.
We consider the same  natural vector field $v$,
control region $\omega$ and admissible controls $\mathds{1}_{\omega}u$.
The dynamics is given by the following  linear transport equation 
(macroscopic model)
\begin{equation}\label{eq:transport}
	\left\{
	\begin{array}{ll}
\partial_t\mu +\nabla\cdot((v+\mathds{1}_{\omega}u)\mu)=0&\mbox{ in }\mb{R}^d\times\mb{R}^+,\\\noalign{\smallskip}
\mu(\cdot,0)=\mu^0&\mbox{ in }\mb{R}^d,\\
	\end{array}
	\right.
\end{equation}
where $\mu^0$ is the initial density of the crowd. The function $v+\mathds{1}_{\omega}u$ represents the vector field acting on $\mu$.

To a microscopic 
configuration $X:=\{x_1,...,x_n\}$, we can associate the empirical  measure 
\begin{equation}\label{eq:emp}
\mu:=\sum_{i=1}^n\frac{1}{n}\delta_{x_i}.
\end{equation}
With this notation, System \eqref{eq ODE} is a particular case of System \eqref{eq:transport}. 
This identification can be applied only if the different microscopic agents are considered identical or interchangeable, as it is often the case for crowd models with a large number of agents. This identification will be used several times in the following, namely to approximate macroscopic models with microscopic ones.

Systems \eqref{eq ODE} and \eqref{eq:transport} are first approximations for crowd modeling, since the uncontrolled vector field $v$ is given, and it does not describe interactions between agents. Nevertheless, it is necessary to understand control properties for such simple equations
as a first step, before dealing with vector fields depending on the crowd itself. 
Thus, in a future work, using the results for linear equations presented here, 
we will study control problems for crowd models with a non-local term $v[x_1,...,x_n]$ in the microscopic model \eqref{eq ODE} and  $v[\mu]$ in the the macroscopic model \eqref{eq:transport}.

From now on, we always assume that the following geometric condition is satisfied:

\begin{cond}[Geometric condition]\label{cond1}
Let  $\mu^0,\mu^1$ be two probability measur\-es on $\mb{R}^d$ satisfying:
\begin{enumerate}
\item[(i)] For each  $x^0\in \supp(\mu^0)$, 
there exists $t^0>0$ such that $\Phi_{t^0}^v(x^0)\in \omega,$
where $\Phi_{t}^v$ is the \textit{flow} associated to $v$
(see Definition \ref{def:flow} below).

\item[(ii)] For each $x^1\in \supp(\mu^1)$,  
there exists $t^1>0$ such that $\Phi_{-t^1}^{v}(x^1)\in \omega$.
\end{enumerate}
\end{cond}

The Geometric Condition \ref{cond1} means that each particle crosses the control region.
It is illustrated in Figure \ref{fig:cond geo}.
It is the minimal condition that we can expect to steer any initial condition to 
any target.
Indeed, if Item (i) of the Geometric Condition \ref{cond1} is not satisfied, then 
there exists a whole subpopulation of the measure $\mu^0$ 
 that never intersects the control region,
thus the localized control cannot act on it. 
\begin{figure}[htb]
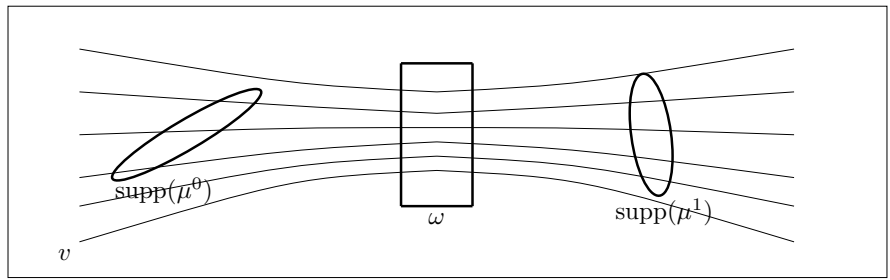

\begin{center}
\scalebox{0.95}{
\begin{pgfpictureboxed}{0cm}{0cm}{12.3 cm}{3.8cm}
\begin{pgfscope}
\pgfsetstartarrow{\pgfarrowswap{\pgfarrowto}}
\pgfsetendarrow{\pgfarrowto}
\pgfsetlinewidth{.2pt}
\storto{1,0.5}{4,1.4}{6,2-0.5}{8,1.9-0.5}{11,1-0.5}
\storto{1,1.5-0.5}{4,2.1-0.5}{6,2.2-0.5}{8,2.1-0.5}{11,1.5-0.5}
\storto{1,1.9-0.5}{4,2.3-0.5}{6,2.4-0.5}{8,2.3-0.5}{11,1.9-0.5}
\storto{1,2.5-0.5}{4,2.6-0.5}{6,2.6-0.5}{8,2.6-0.5}{11,2.5-0.5}
\storto{1,3.1-0.5}{4,2.9-0.5}{6,2.8-0.5}{8,2.9-0.5}{11,3.1-0.5}
\storto{1,3.7-0.5}{4,3.2-0.5}{6,3.1-0.5}{8,3.2-0.5}{11,3.7-0.5}
\end{pgfscope}
\begin{pgfscope}
\pgfsetlinewidth{1pt}
\pgfline{\pgfxy(5.5,1)}{\pgfxy(5.5,3)}
\pgfline{\pgfxy(6.5,1)}{\pgfxy(6.5,3)}
\pgfline{\pgfxy(5.5,1)}{\pgfxy(6.5,1)}
\pgfline{\pgfxy(5.5,3)}{\pgfxy(6.5,3)}
\pgfellipse[stroke]{\pgfxy(2.5,2)}{\pgfxy(0.3,0.4)}{\pgfxy(1,0.5)}
\pgfellipse[stroke]{\pgfxy(9,2)}{\pgfxy(0,.8)}{\pgfxy(-.3,0.3)}
\end{pgfscope}
\pgfputat{\pgfxy(1.5,1.4)}{\pgfbox[left,top]{$\supp(\mu^0)$}}
\pgfputat{\pgfxy(6,.9)}{\pgfbox[center,top]{$\omega$}}
\pgfputat{\pgfxy(8.5,1.1)}{\pgfbox[left,top]{$\supp(\mu^1)$}}
\pgfputat{\pgfxy(0.7,0.4)}{\pgfbox[left,top]{$v$}}
\end{pgfpictureboxed}}
\end{center}
\caption{Geometric condition.}\label{fig:cond geo}
\end{figure}

To ensure the well-posedness of Systems \eqref{eq ODE} and \eqref{eq:transport}, we fix regularity conditions on the natural vector field $v$ and the control  $\mathds{1}_{\omega}u$ as follows:
\begin{cond}[Carath\'eodory condition]\label{cond:cata}
Let $v$ and $\mathds{1}_{\omega}u$ be such that
\begin{itemize}
\item[(i)] The applications 
$x\mapsto v(x)$ and
$x\mapsto \mathds{1}_{\omega}u(x,t)$ for each $t\in\mathbb{R}$ are Lipschitz.
\item[(ii)] For all $x\in\mathbb{R}^d$, the application $t\mapsto \mathds{1}_{\omega}u(x,t)$ is measurable.
\item[(iii)] There exists $M>0$ such that $\|v\|_{\infty}\leqslant M$ and $\|\mathds{1}_{\omega}u\|_{\infty}\leqslant M$.
\end{itemize}
\end{cond}

For arbitrary macroscopic models, one can expect  approximate 
controllability only, since for general measures there exists no homeomorphism sending one to another.
 Indeed, if we impose the  Carath\'eodory condition,  then the flow $\Phi^{v+\mathds{1}_{\omega}u}_t$ is an homeomorphism (see \cite[Th. 2.1.1]{BP07}). For more details, see Remark \ref{rmq:exact}.

Similarly, in the microscopic case, the control vector field $\mathds{1}_{\omega}u$ cannot separate points, due to uniqueness of the solution of System \eqref{eq ODE}.  In the microscopic model, we then assume that the initial configuration $X^0$ 
and the final one $X^1$ are 
disjoint, in the following sense:
\begin{definition} \label{def:disjoint}
A configuration $X=\{x_1,...,x_n\}$ is disjoint if $x_i\neq x_j$ for all $i\neq j$.
\end{definition}
In other words, in a disjoint configuration, two agents cannot lie in the same point of the space. 
Since we deal with velocities $v+\mathds{1}_{\omega}u$ satisfying the Carath\'eo\-dory condition, if $X^0$ is a disjoint configuration,
then the solution $X(t)$ to microscopic System \eqref{eq ODE} is also a disjoint configuration at each time $t\geqslant 0$.

In this article, we aim to study the minimal time problem. We denote by  $T_a$ the minimal time 
to approximately steer   the initial configuration 
 $\mu^0$ to a final one $\mu^1$ in the following sense:
it is the infimum of times for which there exists a control $\mathds{1}_{\omega}u:\mb{R}^d\times\mb{R}^+\rightarrow\mb{R}^d$ satisfying the Carath\'eodory condition steering $\mu^0$ arbitrarily close to $\mu^1$.
We similarly define  the minimal time $T_e$  to exactly steer   the initial configuration  $\mu^0$ to a final one $\mu^1$.
A precise definition will be given below. Since the minimal time is not always reached, we will speak about \textit{infimum time}.

In the sequel, we will use the following notation for all $x\in\mb{R}^d$:
\begin{equation}\begin{array}{lll}
\overline{t}^0(x):=\inf\{t\in\mb{R}^+:\Phi_t^v(x)\in\overline{\omega}\},&\quad&
\overline{t}^1(x):=\inf\{t\in\mb{R}^+:\Phi_{-t}^v(x)\in\overline{\omega}\},\\\noalign{\smallskip}
t^0(x):=\inf\{t\in\mb{R}^+:\Phi_t^v(x)\in\omega\},&~&
t^1(x):=\inf\{t\in\mb{R}^+:\Phi_{-t}^v(x)\in\omega\}.
\end{array}\label{e-temps}
\end{equation}
The quantity $\overline{t}^0(x)$ is the infimum time at which the particle localized at point $x$ with the velocity $v$ belongs to $\overline{\omega}$.  Idem for the other quantities.
Under the Geometric Condition \ref{cond1}, $\overline{t}^0(x)$ and $t^0(x)$ are finite for all $x\in\supp(\mu^0)$, and similarly for $\overline{t}^1(x)$ and $t^1(x)$ when $x\in\supp(\mu^1)$. 
Moreover, it is clear that it always holds $\overline{t}^j(x)\leqslant t^j(x)$.
In some situations, this inequality can be strict. For example, in Figure \ref{fig:ex CE CA},
it holds $\overline{t}^1(x^1_1)<t^1(x^1_1)$. 
It is also important to remark that these quantities are not continuous with respect to $x$.
Indeed, for instance, in Figure \ref{fig:ex CE CA}, if we shift $x^1_1$ to the right (resp. to the left),
we observe a discontinuity of $\overline{t}^1$ (resp. $t^1$) at the point $x^1_1$.% \textcolor{red}{Cette image peut occuper moins de place}

If the set $\omega$ admits everywhere an outer normal vector $n$, \textit{e.g.} 
when it is $\mathcal{C}^1$, one can have $\overline{t}^1(x^1_1)<t^1(x^1_1)$ only if %
$v(\Phi^v_{-\overline{t}^1(x^1_1)}(x^1_1))\cdot n=0.$
This means that the trajectory $\Phi^v_{-t}(x^1_1)$ touches the boundary of $\omega$ with a parallel tangent vector. The proof of such simple statement can be recovered by applying the implicit function theorem.
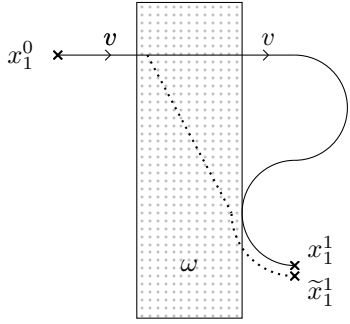
\begin{figure}[ht]
\begin{center}
\begin{tikzpicture}[scale=0.7]
\fill[pattern=dots,opacity = 0.5] (-1,-3) -- (1,-3) -- (1,3) -- (-1,3) -- cycle;
\draw (-1,-3) -- (1,-3) -- (1,3) -- (-1,3) -- cycle;
\node[cross,thick,scale=0.5] at (-2.5,2) {};
\node[cross,thick,scale=0.5] at (2,-2) {};
\path (-3.2,2) node {$x^0_1$};
\path (2.5,-1.7) node {$x^1_1$};
\path (0,-2) node {$\omega$};
\path (-1.5,2.3) node {$v$};
\draw (-2.5,2) -- (-1,2);
\path (1.5,2.3) node {$v$};
\draw (2,2) -- (-1,2);
\draw  (2,0) arc (-90:90:1);
\draw  (2,0) arc (-90:90:-1);
\path (-1.5,2.3) node {$v$};
\draw (-1.5,2) -- (-1.6,2.1);
\draw (-1.5,2) -- (-1.6,1.9);
\draw (1.5,2) -- (1.4,2.1);
\draw (1.5,2) -- (1.4,1.9);
\draw[dotted,thick]  (2,-2.2) arc (90:0:-1.2);
\draw[dotted,thick] (-0.8,2) -- (0.8,-1);
\node[cross,thick,scale=0.5] at (2,-2.2) {};
\path (2.5,-2.5) node {$\widetilde{x}^1_1$};
\end{tikzpicture}\end{center}
%\begin{center}
%\begin{tikzpicture}[scale=1]
%\fill[pattern=dots,opacity = 0.5] (-3,-1) -- (-3,1) -- (3,1) -- (3,-1) -- cycle;
%\draw (-3,-1) -- (-3,1) -- (3,1) -- (3,-1) -- cycle;
%\node[cross,thick,minimum size=4pt] at (2,-2) {};
%\node[cross,thick,minimum size=4pt] at (-2,2) {};
%\path (1.7,-2) node {$x^0_1$};
%\path (-1.7,2) node {$x^1_1$};
%\path (-2,0) node {$\omega$};
%\path (2.3,-1.5) node {$v$};
%\draw (2,-2) -- (2,-1);
%\path (2.3,1.5) node {$v$};
%\draw (2,2) -- (2,-1);
%\draw  (0,2) arc (180:0:1);
%\draw  (0,2) arc (180:0:-1);
%\path (2.3,-1.5) node {$v$};
%\draw (2,-1.5) -- (2.1,-1.6);
%\draw (2,-1.5) -- (1.9,-1.6);
%\draw (2,1.5) -- (2.1,1.4);
%\draw (2,1.5) -- (1.9,1.4);
%\draw[dotted,thick]  (-2.2,2) arc (0:90:-1.2);
%\draw[dotted,thick] (2,-0.8) -- (-1,0.8);
%\node[cross,thick,minimum size=4pt] at (-2.2,2) {};
%\path (-2.5,2) node {$\widetilde{x}^1_1$};
%\end{tikzpicture}\end{center}
\caption{Example of difference between $t^1(x^1_1)$ and $\overline{t}^1(x^1_1)$. }\label{fig:ex CE CA}
\end{figure}
\subsection{Infimum time for microscopic models}
In this section, we state the two main results about the infimum time for microscopic models. We evaluate the minimal time both for exact and approximate controllability, highlighting the different role of $t^1(x)$ and $\overline{t}^1(x)$ in the two cases.

For simplicity, we denote by
 \begin{equation}\label{def:t^l_i}
 \overline{t}^0_i:=\overline{t}^0(x_i^0),\qquad \overline{t}^1_i:=\overline{t}^1(x_i^1),\qquad
 t^0_i:=t^0(x_i^0),\qquad t^1_i:=t^1(x_i^1),
 \end{equation}
 for $i\in\{1,...,n\}$. We then define 
$$
M^*_e(X^0,X^1):=\max\{t_i^0,t_i^1:i=1,...,n\}, \qquad
M^*_a(X^0,X^1):=\max\{t_i^0,\overline{t}_i^1:i=1,...,n\}.
$$
The value $M^*_e(X^0,X^1)$ has the following meaning for exact controllability: a time $T>M^*_e(X^0,X^1)$ is sufficiently large for each particle $x_i^0$ to enter $\omega$ and for each particle $x_i^1$ to enter it backward in time. The value $M^*_a(X^0,X^1)$ plays the same role for approximate controllability.

We now state our main result on the infimum time about exact control of System \eqref{eq ODE}.
\begin{theorem}[Infimum time for exact control of microscopic models]\label{th:discret exact}
Let $X^0:=\{x^0_1,...,x^0_n\}$ and $X^1:=\{x^1_1,...,x^1_n\}$ be two disjoint configurations (see Definition \ref{def:disjoint}).
Assume that $\omega$ is a bounded open connected set, the  empirical measures \eqref{eq:emp} associated to $X^0$ and $X^1$ 
 satisfy the Geometric Condition \ref{cond1} and 
  the velocities $v$ and $\mathds{1}_{\omega}u%:\mb{R}^d\times\mb{R}^+\rightarrow\mb{R}^d
 $ satisfy the Carath\'eodory Condition \ref{cond:cata}.
 
Arrange the sequences 
$\{t^0_i\}_i$ and $\{t^1_j\}_j$ to be non-decreasingly and non-increasingly ordered, respectively.
Then 
\begin{equation}\label{OT disc CE}
M_e(X^0,X^1):=\max_{i\in\{1,...,n\}}|t^0_i+t^1_i|
\end{equation}
is the infimum time $T_e(X^0,X^1)$ for  exact control of System \eqref{eq ODE} in the following sense:
\begin{itemize}
\item[(i)] For each $T>M_e(X^0,X^1)$, System \eqref{eq ODE} is exactly controllable  from $X^0$ to $X^1$ at time $T$, \textit{i.e.}
there exists a control  $\mathds{1}_{\omega}u:\mb{R}^d\times\mb{R}^+\rightarrow\mb{R}^d$ %satisfying the Carath\'eodory condition  and
 steering $X^0$ exactly to $X^1$.
%\textcolor{red}{J'ENLEVERAIS Moreover, at each time  $t\in[0,T]$, the configuration $X(t)$ is disjoint, \textit{i.e.} $x_i(t)\neq x_j(t)$ for all $i\neq j$.}
\item[(ii)] For each $T\in (M^*_e(X^0,X^1),M_e(X^0,X^1)]$,  System \eqref{eq ODE} is not exactly controllable  from $X^0$ to $X^1$.
\item[(iii)] There exists at most a finite number of times $T\in[0,M^*_e(X^0,X^1)]$ for 
which System \eqref{eq ODE} is exactly controllable  from $X^0$ to $X^1$.
\end{itemize}
\end{theorem}

We give a proof of Theorem \ref{th:discret exact} in Section \ref{sec:opt time finite dim}.
As stated in item (iii), it can exist some times  $T\leqslant M^*_e(X^0,X^1)$ 
at which it is possible to steer $X^0$ to $X^1$, but 
it will be not entirely thanks to the control. 
We give an example of this situation in Remark \ref{rmq:T2*} below.
We now give a characterization of the minimal time for approximate control of microscopic models.
\begin{theorem}[Infimum time for approximate control of microscopic models]\label{th:discret approx}
Let $X^0:=\{x^0_1,...,x^0_n\}$ and $X^1:=\{x^1_1,...,x^1_n\}$ be two  disjoint configurations (see Definition \ref{def:disjoint}).
Assume that $\omega$ is a bounded open connected set, the  empirical measures  \eqref{eq:emp} associated to $X^0$ and $X^1$ 
 satisfy the Geometric Condition \ref{cond1} and 
 the velocities $v$ and $\mathds{1}_{\omega}u%:\mb{R}^d\times\mb{R}^+\rightarrow\mb{R}^d
 $ satisfy the Carath\'eodory Condition \ref{cond:cata}.
%Assume that $\omega$ is convex and that the  empirical measures associated to $X^0$ and $X^1$ (see \eqref{eq:emp})
% satisfy the Geometric Condition \ref{cond1}.
Arrange the sequences 
$\{t^0_i\}_i$ and $\{\overline{t}^1_j\}_j$ to
 be non-decreasingly and non-increasingly ordered, respectively.
Then  \begin{equation*}M_a(X^0,X^1):=\max_{i\in\{1,...,n\}}|t^0_i+\overline{t}^1_i|
\end{equation*}
is the infimum time $T_a(X^0,X^1)$ for approximate  control of System \eqref{eq ODE} in the following sense:
\begin{itemize}
\item[(i)] For each $T>M_a(X^0,X^1)$, System \eqref{eq ODE} is approximately controllable  from $X^0$ to $X^1$ at time $T$, \textit{i.e.}
there exists a permutation 
 $\sigma\in \mathbb{S}_n$ and a sequence of controls $\{\mathds{1}_{\omega}u_k\}_{k\in\mb{N}^*}$ %satisfying the Carath\'eodory condition 
 such that the associated solution $x_{i,k}(t)$ to System \eqref{eq ODE} satisfies 
 $x_{i,k}(T)\xrightarrow[k\to\infty]{} x^1_{\sigma(i)}.$
 
% \textcolor{red}{Ici en effet on pourrait enlever la permutation, car le fait d'avoir ordonn\'e\ les index nous donne que la strat\'egie optimale envoie $x^0_i$ dans $x^1_i$.}
 
 \item[(ii)] For each $T\in (M^*_a(X^0,X^1),M_a(X^0,X^1)]$,  System \eqref{eq ODE} is not approximately controllable  from $X^0$ to $X^1$.
\item[(iii)] There exists at most a finite number of times $T\in[0,M^*_a(X^0,X^1)]$ for which System 
\eqref{eq ODE} is approximately controllable  from $X^0$ to $X^1$.
\end{itemize}

\end{theorem}
We give a proof of Theorem \ref{th:discret approx} in Section \ref{sec:opt time finite dim}.
As for  the exact controllability, one might have approximate controllability for a time $T\leqslant M_a^*(X^0,X^1)$, but not entirely  due to the control. We refer again to Remark \ref{rmq:T2*} below for an example.

It is well know that the notions of approximate and exact controllability 
are equivalent for finite dimensional linear systems when the control acts linearly, see \textit{e.g.} \cite{C09}. This is not the case for System \eqref{eq ODE}, which highlights the fact that we are dealing with
a non-linear control problem. The difference is indeed related to the fact that for exact and approximate controllability, tangent trajectories  give different behaviors.
 For example, in Figure \ref{fig:ex CE CA}, if we denote by $X^0:=\{x^0_1\}$ and $X^1:=\{x^1_1\}$, 
 then it holds   $M_a(X^0,X^1)< M_e(X^0,X^1)$ due to the presence of a tangent trajectory.
A trajectory achieving approximate controllability is represented as dashed lines in the case $T\in (  M_a(X^0,X^1), M_e(X^0,X^1))$
in Figure \ref{fig:ex CE CA}.%

\subsection{Infimum time for macroscopic models}

In this section, we state the main result about the infimum time for approximate control of macroscopic models. 

Let  $\mu^0$ and $\mu^1$ be two probability measures, with compact support.
We introduce the maps $\mc{F}_{\mu^0}$ and $\mc{B}_{\mu^1}$ defined for all $t\geqslant 0$ by
\begin{equation*}\left\{\begin{array}{l}
\mc{F}_{\mu^0}(t):=\mu^0(\{x\in \supp(\mu^0):t^0(x)\leqslant t\}),\\\noalign{\smallskip}
\mc{B}_{\mu^1}(t):=\mu^1(\{x\in \supp(\mu^1):t^1(x)\leqslant t\}).
\end{array}\right.
\end{equation*} 
The function $\mc{F}_{\mu^0}$ (resp. $\mc{B}_{\mu^1}$) gives the quantity of mass coming from $\mu^0$  forward in time
(resp.  the quantity of mass coming from $\mu^1$ backward in time) which has entered in $\omega$ at time $t$. Observe that we do not decrease $\mc{F}_{\mu^0}$ when the mass eventually leaves $\omega$, and similarly for $\mc{B}_{\mu^1}$. Define the generalised inverse functions $\mc{F}_{\mu^0}^{-1}$ and $\mc{B}_{\mu^1}^{-1}$ of  $\mc{F}_{\mu^0}$ and $\mc{B}_{\mu^1}$ given for all $m\in[0,1]$ by
\begin{equation*}\left\{\begin{array}{l}
\mc{F}_{\mu^0}^{-1}(m):=\inf\{t\geq0 : \mc{F}_{\mu^0}(t)\geqslant m\},\\\noalign{\smallskip}
\mc{B}_{\mu^1}^{-1}(m):=\inf\{t\geq0 : \mc{B}_{\mu^1}(t)\geqslant m\}.
\end{array}\right.
\end{equation*}
The function $\mc{F}_{\mu^0}^{-1}$
is increasing, lower semi-continuous and gives the time at which a mass $m$ has entered in $\omega$; similarly for $\mc{B}_{\mu^1}^{-1}$.
Define 
\begin{equation}
S^*(\mu^0,\mu^1):=\sup\{t^l(x):x\in\supp(\mu^l)\mbox{ and }l\in\{0,1\}\}.\label{e:Sstar}
\end{equation}
We then have the following main result about infimum time in the macroscopic case:
\begin{theorem}[Infimum time for approximate control of macroscopic models]\label{th opt}
Let  $\mu^0$ and $\mu^1$ be two probability measures, with compact support, absolutely continuous with respect to the Lebesgue measure and satisfying the Geometric Condition \ref{cond1}. Assume that  $\omega$ is a bounded open connected set and 
  the velocities $v$ and $\mathds{1}_{\omega}u%:\mb{R}^d\times\mb{R}^+\rightarrow\mb{R}^d
 $ satisfy the Carath\'eodory Condition \ref{cond:cata}.
Then 
\begin{equation}\label{def T0}
S(\mu^0,\mu^1):=\sup_{m\in[0,1]} \{\mc{F}_{\mu^0}^{-1}(m)+\mc{B}_{\mu^1}^{-1}(1-m)\}
\end{equation}
is the infimum time $T_a(\mu^0,\mu^1)$ to approximately steer $\mu^0$ to $\mu^1$  in the following sense:
\begin{enumerate}
\item[(i)] Assume that there exists $\alpha>0$ for which it holds
\begin{equation}\label{eq:cond omega}
\{x\in\mathbb{R}^d:d(x,\omega^c)>\varepsilon\} \mbox{ is connected for each }\varepsilon\in(0,\alpha),
\end{equation}
(e.g. when $\omega$ is convex.) Then for all $T> S(\mu^0,\mu^1)$,  System \eqref{eq:transport}
 is approximately controllable 
from $\mu^0$ to $\mu^1$ at time $T$.
%with a control $\mathds{1}_{\omega}u:\mb{R}^d\times\mb{R}^+\rightarrow\mb{R}^d$ satisfying the Carath\'eodory condition.
\item[(ii)] Assume that the boundary of the evoluted set 
\begin{equation}
\omega^t:=\cup_{\tau\in(0,t)}\Phi^v_\tau(\omega) \label{e-evol1}
\end{equation}
with respect to the flow $\Phi^v_\tau$ of the vector field  $v$ has zero Lebesgue measure for each $t>0$. Then, for all $T\in(S^*(\mu^0,\mu^1),S(\mu^0,\mu^1)]$, System \eqref{eq:transport} is not approximately controllable
from $\mu^0$ to $\mu^1$.
\end{enumerate}
\end{theorem}

We give a proof  of Theorem \ref{th opt} in  Section \ref{sec:optimal time cont}.
We observe that the quantity $S(\mu^0,\mu^1)$ given in \eqref{def T0} is the continuous equivalent of $M_e(X^0,X^1)$ in \eqref{OT disc CE}.
Contrarily to the microscopic case, System \eqref{eq:transport} 
can be approximately controllable
at each time $T\in(0,S^*(\mu^0,\mu^1))$.
We give some examples in Remark \ref{rmq:T2* cont} below.
We do not analyze the infimum time to exactly control System \eqref{eq:transport}
since it is not exactly controllable with controls satisfying the Carath\'eodory condition. 
For more details, we refer to Remark \ref{rmq:exact} below.

\begin{remark}
 The particular condition on $\omega^t$ in statement (ii) in Theorem \ref{th opt} is crucial to ensure that AC-measures do not concentrate on the boundary. In connection with this condition, we now state the following useful Lemma, which proof is postponed to the Appendix.

It gives simple conditions, though not optimal, ensuring that the assumptions of Theorem~\ref{th opt} (\textit{ii}) hold.
\end{remark}
\begin{lemma}\label{l-omegatreg}
Let $v$ be a $C^1$ vector field and $\omega$ an open bounded set satisfying the  uniform interior cone condition (e.g. convex) as recalled in Definition~\ref{d-intcone}. Then, for each $t>0$ the evoluted set $\omega^t$ defined in \eqref{e-evol1} has boundary with zero Lebesgue measure.
\end{lemma}

In a future work, we aim to investigate if such statement holds with less regularity of the vector field $v$ and of the set $\omega$. For example, we will show that the fact that $\omega$ has a boundary of zero measure does not imply that $\omega^t$ satisfies the same property.

This paper is organised as follows. 
In Section \ref{section 2}, we recall basic properties of the Wasserstein distance, ordinary differential equations and continuity equations.
We prove our main results Theorems \ref{th:discret exact} and \ref{th:discret approx}  in Section \ref{sec:opt time finite dim} 
and Theorem \ref{th opt} in Section \ref{sec:optimal time cont}.
We finally introduce an algorithm to compute the infimum time for approximate control of macroscopic models and give some numerical examples in Section \ref{sec:num sim}.

\section{Wasserstein distance, models and controllability}\label{section 2}

In this section, we highlight the connections of the microscopic model \eqref{eq ODE} and the macroscopic model   \eqref{eq:transport}
with the Wasserstein distance, that is the natural distance associated to these dynamics. 
We also recall our previous results  obtained in \cite{DMR17} about controllability of these systems.
\subsection{Wasserstein distance}%
From now on, we denote by $\mc{P}_c(\mb{R}^d)$ the space of probability measures on $\mb{R}^d$ with compact support. We also denote by ``AC measures''
the measures which are absolutely continuous with respect to the Lebesgue measure
and by $\mc{P}_c^{ac}(\mb{R}^d)$  the subset  of $\mc{P}_c(\mb{R}^d)$ of AC measures. 
\begin{definition}

For $\mu,~\nu\in\mc{P}_c(\mb{R}^d)$, we denote by $\Pi(\mu,\nu)$ the set of \textit{transference plans}
 from $\mu$ to $\nu$, \textit{i.e.} the probability measures on $\mb{R}^d\times\mb{R}^d$ 
 with first marginal $\mu$ and second marginal $\nu$.
Let $p\in[1,\infty)$ and $\mu,\nu\in \mc{P}_c(\mb{R}^d)$. Define 
\begin{eqnarray}
\label{def:Wp}
W_p(\mu,\nu):=\inf\limits_{\pi\in\Pi(\mu,\nu)}
\left\{\left(\int_{\mb{R}^d\times\mb{R}^d}
\|x-y\|^pd\pi(x,y)\right)^{1/p}\right\},\\
\label{def:Winf}
W_{\infty}(\mu,\nu):=\inf\{\pi-\mbox{esssup}~\|x-y\|:\pi\in\Pi(\mu,\nu)\}.
\end{eqnarray}
For $p\in[1,\infty]$, this quantity  is called the \textbf{Wasserstein distance} or $p$-Wasser\-stein distance.
\end{definition}

The idea behind this definition is the problem of \textit{optimal transportation},
consisting  in finding the optimal way to transport
mass from a given measure to another. 
For a thorough introduction, see \textit{e.g.} \cite{V03}.

Between two microscopic configurations, we will use the following distance.

\begin{definition}
Let $p\in [1,\infty]$. Consider  two configurations $X^0:=\{x^0_1,...,x^0_n\}$ and $X^1:=\{x^1_1,...,x^1_n\}$ of $\mb{R}^d$.
Define the Wasserstein distance between $X^0$ and $X^1$ as follows:
$$W_p(X^0,X^1):=W_p(\mu^0,\mu^1),$$
where $\mu^0:=\frac{1}{n}\sum_i\delta_{x^0_i}$ and $\mu^1:=\frac{1}{n}\sum_i\delta_{x^1_i}$.
\end{definition}
In this case, the Wasserstein distance can be rewritten as follows:
\begin{proposition}
Let $p\in [1,\infty)$. Consider two configurations  $X^0:=\{x^0_1,...,x^0_n\}$ and $X^1:=\{x^1_1,...,x^1_n\}$ of $\mb{R}^d$.
It holds
\begin{equation*}
\left\{\begin{array}{l}
W_p(X^0,X^1)=\inf_{\sigma\in \mathbb{S}_n}
\left(\sum_{i=1}^n\frac{1}{n}\|x^0_i-x^1_{\sigma(i)}\|^p\right)^{1/p},\\\noalign{\smallskip}
W_{\infty}(X^0,X^1)=\inf_{\sigma\in \mathbb{S}_n}
\max_{i\in\{1,...,n\}}\|x^0_i-x^1_{\sigma(i)}\|,
\end{array}\right.
\end{equation*}
where $\mathbb{S}_n$ is the set of permutations on $\{1,...,n\}$.
\end{proposition}
For a proof, we refer to \cite[p. 5]{V03}. The idea behind this result is exactly the one of optimal transportation in the discrete setting: one has $n$ initial and final configurations to be paired, with a given cost. The cost minimizer is the minimizer among all the permutations.

The Wasserstein distance satisfies some useful properties.\begin{proposition}[{see \cite[Chap. 7]{V03}} and \cite{CdPJ08}]\label{prop Wp}
It holds:
\begin{enumerate}
\item For each pair $\mu,\nu\in \mc{P}_c(\mb{R}^d)$, the infimum in \eqref{def:Wp} or \eqref{def:Winf} is achieved.
\item
For $p\in[1,\infty]$, the quantity $W_p$ is a  distance on $\mc{P}_c(\mb{R}^d)$.
Moreover, for $p\in[1,\infty)$, the topology induced by the Wasserstein distance $W_p$ on $\mc{P}_c(\mb{R}^d)$ coincides with the weak topology 
of measures, \textit{i.e}, for all sequence $\{\mu_k\}_{k\in\mb{N}^*}\subset\mc{P}_c(\mb{R}^d)$ and all $\mu\in\mc{P}_c(\mb{R}^d)$,
the following statements are equivalent:
\begin{enumerate}
\item[(i)]$W_p(\mu_k,\mu)\underset{k\rightarrow\infty}{\longrightarrow}0$.
\item[(ii)] $\mu_k\underset{k\rightarrow\infty}{\longrightarrow}\mu$ in the weak sense.
\end{enumerate}
\end{enumerate}
\end{proposition}

The Wasserstein distance can be extended to all pairs of measures $\mu,\nu$ 
compactly supported with the same total mass $\mu(\mb{R}^d)=\nu(\mb{R}^d)\neq0$ by the formula
\begin{equation*}W_p(\mu,\nu)=|\mu|^{1/p} W_p\left(\frac{\mu}{|\mu|},\frac{\nu}{|\nu|}\right).\end{equation*}
In the rest of the paper, the following properties of the Wasserstein distance will be helpful.
\begin{proposition}[see \cite{V03}]%\label{p-diam} 
Let $p\in[1,\infty]$ and $\mu,\nu$ be two positive measures satisfying $\mu(\mb{R}^d)=\nu(\mb{R}^d)$ supported in a subset $X$. It then holds
$$W_{p}(\mu,\nu) \leqslant \mu(\mb{R}^d) \mr{diam}(X).$$
\end{proposition}
%\begin{proof} This is a direct consequence of the definition.\end{proof}
\begin{proposition}[see \cite{PR13,V03}]%\label{prop:ine wass}
Let $\mu,~\rho,~\nu,~\eta$ be four positive measures compactly supported satisfying $\mu(\mb{R}^d)=\nu(\mb{R}^d)$ and $\rho(\mb{R}^d)=\eta(\mb{R}^d)$.
\begin{enumerate}
\item[(i)]
 For each  $p\in[1,\infty]$, it holds
\begin{equation*}W^p_p(\mu+\rho,\nu+\eta)
\leqslant W^p_p(\mu,\nu)+W^p_p(\rho,\eta).
\end{equation*}
\item[(ii)]
For each   $p_1,~p_2\in[1,\infty]$ with $p_1\leqslant p_2$, it holds
\begin{equation}\label{ine wasser 4}
\left\{\begin{array}{l}
W_{p_1}(\mu,\nu)\leqslant W_{p_2}(\mu,\nu),\\\noalign{\smallskip}
W_{p_2}(\mu,\nu)
\leqslant \mr{diam}(X)^{1-p_1/p_2}W_{p_1}^{p_1/p_2}(\mu,\nu),
\end{array}\right.
\end{equation}
where $X$ contains the supports of $\mu$ and $\nu$.
\end{enumerate}
\end{proposition}

\subsection{Well-posedness of System \eqref{eq:transport}}
In this section, we study the macroscopic model \eqref{eq:transport}, together with its connections with the Wasserstein distance. Consider  the following system
\begin{equation}\label{eq:transport sec 2}
	\left\{
	\begin{array}{ll}
\partial_t\mu +\nabla\cdot(w\mu)=0&\mbox{ in }\mb{R}^d\times\mb{R}^+,\\\noalign{\smallskip}
\mu(\cdot,0)=\mu^0&\mbox{ in }\mb{R}^d,
	\end{array}
	\right.
\end{equation}
where $w:\mb{R}^d\times\mb{R}^+\rightarrow\mb{R}^d$ is a time-dependent vector field.
This equation is called the \textbf{continuity equation}.
We now introduce the flow associated to System \eqref{eq:transport sec 2}.
\begin{definition}\label{def:flow}
We define the \textbf{flow} associated to a vector field $w:\mb{R}^d\times\mb{R}^+\rightarrow\mb{R}^d$
satisfying the Carath\'eodory condition
as the application $(x^0,t)\mapsto\Phi_t^w(x^0)$ such that, for all $x^0\in\mb{R}^d$, 
$t\mapsto\Phi_t^w(x^0)$ is the unique solution to 
\begin{equation}\label{eq charac}
\left\{\begin{array}{l}
\dot x(t) =w(x(t),t)\mbox{ for a.e. }t\geqslant 0,\\\noalign{\smallskip}
x(0)=x^0.
\end{array}\right.
\end{equation}
\end{definition}
It is classical that, for a vector field $w$ satisfying the Carath\'eodory condition, System \eqref{eq charac} is well-posed. See for instance
\cite{BP07}.

We denote by $\Gamma$  the set of the Borel maps $\gamma:\mb{R}^d\rightarrow\mb{R}^d$. We  recall the definition of the \textit{push-forward} of a measure.
\begin{definition}
For a $\gamma\in\Gamma$, 
we define the {\it push-forward} $\gamma\#\mu$ of a measure $\mu$ of $\mb{R}^d$ as follows:
\begin{equation*}
(\gamma\#\mu)(E):=\mu(\gamma^{-1}(E)),
\end{equation*}
for every subset $E$ such that $\gamma^{-1}(E)$ is $\mu$-measurable.
\end{definition}
We now recall a standard result linking 
  \eqref{eq:transport sec 2} and \eqref{eq charac}, known as the method of characteristics.
 \begin{theorem}[see {\cite[Th. 5.34]{V03}}]
Let $T>0$,  $\mu^0\in \mc{P}_c(\mb{R}^d)$ and  $w$ a vector field satisfying the Carath\'eodory condition.
Then, System \eqref{eq:transport sec 2}
admits a unique solution $\mu$ in $\mc{C}^0([0,T];\mc{P}_c(\mb{R}^d))$, 
where $\mc{P}_c(\mb{R}^d)$ is equipped with the weak topology. Moreover:
\begin{enumerate}
\item[(i)] It holds $\mu(\cdot,t)=\Phi_t^{w}\#\mu^0$, where $\Phi_t^{w}$ is the flow of $w$ as in Definition \ref{def:flow};
\item[(ii)] If  $\mu^0\in \mc{P}_c^{ac}(\mb{R}^d)$, then $\mu(\cdot,t)\in\mc{P}_c^{ac}(\mb{R}^d)$.

\end{enumerate}
 \end{theorem}

We also recall the following proposition:
\begin{proposition}[see {\cite[Prop. 4]{PR16}}]Let  $\mu,~\nu\in\mc{P}_c(\mb{R}^d)$
and $w:\mb{R}^d\times\mb{R}\rightarrow\mb{R}^d$ be a vector field satisfying the Carath\'eodory condition, with a Lipschitz constant equal to $L$.
 For each $t\in\mb{R}$
and  $p\in[1,\infty)$, it holds
\begin{equation*}%\label{ine wasser 2}
W_p(\Phi_t^w\#\mu,\Phi_t^w\#\nu)
\leqslant e^{L|t|} W_p(\mu,\nu).
\end{equation*}
\end{proposition}

%\begin{proof}
%Consider $\pi$ a minimiser  of the Wasserstein distance \eqref{def:Wp} between $\mu$ and $\nu$.
%Then $(\Phi^v_t,\Phi^v_t)\#\pi\in \Pi(\Phi^v_t\#\mu,\Phi^v_t\#\nu)$ and it holds by Gronwall's Lemma
%\begin{equation*}
%\begin{array}{rcl}
%W^p_p(\Phi^v_t\#\mu,\Phi^v_t\#\nu)
%&\leqslant&\displaystyle\int_{\mb{R}^d\times\mb{R}^d}\|\Phi^v_t(x)-\Phi^v_t(y)\|^pd\pi(x,y)\\
%&\leqslant& e^{pL|t|}\displaystyle\int_{\mb{R}^d\times\mb{R}^d}\|x-y\|^pd\pi(x,y).
%\end{array}\end{equation*}
%%\hfill\qed
%\end{proof}

\subsection{Approximate and exact controllability of System \eqref{eq ODE} and System \eqref{eq:transport}}\label{section 2.3}

In this section, we recall the main results of \cite{DMR17} about approximate and exact controllability of microscopic and macroscopic models. We first recall the precise notions of approximate and exact controllability.
\begin{definition} %\label{d-cara}
We say that
\begin{enumerate}
\item[$\bullet$] 
The microscopic system \eqref{eq ODE} (resp. the macroscopic system \eqref{eq:transport}) is 
{\bf approximately controllable}
from $X^0$ to $X^1$ (resp. from $\mu^0$ to $\mu^1$) at time $T$ if 
for each $\varepsilon>0$ there exists a sequence of controls $\{\mathds{1}_{\omega}u_k\}_{k\in\mb{N}^*}$  satisfying the Carath\'eodory condition 
such that, denoting by $X_k$ (resp. $\mu_k$) the corresponding 
solution  to System \eqref{eq ODE} (resp. System \eqref{eq:transport}),
$X_k(T)$ converges to $X^1$ (resp. $\mu_k(T)$ converges weakly to $\mu^1$).
\item[$\bullet$] 
The microscopic system \eqref{eq ODE} (resp. the macroscopic system \eqref{eq:transport}) is 
{\bf exactly controllable}
from $X^0$ to $X^1$ (resp. from $\mu^0$ to $\mu^1$) at time $T$ if 
there exists a control $\mathds{1}_{\omega}u$  satisfying the Carath\'eodory condition such that
the corresponding solution  $X$ (resp. $\mu$)  to System \eqref{eq ODE} (resp. System \eqref{eq:transport}) satisfies
\begin{equation*}X(T)=X^1\mbox{ (resp. }\mu(T)=\mu^1\mbox{)}.
\end{equation*}
\end{enumerate}
\end{definition}

For approximate controllability of System \eqref{eq:transport}, the following result holds. 
\begin{theorem}[see \cite{DMR17}]Let  $\mu^0,\mu^1\in\mc{P}_c^{ac}(\mb{R}^d)$ 
satisfying the Geometric Condition \ref{cond1}. 
Then there exists $T$ such that System \eqref{eq:transport} is 
\textbf{approximately controllable} from $\mu^0$ to $\mu^1$ at time $T$
with a control $\mathds{1}_{\omega}u:\mb{R}^d\times\mb{R}^+\rightarrow\mb{R}^d$ satisfying the Carath\'eodory condition.
\end{theorem}

\begin{remark}\label{rmq:exact} For exact controllability, the picture is completely different. The Carath\'eodo\-ry condition implies that the flow $\Phi^{v+\mathds{1}_{\omega}u}_t$ is an homeomorphism. Since general $\mu^0,\mu^1$ are not homeomorphic,  one needs to search a control with less regularity. The drawback is that one loses the uniqueness of the solution to System \eqref{eq:transport}.  This is the meaning of our recent result Theorem  \ref{th cont exact} below, proving that a class of controls ensuring exact controllability exists, but uniqueness is lost. 
\end{remark}

\begin{theorem}[see \cite{DMR17}]\label{th cont exact}
Let $\mu^0,\mu^1\in\mc{P}_c^{ac}(\mb{R}^d)$ 
satisfying the Geometric Condition \ref{cond1}. 
Then, there exists $T>0$ such that System \eqref{eq:transport} is \textbf{exactly controllable} from $\mu^0$ to $\mu^1$ at time $T$
in the following sense: there exists a couple $(\mathds{1}_{\omega}u,\mu)$ composed of  
a Borel vector field $\mathds{1}_{\omega}u:\mb{R}^d\times\mb{R}^+\rightarrow\mb{R}^d$ 
and  a time-evolving measure $\mu$
being  weak  solution to  System \eqref{eq:transport}  and satisfying  
 $\mu(T)=\mu^1$.
\end{theorem}

Using Proposition \ref{prop Wp}, the approximate controllability of System \eqref{eq:transport}
can be rewritten in terms of the Wasserstein distance:
\begin{proposition}
The macroscopic system \eqref{eq:transport} is  approximately controllable
 fr\-om $\mu^0$ to $\mu^1$ at time $T$ if 
for each $\varepsilon>0$ there exists a control $\mathds{1}_{\omega}u$  satisfying the Carath\'eodory condition such that the corresponding 
solution $\mu$ to System \eqref{eq:transport} satisfies
\begin{equation}\label{estim Wp approx bis}
W_p(\mu^1,\mu(T))\leqslant \varepsilon.
\end{equation}
\end{proposition}
\begin{remark}
Properties \eqref{ine wasser 4} imply that
all the Wasserstein distances $W_p$ are equi\-valent for 
measures compactly supported and $p\in[1,\infty)$, see \cite{V03}.
Thus, we can replace \eqref{estim Wp approx bis}  by 
$$W_1(\mu^1,\mu(T))\leqslant \varepsilon.$$
\end{remark}
Thus, in this work, we study  approximate controllability by considering the distance $W_1$ only.
We will use the distances $W_2$ and $W_{\infty}$, in some other specific cases only.
\section{Infimum time for microscopic models}\label{sec:opt time finite dim}
In this section, we prove Theorem \ref{th:discret exact} and \ref{th:discret approx}, \textit{i.e.} the infimum time for the approximate and exact controllability of microscopic models. 
We first obtain the following result:

\begin{proposition}\label{prop: dim finie}
Let $X^0:=\{x^0_1,...,x^0_n\}\subset\mb{R}^d$ and $X^1:=\{x^1_1,...,x^1_n\}\subset\mb{R}^d$ be two disjoint configurations 
(see Definition \ref{def:disjoint}).
Assume that $\omega$ is a bounded open connected set, the  empirical measures associated to $X^0$ and $X^1$ (see \eqref{eq:emp})
 satisfy the Geometric Condition \ref{cond1} and 
  the velocities $v$ and $\mathds{1}_{\omega}u%:\mb{R}^d\times\mb{R}^+\rightarrow\mb{R}^d
 $ satisfy the Carath\'eodory Condition \ref{cond:cata}.
Consider the sequences $\{t_i^0\}_{i}$ and $\{t_i^1\}_{i}$ given in \eqref{def:t^l_i}. 
Then 
\begin{equation}\label{minimal time}
\widetilde{M}_e(X^0,X^1):=\min_{\sigma\in \mathbb{S}_n}\max_{i\in\{1,...,n\}}|t^0_i+t_{\sigma (i)}^1|
\end{equation}
is the infimum time $T_e(X^0,X^1)$ to  exactly control System \eqref{eq ODE} in the sense of Theorem \ref{th:discret exact}.
\end{proposition}

\begin{proof} We first prove the result corresponding to {\bf Item (i)} of Theorem \ref{th:discret exact}. Let $T:=\widetilde{M}_e(X^0,X^1)+\delta$ with $\delta>0$. Using the Geometric Condition \ref{cond1}, for all $i\in\{1,...,n\}$, there exist $s_i^0\in (t_i^0,t_i^0+\delta/3)$ and $s_i^1\in (t_i^1,t_i^1+\delta/3)$
such that 
\begin{equation*}y_i^0:=\Phi_{s_i^0}^v(x_i^0)\in\omega\mbox{ and }y_i^1:=\Phi_{-s_i^1}^v(x_i^1)\in\omega.
\end{equation*}
This part of the proof is divided into two steps:
\begin{itemize}
\item[$\bullet$] In Step 1, we build a permutation $\sigma$ and a flow on $\omega$ sending $y_i^0$ 
to  $y_{\sigma(i)}^1$ for all $i\in\{1,...,n\}$ with no intersection of trajectories.
\item[$\bullet$] In Step 2, we define a corresponding control sending $x_i^0$
to $x_{\sigma(i)}^1$ for all $i\in\{1,...,n\}$.
\end{itemize}

\textbf{Item (i), Step 1:}  We first assume that $\omega$ is convex.
The goal is to build  a flow with no intersection of the characteristic. 
For all $i,j\in\{1,...,n\}$, we define the cost\begin{equation}\label{Kij}
K_{ij}(y_i^0,s_i^0,y_j^1,s_j^1):=\left\{\begin{array}{ll}
\|(y_i^0,s_i^0)-(y_j^1,T-s_j^1)\|_{\mb{R}^{d+1}}&\mbox{ if }s_i^0<T-s_j^1,\\
\infty& \mbox{ otherwise.}
\end{array}\right.
\end{equation}
Consider the minimization problem:
\begin{equation}\label{eq:inf}
\inf\limits_{\pi\in\mb{B}_n}\frac{1}{n}\sum_{i,j=1}^n K_{ij}(y_i^0,s_i^0,y_j^1,s_j^1)\pi_{ij},
\end{equation}
where $\mb{B}_n$ is the set of the bistochastic $n\times n$ matrices, \textit{i.e.} 
the matrices $\pi:=(\pi_{ij})_{1\leqslant i,j\leqslant n}$ satisfying, for all $i,j\in\{1,...,n\}$,
$$\sum_{i=1}^n\pi_{ij}=1,~\sum_{j=1}^n\pi_{ij}=1,~\pi_{ij}\geqslant 0.$$
The infimum in \eqref{eq:inf} is finite since $T>\widetilde{M}_e(X^0,X^1)$.
The problem \eqref{eq:inf} is a linear minimization problem on the closed convex set $\mb{B}_n$.
Hence, as a consequence of Krein-Milman's Theorem (see \cite{KM40}), the functional \eqref{eq:inf} admits a minimum
at an extremal point of $\mathbb{B}_n$,
\textit{i.e.} a permutation matrix.

Let $\sigma$ be a permutation, for which the associated matrix minimizes \eqref{eq:inf}. Consider the straight trajectories  $y_i(t)$ steering $y_i^0$ at time $s_i^0$ to $y_{\sigma(i)}^1$ at time $T-s_{\sigma(i)}^1$, that are explicitly defined by
\begin{equation}
y_i(t):=
\frac{T-s_{\sigma(i)}^1-t}{T-s_{\sigma(i)}^1-s_i^0}y_i^0+\frac{t-s_{i}^0}{T-s_{\sigma(i)}^1-s_i^0}y_{\sigma(i)}^1.
\label{e:straight}
\end{equation}
We now prove by contradiction that these trajectories have no intersection:
Assume that there exist $i$ and $j$ such that the associated trajectories $y_i(t)$ and $y_j(t)$ intersect. 
If we associate $y^0_i$ and $y^0_j$ to $y^0_{\sigma(j)}$ and 
$y^0_{\sigma(i)}$ respectively, \textit{i.e.} we consider the permutation 
 $\sigma\circ\mc{T}_{i,j}$,  where $\mc{T}_{i,j}$ is the transposition between the  $i$-th 
and the $j$-th elements,
then the associated  cost \eqref{eq:inf} is strictly smaller than 
the cost associated to $\sigma$. 
Indeed, using some geometric considerations in the space/time set (see Figure \ref{fig:geo}), we obtain
\begin{equation*}
\left\{\begin{array}{ll}
\|(y_i^0,s_i^0)-(y_{\sigma\circ\mathcal{T}_{ij}(i)}^1,T-s_{\sigma\circ\mathcal{T}_{ij}(i)}^1)\|_{\mathbb{R}^{d+1}}
=\|(y_i^0,s_i^0)-(y_{\sigma(j)}^1,T-s_{\sigma(j)}^1)\|_{\mathbb{R}^{d+1}}\\\noalign{\smallskip}
\hspace*{58mm}
<\|(y_i^0,s_i^0)-(y_{\sigma(i)}^1,T-s_{\sigma(i)}^1)\|_{\mathbb{R}^{d+1}},\\\noalign{\smallskip}
\|(y_j^0,s_j^0)-(y_{\sigma\circ\mathcal{T}_{ij}(j)}^1,T-s_{\sigma\circ\mathcal{T}_{ij}(j)}^1)\|_{\mathbb{R}^{d+1}}
=\|(y_j^0,s_j^0)-(y_{\sigma(i)}^1,T-s_{\sigma(i)}^1)\|_{\mathbb{R}^{d+1}}\\\noalign{\smallskip}
\hspace*{58mm}<\|(y_j^0,s_j^0)-(y_{\sigma(j)}^1,T-s_{\sigma(j)}^1)\|_{\mathbb{R}^{d+1}}.
\end{array}\right.
\end{equation*}
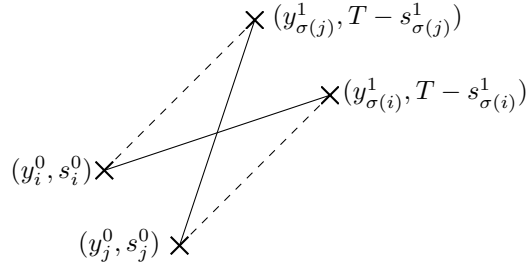
\begin{figure}[ht]
\begin{center}
\begin{tikzpicture}[scale=1]
\node[cross,thick,minimum size=4pt] at (1,0) {};
\node[cross,thick,minimum size=4pt] at (0,1) {};
\node[cross,thick,minimum size=4pt] at (3,2) {};
\node[cross,thick,minimum size=4pt] at (2,3) {};
\draw[-] (0,1) -- (3,2);
\draw[-] (1,0) -- (2,3);
\draw[dashed] (0,1) -- (2,3);
\draw[dashed] (1,0) -- (3,2);
\path (-0.7,1) node {$(y_i^0,s_i^0)$};
\path (0.2,0) node {$(y_j^0,s_j^0)$};
\path (3.5,3) node {$(y_{\sigma(j)}^1,T-s_{\sigma(j)}^1)$};
\path (4.4,2) node {$(y_{\sigma(i)}^1,T-s_{\sigma(i)}^1)$};
\end{tikzpicture}
\caption{An optimal permutation.}\label{fig:geo}
\end{center}\end{figure}
Here, we denote with $\|\cdot\|_{\mathbb{R}^{d+1}}$ the Euclidean norm in $\mathbb{R}^{d+1}$. This is in contradiction with the fact that $\sigma$ minimizes \eqref{eq:inf}.

Assume now that $\omega$ is not convex, but just open connected, hence arc-connected. Then, one replaces the norm $\|(y_i^0,s_i^0)-(y_j^1,T-s_j^1)\|_{\mb{R}^{d+1}}$ in \eqref{Kij} with the distance $d_{\overline{\omega}}(y_i^0,y_j^1)+|T-s_j^1-s_i^0|$, where 
$$d_{\overline{\omega}}(y_i^0,y_j^1)=\inf \limits_{\gamma\in C^1([0,1];\mathbb{R}^d)}\left\{\int_0^1|\dot\gamma(t)|\,dt\ \mbox{~~s.t.~~} \gamma(0)=y_i^0, \ \ \gamma(1)=y_j^1,\ \ \gamma(t)\in\overline{\omega} \, \ \forall \ t\in[0,1]\right\}.$$
%is the distance, where paths are restricted to $\omega$. 
For each pair, infimum is finite and it is realized by some curve $\gamma_{ij}$, since $\overline{\omega}$ is compact and  arc-connected. For the global problem \eqref{Kij} , there exist a minimizing permutation $\sigma$  steering $y_i^0$ at time $s_i^0$ to $y_{\sigma(i)}^1$ at time $T-s_{\sigma(i)}^1$. The corresponding curves $\gamma_{i\sigma(i)}$ are  not intersecting, for the same reasoning explained above for convex sets. Then, since $\omega$ is open, one can deform them to find paths in $\omega$ that are arbitrarily close to the $\gamma_{i\sigma(i)}$, still not intersecting and $C^\infty$ functions of time.  We denote by $y_i(t)$ such deformed path steering $y_i^0$ at time $s_i^0$ to $y_{\sigma(i)}^1$ at time $T-s_{\sigma(i)}^1$.

\textbf{Item (i), Step 2: } Consider now the following trajectories $z_i$:
\begin{equation*}
z_i(t):=\left\{\begin{array}{ll}
\Phi_t^v(x_i^0)&\mbox{ for all }t\in(0,s_i^0),\\
y_i(t)&\mbox{ for all }t\in(s_i^0,T-s_{\sigma(i)}^1),\\
\Phi_{t-T}^v(x_i^1)&\mbox{ for all }t\in(T-s_{\sigma(i)}^1,T),
\end{array}\right.
\end{equation*}
where $y_i(t)$ are defined in Item (i), Step 1. The trajectories $z_i$ have no intersection.
By construction of trajectories $y_i(t)$, it holds $z_i(t)\in\omega$ for all $t\in(s_i^0,T-s^1_{\sigma(i)})$.
For all $i\in \{1,...,n\}$, 
choose $r_i,R_i$ satisfying $0<r_i<R_i$ and such that it holds
$$\begin{cases}
B_{r_i}(z_i(t))\subset B_{R_i}(z_i(t))\subset\omega
&\mbox{ for all }t\in(s_i^0,T-s^1_{\sigma(i)})\\
B_{R_i}(z_i(t))\cap B_{R_j}(z_j(t))=\varnothing&\mbox{ for all }t\in(0,T),
\end{cases}$$
with $i,j\in\{1,...,n\}$.
Such radii $r_i,R_i$ exist as a consequence of the fact that we deal with a finite number of trajectories that do not cross.
The corresponding control can be chosen as a function that is piecewise constant with respect to the time variable and $\mc{C}^{\infty}$ with respect to the $x$ variable given by:
\begin{equation*}
u(x,t):=
\begin{cases}
\dot y_i(t)-v & \mbox{ if }t\in (s_i^0,T-s^1_{\sigma(i)}) \mbox{ and }x\in B_{r_i}(z_i(t)),\\
0 & \mbox{ if }t\in (s_i^0,T-s^1_{\sigma(i)}) \mbox{ and }x\not\in B_{R_i}(z_i(t)),\\
\mbox{$C^\infty$-space-spline}& \mbox{ if }t\in (s_i^0,T-s^1_{\sigma(i)}) \mbox{ and }x\in B_{R_i}(z_i(t))\setminus B_{r_i}(z_i(t)),\\
0 & \mbox{ if }t\not\in (s_i^0,T-s^1_{\sigma(i)}),
\end{cases}\\
\end{equation*}
We have  built a control satisfying the Carath\'eodory condition such that each $i$-th component of the associated solution to System \eqref{eq ODE}  is $z_i(t)$, hence steering $x_i^0$ to $x^1_{\sigma(i)}$ in time $T$.

We now prove the result corresponding to {\bf Item (ii)} of Theorem \ref{th:discret exact}. Assume that System \eqref{eq ODE} 
is exactly controllable at a time $T> M^*_e(X^0,X^1)$, and consider $\sigma$ the corresponding permutation defined by $x_{i}(T)=x^1_{\sigma(i)}$. The idea of the proof is that the trajectory steers $x^0_i$ to $\omega$ in time $t^0_i$, then it moves inside $\omega$ for a small but positive time, then it steers a point from $\omega$ to $x^{1}_{\sigma(i)}$ in time $t^1_{\sigma(i)}$, hence $T>t^0_i+t^1_{\sigma(i)}$.

Fix an index $i\in\{1,...,n\}$. First recall the definition of $t^0_i,t^1_{\sigma(i)}$ and observe that it holds both $T>t^0_i$ and $T>t^1_{\sigma(i)}$. Then, the trajectory $x_i(t)$ satisfies\footnote{These estimates hold even in the specific case of $x^0_i\in \omega$, for which it holds $t^0_i=0$.}  $x_i(t)\not\in \omega$ for all $t\in (0,t^0_i)$, as well as $x_i(t)\not\in\omega$ for all $t\in(T-t^1_{\sigma(i)},T)$. Moreover, we prove that it exists $\tau_i\in(0,T)$ for which it holds $x_i(\tau_i)\in\omega$. By contradiction, if such $\tau_i$ does not exist, then the trajectory $x_i(t)$ never crosses the control region, hence it coincides with $\Phi_t^v(x^0_i)$. But in this case, by definition of $t^0_i$ as the infimum of times such that $\Phi_t^v(x^0_i)\in\omega$ and recalling that $t^0_i<T$, there exists $\tau_i\in(t^0_i,T)$ such that it holds $x_i(\tau_i)=\Phi_{\tau_i}^v(x^0_i)\in\omega$. Contradiction. Also observe that $\omega$ is open, hence there exists $\epsilon_i$ such that $x_i(\tau)\in\omega$ for all $\tau\in(\tau_i-\epsilon_i,\tau_i+\epsilon)$.
We merge the conditions $x_i(t)\not\in \omega$ for all $t\in (0,t^0_i)\cup(T-t^1_{\sigma(i)},T)$ with $x_i(\tau)\in\omega$ for all $\tau\in(\tau_i-\epsilon_i,\tau_i+\epsilon_i)$ with a given $\tau_i\in(0,T)$. This implies that it holds $t^0_i<\tau_i<T-t^1_{\sigma(i)}$, hence 
\begin{equation*}
T>t^0_i+t_{\sigma (i)}^1.
\end{equation*}
Such estimate holds for any $i\in\{1,...,n\}$. 
Using the definition of $\widetilde{M}_e(X^0,X^1)$, we deduce that $T>\widetilde{M}_e(X^0,X^1)$.

We finally prove the result corresponding to {\bf Item (iii)} of Theorem \ref{th:discret exact}. By definition of $M_e^*(X^0,X^1)$, 
there exists $l\in\{0,1\}$ and $m\in \{1,...,n\}$ such that $M^*_e(X^0,X^1)=t_m^l$. 
We only study the case $l=0$, since the case $l=1$ can be recovered by reversing time. We consider the trajectory starting from such $x^0_m$ only. By definition of $t_m^0$, the trajectory $\Phi^v_t(x^0_m)$ satisfies $\Phi^v_t(x^0_m)\not\in \omega$ for all $t\in[0,M^*_e(X^0,X^1)]$. Then, for any choice of the control $u$ localized in $\omega$, it holds $\Phi^{v+\mathds{1}_{\omega}u}_{t}(x_m^0)=\Phi^v_t(x^0_m)$, \textit{i.e.} the choice of the control plays no role in the trajectory starting from $x^0_m$ on the time interval $t\in[0,M^*_e(X^0,X^1)]$. Observe that it holds $v(\Phi^v_t(x^0_m))\neq 0$ for all $t\in[0,M_e^*(X^0,X^1)]$, due to the fact that the vector field is time-independent and the trajectory $\Phi^v_t(x^0_m)$ enters $\omega$ for some $t>M_e^*(X^0,X^1)$.

We now prove that the set of times $t\in[0,M_e^*(X^0,X^1)]$ for which exact controllability holds is finite. A necessary condition to have exact controllability at time $t$ is that the equation $\Phi^v_t(x^0_m)=x^1_i$ admits a solution for some time $t\in[0,M_e^*(X^0,X^1)]$ and index $i\in\{1,...,n\}$. Then, we aim to prove that the set of times-indexes $(t,i)$ solving such equation is finite. By contradiction, assume to have an infinite number of solutions $(t,i)$. Since the set $i\in\{1,...,n\}$ is finite, this implies that there exists an index $I$ and an infinite number of (distinct) times $t_k\in[0,M_e^*(X^0,X^1)]$ such that $\Phi^v_{t_k}(x^0_m)=x^1_I$. Since $v$ is an autonomous vector field, this implies that, for each pair $t_{k_1},t_{k_2}$ it holds $\Phi^v_{t_{k_1}-t_{k_2}}(x^1_i)=x^1_i$, hence $\Phi^v_t(x^1_i)$ is a periodic trajectory, with period $|t_{k_1}-t_{k_2}|$. By compactness of $[0,M_e^*(X^0,X^1)]$, there exists a converging subsequence (that we do not relabel) $t_k\to t_*\in[0,M_e^*(X^0,X^1)]$. Then $\Phi^v_t(x^1_i)$ is a periodic trajectory with arbitrarily small period, hence $x^1_i$ is an equilibrium\footnote{Such final statement can be easily proved by contradiction: if $v(x^1_i)\neq 0$, then apply the local rectifiability theorem for Lipschitz vector field \cite{rectify} and prove that any eventually periodic trajectory passing through $x^1_i$ has a strictly positive minimal period.}. This contradicts the fact that for all $t\in[0,M_e^*(X^0,X^1)]$ it holds $v(\Phi^v_t(x^0_m))\neq 0$.
%\hfill\qed
\end{proof}

\begin{remark}
Proposition \ref{prop: dim finie} can be interpreted as follows:
Each particle at point $x^0_i$ needs to be sent on a target point $x_{\sigma(i)}^1$. The time $\widetilde{M}_e(X^0,X^1)$ coincides with the infimum time for a particle at $x^0_i$ 
 to enter in $\omega$ and then go from $\omega$ to $x^1_{\sigma(i)}$.
We are thus assuming that the particle  travels with an arbitrarily large speed in $\omega$.
\end{remark}
Formula \eqref{minimal time} leads to the proof of Theorem \ref{th:discret exact}.

\noindent {\it Proof of Theorem \ref{th:discret exact}.} 
Let $\sigma_0$ be a minimizing permutation in  \eqref{minimal time}.
We build recursively a sequence of permutations $\{\sigma_1,...,\sigma_n\}$ as follows:
\begin{itemize}
\item[$\bullet$]Let $k_1$ be such that $t_{\sigma_0(k_1)}^1$ is a maximum of $\{t_{\sigma_0(1)}^1,...,t_{\sigma_0(n)}^1\}$.
We denote by $\sigma_1:=\sigma_0\circ\mc{T}_{1,k_1}$,  where, for all $i,j\in\{1,...,n\}$,
$\mc{T}_{i,j}$ is the transposition between the  $i$-th 
and the $j$-th elements. %As illustrated in  Figure \ref{fig:echange temps},
It holds 
\begin{equation*}
\left\{\begin{array}{l}
t^0_{k_1}+t^1_{\sigma_0(k_1)}
\geqslant t_1^0+t^1_{\sigma_0(1)},\\\noalign{\smallskip}
t^0_{k_1}+t^1_{\sigma_0(k_1)}
\geqslant t_1^0+t^1_{\sigma_0({k_1})}
=t_1^0+t^1_{\sigma_0\circ\mathcal{T}_{1,k_1}({1})}
=t_1^0+t^1_{\sigma_1({1})},\\\noalign{\smallskip}
t^0_{k_1}+t^1_{\sigma_0(k_1)}
\geqslant t_{k_1}^0+t^1_{\sigma_0({1})}
=t_{k_1}^0+t^1_{\sigma_0\circ\mathcal{T}_{1,k_1}({k_1})}
=t_{k_1}^0+t^1_{\sigma_1({k_1})}.
\end{array}\right.
\end{equation*}
Thus $\sigma_1$ minimizes  \eqref{minimal time} too, since it holds
\begin{equation*}
\max_{i\in\{1,...,n\}}\{t_i^0+t_{\sigma_0(i)}^1\}
\geqslant \max_{i\in\{1,...,n\}}\{t_i^0+t_{\sigma_1(i)}^1\}.
\end{equation*}
\item[$\bullet$]
Assume that $\sigma_{j}$ is built. 
Let $k_{j+1}$ be such that $t_{\sigma_j(k_{j+1})}^1$ is a maximum  of $\{t_{\sigma_{j}(j+1)}^1,...,\allowbreak t_{\sigma_j(n)}^1\}$.
Again, we clearly have 
\begin{equation*}
\left\{\begin{array}{rcl}
t^0_{k_{j+1}}+t^1_{\sigma_j(k_{j+1})}\geqslant t_{j+1}^0+t^1_{\sigma_j(j+1)},\\\noalign{\smallskip}
t^0_{k_{j+1}}+t^1_{\sigma_j(k_{j+1})}
\geqslant t_{j+1}^0+t^1_{\sigma_{j}({k_{j+1}})}
&=&t_{j+1}^0+t^1_{\sigma_{j}\circ\mathcal{T}_{j+1,k_{J+1}}({j+1})}\\\noalign{\smallskip}
&=&t_{j+1}^0+t^1_{\sigma_{j+1}({j+1})},\\\noalign{\smallskip}
t^0_{k_{j+1}}+t^1_{\sigma_j(k_{j+1})}
\geqslant t_{k_{j+1}}^0+t^1_{\sigma_{j}(k_{j})}
&=&t_{k_{j+1}}^0+t^1_{\sigma_{j}\circ\mathcal{T}_{j+1,k_{J+1}}(k_{j+1})}\\\noalign{\smallskip}
&=&t_{k_{j+1}}^0+t^1_{\sigma_{j+1}(k_{j+1})}.
\end{array}\right.
\end{equation*}
Thus $\sigma_{j+1}:=\sigma_j\circ\mc{T}_{j+1,k_{j+1}}$ 
minimizes  \eqref{minimal time} too:
\begin{equation*}
\max_{i\in\{1,...,n\}}\{t_i^0+t_{\sigma_{j}(i)}^1\}
\geqslant \max_{i\in\{1,...,n\}}\{t_i^0+t_{\sigma_{j+1}(i)}^1\}.
\end{equation*}
\end{itemize}
The sequence $\{t^1_{\sigma_n(1)},...,t^1_{\sigma_n(n)}\}$ is then decreasing and $\sigma_n$ is a minimizing permutation in \eqref{minimal time}.
We deduce that $\widetilde{M}_e(X^0,X^1)=M_e(X^0,X^1)$.
\hfill\qed

With Theorem \ref{th:discret exact}, we give an explicit and simple expression of the infimum time for microscopic models. This result is useful in numerical simulations of Section \ref{sec:num sim} (in particular, see Algorithm \ref{algo 1}).
Let $\mu^0$ and $\mu^1$ be the probability  measures defined by
\begin{equation}\label{def mu dim finie}
\mu^0:=\sum_{i=1}^n\frac{1}{n}\delta_{x_i^0}\mbox{ and }\mu^1:=\sum_{i=1}^n\frac{1}{n}\delta_{x_i^1},
\end{equation}
where the points $x^0_i$ (resp. $x^1_i$) are disjoint. 
We now deduce that $T_e(X^0,X^1)$  is equal to $S(\mu^0,\mu^1)$ given in \eqref{def T0}, when $\mu^0$ and $\mu^1$ are given by \eqref{def mu dim finie}.

\begin{corollary} \label{Coro:micro_macro}
Let $X^0:=\{x^0_1,...,x^0_n\}\subset\mb{R}^d$ and $X^1:=\{x^1_1,...,x^1_n\}\subset\mb{R}^d$ be two disjoint configurations.
Consider $\mu^0$ and $\mu^1$ the corresponding  measures given in \eqref{def mu dim finie}.
Assume that $\omega$ is is a bounded open connected set, $\mu^0$, $\mu^1$
 satisfy the Geometric Condition \ref{cond1} and 
  the velocities $v$ and $\mathds{1}_{\omega}u%:\mb{R}^d\times\mb{R}^+\rightarrow\mb{R}^d
 $ satisfy the Carath\'eodory Condition \ref{cond:cata}.
 
Then the infimum time $T_e(X^0,X^1)$  is equal to $S(\mu^0,\mu^1)$ given in \eqref{def T0}.
\end{corollary}
\begin{proof}
Remark that if the sequences $\{t_i^0\}_{i\in\{1,...n\}}$ and $\{t_i^1\}_{i\in\{1,...n\}}$ are non-decreasingly and non-increasingly  ordered respectively, 
then for all  $m\in \left(\frac{i-1}{n},\frac{i}{n}\right]$ it holds
\begin{equation*}
\mc{F}_{\mu^0}^{-1}(m)=t_i^0
\mbox{ ~~and~~ }
\mc{B}_{\mu^1}^{-1}(1-m)=t_i^1.
\end{equation*}
Then, the result is given by identification of the expression of $M_e(X^0,X^1)$ given in \eqref{OT disc CE} with the expression of $S(\mu^0,\mu^1)$ in \eqref{def T0}.
%\hfill\qed
\end{proof}

We now prove Theorem \ref{th:discret approx}, which characterizes the infimum time for approximate control of System \eqref{eq ODE}.

\noindent{\it Proof of Theorem \ref{th:discret approx}.} We first prove \textbf{Item (i)}. Consider $X^0:=\{x_1^0,...,x_n^0\}$ and  $X^1:=\{x_1^1,...,x_n^1\}$ two disjoint configurations  satisfying 
the Geometric Condition \ref{cond1}.
We first prove that the infimum time $T_a(X^0,X^1)$ is equal to
\begin{equation*}
\widetilde{M}_a(X^0,X^1):=\min_{\sigma\in \mathbb{S}_n}\max_{i\in\{1,...,n\}}|t^0_i+\overline{t}_{\sigma (i)}^1|.
\end{equation*}

First assume that $T>\widetilde{M}_a(X^0,X^1)$. Let $\varepsilon>0$. 
For each $x^1_i$, we need to find points $y_i^1$ satisfying 
\begin{equation}
\|y_i^1-x^1_i\|\leqslant \varepsilon
\mbox{ and }y_i:=\Phi^v_{-\overline{t}_i^1}(y_i^1)\in\omega.\label{e-yi1}
\end{equation}
For each $x^1_i$, observe that the Geometric Condition \ref{cond1} implies that either $x^1_i\in\omega$ or that the trajectory enters $\omega$ backward in time. In the first case, define $y_i^1:=x_i^1$. In the second case, remark that $v(\Phi^v_{-t}(x^1_i))$ is nonzero for a whole interval $t\in[0,\tilde t]$, with $\tilde t>\bar t^1_i$, 
and $\Phi^v_{-\bar t_i^1}(x^1_i)\in\overline{\omega}$, hence the flow $\Phi^v_{-\bar t^1_i}(\cdot)$ is an homeomorphism in a neighborhood of $x^1_i$. Then, there exists $y_i^1\in\mb{R}^d$ such that \eqref{e-yi1} is satisfied.

We denote by $Y^1:=\{y^1_1,...,y_n^1\}$.
For all $i\in\{1,...,n\}$, since $y_i\in\omega$, then $t^1(y_i^1)\leqslant \overline{t}_i^1$, hence 
$\widetilde{M}_e(X^0,Y^1)\leqslant\widetilde{M}_a(X^0,X^1)<T.$
Proposition \ref{prop: dim finie} implies that we can exactly steer $X^0$ to $Y^1$ at time $T$ with a control $u$
satisfying the Carath\'eodory condition. 
Denote by $X(t)$ the solution to System \eqref{eq ODE} for the initial condition $X^0$ and the control $u$.
It holds 
\begin{equation}W_1(X^1,X(T))=W_1(X^1,Y^1)\leqslant \sum_{i=1}^n\frac{1}{n}\|y_i^1-x^1_i\|\leqslant\varepsilon.\label{e-apprW1}
\end{equation}
Choose now $\varepsilon=\frac{1}{k}$ for each $k\in\mathbb{N}^*$, and denote by $u_k$ the corresponding control and with $x_{i,k}(t)$ the associated solution to System \eqref{eq ODE} for the $i$-th particle. By the definition of $W_1$ in the discrete case and \eqref{e-apprW1}, there exists a permutation $\sigma_k\in \mathbb{S}_n$ for which it holds 
$$\left\| x_{i,k}(T)-x^1_{\sigma_k(i),k}\right\|\leq \frac{n}{k},$$
for all $i=1,\ldots, n$. Since the space of permutations $\mathbb{S}_n$ is finite, one can extract a subsequence with $\sigma_k$ constant. With such subsequence, the statement (i) is proved.

We now prove \textbf{Item (ii)}. Consider a time $T>M^*_a(X^0,X^1)$ at which System \eqref{eq ODE} is approximately controllable. We aim to prove that it satisfies $T>\widetilde{M}_a(X^0,X^1)$. 
For each $k\in\mathbb{N}^*$, there exists a control $u_k$
satisfying the Carath\'eodory condition such that the corresponding solution $X_k(t)$ to System \eqref{eq:transport} satisfies
\begin{equation}\label{eq:W1 1/k}
W_1(X^1,X_k(T))\leqslant \frac{1}{k}.
\end{equation}
We denote by $Y^1_k:=\{y_{k,1}^1,...,y_{k,n}^1\}$ the configuration defined by $$y_{k,i}^1:=X_{k,i}(T),$$ where $X_{k,i}$ is the $i$-th component of $X_k$. Since $X^0$ is disjoint and $u_k$ satisfies the Carath\'eodory condition, then $Y^1_k$ is disjoint too.
We now prove that it holds
\begin{equation} \label{T_Me*}
T > M^*_e(X^0, Y^1_k).
\end{equation}
Since  $T>M^*_a(X^0,X^1)$, then \eqref{T_Me*} is equivalent to
$T > t_i^1(y_{k,i}^1)$ for all $i\in\{1,...,n\}.$
By contradiction, assume that there exists $j \in \{1,\dots,n\}$ such that $$t^1(y^1_{k,j}) \geqslant T.$$
We distinguish two cases:
\begin{itemize}
\item[$\bullet$]If $t^1(y^1_{k,j}) > T$, then for any $t \in [0,T]$ it holds $\Phi_{-t}^v(y^1_{k,j}) \not \in \omega$. Thus, the localized control does not act on the trajectory, \textit{i.e.} for each $t\in[0,T]$ it holds
$\Phi_{-t}^v(y^1_{k,j}) = \Phi_{-t}^{v+\mathds{1}_{\omega}u_k}(y^1_{k,j}).$
Since $y_{k,j}^1=\Phi_{T}^{v+\mathds{1}_{\omega}u_k}(x^0_j)=\Phi^v_T(x^0_j)$, then
$\Phi_{t}^{v}(x^0_j)\not\in\omega$
for all $t\in[0,T]$.
This is a contradiction with the fact that  $t^0_j\leqslant M_a^*(X^0,X^1) < T.$ Thus~\eqref{T_Me*} holds.
\item[$\bullet$] If $t^1(y^1_{k,j}) = T$, then for all $t \in [0,T)$ it holds $\Phi_{-t}^v(y^1_{k,j}) \not \in \omega$. Since $\omega$ is open, then it also holds $\Phi_{-T}^v(y^1_{k,j}) \not \in \omega$. We then conclude as in the previous case.
\end{itemize}
Since $Y_k^1=X_k(T)$, then Proposition \ref{prop: dim finie} implies that 
\begin{equation}\label{ine Me Ma}
T>\widetilde{M}_e(X^0,Y^1_k).
\end{equation}
For each control $u_k$, denote by $\sigma_k$ the permutation for which $y^1_{k,i}=\Phi_{T}^{v+\mathds{1}_{\omega}u_k}(x^0_{\sigma_k(i)})$. Up to extract a subsequence, for all $k$ large enough, $\sigma_k$ is equal to a permutation $\sigma$.
Inequality \eqref{eq:W1 1/k} implies that for all $i\in\{1,...,n\}$ it holds
\begin{equation}\label{y_k,i}
y_{k,i}^1\underset{k\rightarrow\infty}{\longrightarrow}x_{\sigma(i)}^1.
\end{equation}
Since $t^1(y_{k,i}^1)\leqslant \widetilde{M}_e(X^0,Y^1_k)<T$, up to a subsequence, for a $s_i\geqslant 0$, it holds
\begin{equation}\label{t_k,i}
t^1(y_{k,i}^1)\underset{k\rightarrow\infty}{\longrightarrow}s_{i}.
\end{equation}
Using \eqref{y_k,i}, \eqref{t_k,i}  and the continuity of the flow, it holds
$$\begin{array}{l}\|\Phi_{-t^1(y_{k,i}^1)}^v(y_{k,i}^1)-\Phi_{-s_i}^v(x_{\sigma(i)}^1)\|\\
\hspace{1cm}\leqslant \|\Phi_{-t^1(y_{k,i}^1)}^v(y_{k,i}^1)-\Phi_{-s_i}^v(y_{k,i}^1)\|+\|\Phi_{-s_i}^v(y_{k,i}^1)-\Phi_{-s_i}^v(x_{\sigma(i)}^1)\|
\underset{k\rightarrow\infty}{\longrightarrow}0.\end{array}$$
The fact that $\Phi_{-t^1(y_{k,i}^1)}^v(y_{k,i}^1)\in\overline{\omega}$ for each $i=1,\ldots,n$ leads to 
$\Phi_{-s_i}^v(x_{\sigma(i)}^1)\in\overline{\omega}.$
Thus
\begin{equation*}\overline{t}^1(x_{\sigma(i)}^1)\leqslant \lim\limits_{k\rightarrow\infty} t^1(y^1_{k,i}).
\end{equation*}
Denoting by $\delta:=(T-\widetilde{M}_e(X^0,X^1))/2$ and 
using \eqref{ine Me Ma}, there exists $K\in\mathbb{N}$ such that for all $k>K$ it holds
\begin{equation*}\begin{array}{rcl}
\widetilde{M}_a(X^0,X^1)&=&\min_{\widetilde{\sigma}\in \mathbb{S}_n}\max_{i\in\{1,...,n\}}|t^0_i+\overline{t}_{\widetilde{\sigma} (i)}^1|\\
&\leqslant&\max_{i\in\{1,...,n\}}|t^0_i+\overline{t}_{\sigma (i)}^1|\\
&\leqslant&\max_{i\in\{1,...,n\}}|t^0_i+t^1(y^1_{k,\sigma(i)})|+\delta\\
&=&\widetilde{M}_e(X^0,Y^1_k)+\delta<T.
\end{array}
\end{equation*}
We finally prove \textbf{Item (iii)} of Theorem~\ref{th:discret approx}. 
Let $T \in (0,M_a^*(X^0,X^1))$ be such that System~\eqref{eq ODE} is approximately controllable. 
For any $\varepsilon>0$, there exists $u_\varepsilon$ such that the associated trajectory to System~\eqref{eq ODE} satisfies
\begin{equation}\label{W1 Ma*}
W_1 (X_{\varepsilon}(T),X^1) < \varepsilon.
\end{equation}
There exists $j \in \{1, \dots,n\}$ s.t. $t^0(x^0_j) = M_a^*(X^0,X^1) > T$ or $\overline{t}^1(x^1_j) = M_a^*(X^0,X^1) > T$. 
We distinguish these two cases:
\begin{itemize}
\item[$\bullet$] Assume that %
$t^0(x^0_j) = M_a^*(X^0,X^1) > T$. 

Define $x_{\varepsilon,j}(t):= \Phi_t^{v+\mathds{1}_{\omega}u_{\varepsilon}}(x_j^0)$. 
 Inequality \eqref{W1 Ma*} implies that there exists $k \in \{1, \dots, n\}$ such that
\begin{equation} \label{proche_cible}
\| x_{\varepsilon,j}(T) - x^1_{k_{\varepsilon}} \| < \varepsilon.
\end{equation}
As $t^0(x^0_j) > T$, the trajectory $ \Phi_t^{v}(x_j^0)$ does not cross the control set $\omega$ for $t\in[0,T)$, hence 
$x_{\varepsilon,j}(T)= \Phi_T^{v+\mathds{1}_{\omega}u_{\varepsilon}}(x_j^0)=\Phi_T^v(x_j^0)$
 does not depend on $\varepsilon$.
Define $R:= \frac{1}{2} \min_{p,q}\|x^1_p\allowbreak - x^1_q \|$, that is strictly positive since $X^1$ is disjoint. 
For each $\epsilon<R$, Estimate~\eqref{proche_cible} gives $k_{\varepsilon}=k$ independent on $\varepsilon$ and
\begin{equation*}%
x_{\varepsilon,j}(T) =  \Phi_t^{v}(x_j^0)=x_k^1.
\end{equation*}
Use now the proof of Item (iii) in Proposition \ref{prop: dim finie} to prove that the equation $\Phi_t^{v}(x_j^0)=x_k^1$ admits a finite number of solutions $(t,k)$ with $t\in[0,t^0(x^0_j)]$ and $k\in \{1, \dots,n\}$.
\item[$\bullet$] Assume now that 
$\overline{t}^1(x^1_j) = M_a^*(X^0,X^1) > T$. 
Again, inequality \eqref{W1 Ma*} implies that there exists $k_{\varepsilon} \in \{1, \dots, n\}$ such that
\begin{equation*}\| x_{\varepsilon,k_{\varepsilon}}(T) - x^1_j \|=\| \Phi_T^{v+\mathds{1}_{\omega}u_{\varepsilon}}(x_{k_{\varepsilon}}^0) - x^1_j \| < \varepsilon.
\end{equation*}
As $\overline{t}^1(x^1_j) > T$, the trajectory $ \Phi_{-t}^{v}(x_j^1)$ does not cross the control set $\overline{\omega}$ for all  $t\in[0,T]$.
Then for all $t\in[0,T]$, there exists $r(t)>0$ such that 
$B_{r(t)}(\Phi_{-t}^{v}(x_j^1))\cap\omega=\varnothing.$
By compactness of the set $\{\Phi_{-t}^{v}(x_j^1):t\in[0,T]\}$, there exist $t_1,...,t_N\in[0,T]$ for which it holds
$$\{\Phi_{-t}^{v}(x_j^1):t\in[0,T]\}\subset\bigcup_{i=1}^NB_{r(t_i)}(\Phi_{-t_i}^{v}(x_j^1))\subset\omega^c.$$
Thus there exists a common $r>0$ such that 
$B_r(\Phi_{-t}^{v}(x_j^1))\cap\omega=\varnothing$
for all $t\in[0,T]$.
For $\varepsilon<r$,  it holds $k_{\varepsilon}=k$ and
$$x_{\varepsilon,k}(T)=\Phi^v_T(x_k^0)=x^1_j.$$
We conclude as in the previous case.
\end{itemize}
Finally, apply the permutation method used in the proof of Theorem \ref{th:discret exact} to prove that it holds
$\widetilde{M}_a(X^0,X^1)=M_a(X^0,X^1).$
\hfill\qed
\begin{remark}\label{rmq:T2*}
We  illustrate Item (iii) of 
 Theorem \ref{th:discret exact}
 by giving an example in which System \eqref{eq ODE} is never exactly controllable on 
 $[0,M_e^*(X^0,X^1))$  and another where System \eqref{eq ODE} is exactly controllable
 at a time $T\in[0,M_e^*(X^0,X^1))$.\\
  {\it Figure \ref{fig:ex (0,T*)} (left):}
Consider $\omega:=(-1,1)\times(-1.5,1.5)$, $v:=(1,0)$, $X^0:=(-2,0)$ and $X^1:=(2,0)$.  The time $M_e^*(X^0,X^1)$ 
at which we can act on the particles and the minimal time $M_e(X^0,X^1)$ 
are respectively equal to $1$ and $2$.
We observe that System \eqref{eq ODE} is neither exactly controllable nor approximately controllable
on the interval $[0,M_e^*(X^0,X^1))$.\\
{\it Figure \ref{fig:ex (0,T*)} (right):}
Consider $\omega:=(-2,0)\times(-1.5,1.5)$, the vector field $v(x,y)=(-y,x)$ and $X^0$, $X^1$ as follows
$$
X^0:=\{(\sqrt{2}/2,-\sqrt{2}/2)\},\qquad
X^1:=\{(\sqrt{2}/2,\sqrt{2}/2)\}.$$ \\
The time $M_e^*(X^0,X^1)$ 
at which we can act on the particles and the minimal time $M_e(X^0,X^1)$ 
are respectively equal to $3\pi/4$ and $\pi$.
We remark that 
System \eqref{eq ODE} is exactly controllable, then approximately controllable 
at time $T=\pi/2\in [0,M_e^*(X^0,X^1))$.

\begin{figure}[ht]
\begin{center}
\hfill
\begin{tikzpicture}[scale=1]
\fill[pattern=dots,opacity = 0.5] (-2,-1.5) -- (0,-1.5) -- (0,1.5) -- (-2,1.5) -- cycle;
\draw (-2,-1.5) -- (0,-1.5) -- (0,1.5) -- (-2,1.5) -- cycle;
\node[cross,thick,minimum size=4pt] at (-3,0) {};
\node[cross,thick,minimum size=4pt] at (1,0) {};
\path (-3,0.3) node {$x^0_1$};
\path (1,0.3) node {$x^1_1$};
\path (-1,-1) node {$\omega$};
\path (-2.5,-0.3) node {$v$};
\draw (-3,0) -- (1,0);
\draw (-2.4,0) -- (-2.5,0.1);
\draw (-2.4,0) -- (-2.5,-0.1);
\draw (0.5,1) -- (1.5,1);
\draw (0.5,1.1) -- (0.5,0.9);
\draw (1.5,1.1) -- (1.5,0.9);
\path (1,1.25) node {\scriptsize 1};
\end{tikzpicture}
\hfill
\begin{tikzpicture}[scale=1]
\fill[pattern=dots,opacity = 0.5] (-2,-1.5) -- (0,-1.5) -- (0,1.5) -- (-2,1.5) -- cycle;
\draw (-2,-1.5) -- (0,-1.5) -- (0,1.5) -- (-2,1.5) -- cycle;
\node[cross,thick,minimum size=4pt] at (0.7071,0.7071) {};
\node[cross,thick,minimum size=4pt] at (0.7071,-0.7071) {};
\path (0.8,1) node {$x^1_1$};
\path (0.8,-1) node {$x^0_1$};
\path (-1,-1) node {$\omega$};
\path (1.3,0) node {$v$};
\draw (1,0) -- (0.9,-0.1);
\draw (1,0) -- (1.1,-0.1);
\draw (0,0) circle (1);
\end{tikzpicture}\hfill~
\caption{Examples in the case $T\in(0,M_e^*(X^0,X^1))$.}\label{fig:ex (0,T*)}
\end{center}\end{figure}
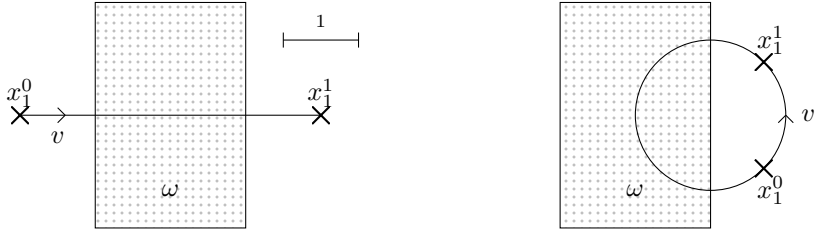

\end{remark}

\section{Infimum time for macroscopic models}\label{sec:optimal time cont}

This section is devoted to the proof of main Theorem \ref{th opt} about infimum time for AC measures. 
We first introduce the notion of infimum time up to small mass in Section \ref{sec:up mass}. We then give some comparisons between the microscopic and macroscopic cases in Section \ref{sec: comp}. We finally use the obtained results to prove Theorem \ref{th opt} in Section \ref{sec:proof th 3}.
\subsection{Infimum time in the microscopic setting up to small masses}\label{sec:up mass}

In this section, we introduce the notion of infimum time up to small mass
and prove some results about this notion in the microscopic case.

From now on, we denote by $\lfloor x \rfloor$ the floor of $x\in\mathbb{R}$.
\begin{definition}[\textbf{Infimum time up to small mass}] Consider $\mu^0,\mu^1$ be two AC-measures compactly supported with the same total mass $\mu^0(\mb{R}^d)=\mu^1(\mb{R}^d)=\gamma>0$. 
Recall that $S^*(\mu^0,\mu^1)$ is defined by~\eqref{e:Sstar}.
We denote by $T_{e,\varepsilon}(\mu^0,\mu^1)$ the infimum time to exactly steer $\mu^0$ to $\mu^1$ up to a mass $\varepsilon$. More precisely, it is
 the infimum of times $T\geqslant S^*(\mu^0,\mu^1)$ for which
 there exists a control 
$\mathds{1}_{\omega}u:\mb{R}^d\times\mb{R}^+\rightarrow\mb{R}^d$ satisfying the Carath\'eodory condition
and two sets $R^0,R^1$ of $\mb{R}^d$ such that $\mu^0(R^0)=\mu^1(R^1)=\varepsilon$ and 
$$\Phi_T^{v+\mathds{1}_{\omega}u}\#\mu^0_{|(R^0)^c}=\mu^1_{|(R^1)^c}.$$
We similarly define the minimal time $T_{a,\varepsilon}(\mu^0,\mu^1)$ to approximately steer $\mu^0$ to $\mu^1$  up to a mass $\varepsilon$.

Let $\gamma>0$ and  $X^0:=\{x^0_1,...,x^0_n\}\subset\mb{R}^d$, $X^1:=\{x^1_1,...,x^1_n\}\subset\mb{R}^d$ be two  disjoint configurations (see Definition \ref{def:disjoint}).
We denote by $T_{e,\varepsilon}(X^0,X^1,\gamma)$ the infimum time  to exactly steer $X^0$ to $X^1$ up to a mass $\varepsilon$ for a total mass $\gamma$. 
More precisely, it is the infimum of times $T\geqslant M^*_e(X^0,X^1)$ such that there exists a control 
$\mathds{1}_{\omega}u:\mb{R}^d\times\mb{R}^+\rightarrow\mb{R}^d$ satisfying the Carath\'eodory condition and two permutations 
$\sigma_0,\sigma_1\in \mathbb{S}_n$ with
% $R:=\lfloor n\varepsilon \gamma \rfloor$ 
$R:= \left\lfloor \frac{n\varepsilon}{\gamma} \right\rfloor$ such that
$$x_{\sigma_0(i)}(T)=x^1_{\sigma_1(i)},$$
for all $i\in\{1,...,n-R\}$.
We similarly define $T_{a,\varepsilon}(X^0,X^1,\gamma)$ the minimal time  to approximately steer $X^0$ to $X^1$  up to a mass $\varepsilon$ for a total mass $\gamma$.
\end{definition}
\begin{remark}
If $\varepsilon<\gamma/n$, then the infimum time up to a small mass clearly coincides with the standard infimum time studied in Section \ref{sec:opt time finite dim}.
\end{remark}
\begin{remark}
Let $\gamma>0$ and $X^0:=\{x^0_1,...,x^0_n\}\subset\mb{R}^d$, $X^1:=\{x^1_1,...,x^1_n\}\subset\mb{R}^d$ be two  disjoint configurations (see Definition \ref{def:disjoint}).
We remark that 
$$T_{e,\varepsilon}(X^0,X^1,\gamma)=T_{e,\varepsilon}(\mu^0,\mu^1)$$
for the  measures $\mu^0$ and $\mu^1$ given by
\begin{equation}\label{def mu dim finie m}
\mu^0:=\sum_{i=1}^n\frac{\gamma}{n}\delta_{x_i^0}\mbox{ and }\mu^1:=\sum_{i=1}^n\frac{\gamma}{n}\delta_{x_i^1}.
\end{equation}
\end{remark}
For all pair of measures $\mu^0$ and $\mu^1$ compactly supported with the same total mass $\gamma:=\mu^0(\mb{R}^d)=\mu^1(\mb{R}^d)$,
we define the quantity
\begin{equation} \label{e:Seps}
S_{\varepsilon}(\mu^0,\mu^1)
:=\sup\limits_{m\in[0,\gamma-\varepsilon]}\{\mc{F}^{-1}_{\mu^0}(m)+\mc{B}^{-1}_{\mu^1}(\gamma-\varepsilon-m)\}.
\end{equation}
We have the following result:

\begin{proposition}\label{prop:up to mass discret}
Let $\gamma>0$ and $X^0:=\{x^0_1,...,x^0_n\},X^1:=\{x^1_1,...,x^1_n\}\subset\mb{R}^d$ be two disjoint configurations.
Assume that $\omega$ is a bounded open connected set, the  empirical measures associated to $X^0$ and $X^1$ (see \eqref{eq:emp})
 satisfy the Geometric Condition \ref{cond1} and 
  the velocities $v$ and $\mathds{1}_{\omega}u%:\mb{R}^d\times\mb{R}^+\rightarrow\mb{R}^d
 $ satisfy the Carath\'eodory Condition \ref{cond:cata}.
%Assume that $\omega$ is convex and the  empirical measures associated to $X^0$ and $X^1$ defined in \eqref{eq:emp} satisfy the Geometric Condition \ref{cond1}.
%%
  Assume that  the sequences $\{t_i^0\}_{i\in\{1,...n\}}$ and $\{t_i^1\}_{i\in\{1,...n\}}$ 
are respectively non-decreasingly and non-increasingly ordered. 
The infimum time $T_{e,\varepsilon}(X^0,X^1,\gamma)$ to exactly steer  $\mu^0$ to $\mu^1$  up to a mass $\varepsilon$ for a total mass $\gamma$
 is equal to
\begin{equation*}
M_{e,\varepsilon}(X^0,X^1,\gamma):=\max\limits_{1\leq i\leq n-R} \{t^0(x^0_i)+t^1(x^1_{i+R})\},
\end{equation*}
where $R:=\lfloor n\varepsilon \gamma \rfloor$. Moreover, it is also equal to
\begin{equation} \label{e:Seps bis}
S_{\varepsilon}(X^0,X^1,\gamma):=S_{\varepsilon}(\mu^0,\mu^1),
\end{equation}
where the measures $\mu^0$ and $\mu^1$ are defined in \eqref{def mu dim finie m}.
\end{proposition}

%\textcolor{red}{A mon avis: The proof is a simple adaptation of the proof of Proposition \ref{prop: dim finie} and is then omitted.}

\begin{proof}
We can adapt the proof of Proposition \ref{prop: dim finie} as follows. We first replace the set $\mathbb{B}_n$ of bistochastic matrices in \eqref{eq:inf} by the set $\mathbb{B}_{n,R}$ composed with the matrices satisfying 
\begin{equation*}
\sum_i\pi_{ij}\leqslant 1,~~  \sum_j\pi_{ij}\leqslant 1,
~~  \sum_{ij}\pi_{ij}=n-R\mbox{ and }\pi_{ij}\geqslant0.
\end{equation*}
The set $\mathbb{B}_{n,R}$ is closed and convex. Consider the minimization problem  
\begin{equation}\label{eq:inf2}
\inf\limits_{\pi\in\mb{B}_{n,R}}\frac{1}{n}\sum_{i,j=1}^n K_{ij}(y_i^0,s_i^0,y_j^1,s_j^1)\pi_{ij},
\end{equation}
where the quantities $K_{ij}(y_i^0,s_i^0,y_j^1,s_j^1)$ are defined in \eqref{Kij}.
Also in this case, the infimum in \eqref{eq:inf2} is finite since $T>\widetilde{M}_{e,\varepsilon}(X^0,X^1,\gamma)$.
Problem \eqref{eq:inf2} is  linear, hence, again as a consequence of Krein-Milman's Theorem (see \cite{KM40}), 
some minimisers of this functional are extremal points, 
that are matrices composed of a permutation sub-matrix for some rows and columns, and zeros for other rows and columns.
We then define 
\begin{equation*}\widetilde{M}_{e,\varepsilon}(X^0,X^1,\gamma):=\min_{\sigma_0,\sigma_1\in \mathbb{S}_n}\max_{1\leq i\leq n-R}
|t_{\sigma_0(i)}^0+t_{\sigma_1 (i)}^1|.
\end{equation*}
By applying the permutation method of the proof of Theorem \ref{th:discret exact}, 
we prove that 
$M_{e,\varepsilon}(X^0,\allowbreak X^1,\gamma)$ is equal to $\widetilde{M}_{e,\varepsilon}(X^0,X^1,\gamma)$. 

Following the lines of Corollary~\ref{Coro:micro_macro}, this also implies that formula \eqref{e:Seps bis} holds, by using the definitions of $\mc{F}_{\mu^0}^{-1}$ and $\mc{B}_{\mu^1}^{-1}$.
%\hfill\qed
\end{proof}

\subsection{Comparison of microscopic and macroscopic cases}\label{sec: comp}

The goal of this section is to give some estimates of the functions  
$S^*$ and $S_{\varepsilon}$  defined in \eqref{e:Sstar} and \eqref{e:Seps}, depending both on their arguments $\mu^0,\mu^1$ and the control region.
 They will be instrumental to prove Theorem \ref{th opt} in Section \ref{sec:proof th 3}.
For this reason, we will specify the considered control region 
in the notation $t^0$, $t^1$, $\mc{F}_{\mu^0}$, $\mc{B}_{\mu^1}$, $\mc{F}^{-1}_{\mu^0}$ and  $\mc{B}^{-1}_{\mu^1}$
 in an index, \textit{i.e.}  $t^0_{\omega}$, $t^1_{\omega}$, $\mc{F}_{\mu^0,\omega}$, $\mc{B}_{\mu^1,\omega}$, $\mc{F}^{-1}_{\mu^0,\omega}$ and  $\mc{B}^{-1}_{\mu^1,\omega}$. For the functions $S$, $S_{\varepsilon}$ and $S^*$, the control region will be specified as follows: 
 $S(\mu^0,\mu^1,\omega)$, $S_{\varepsilon}(\mu^0,\mu^1,\omega)$ and $S^*(\mu^0,\mu^1,\omega)$.
\begin{proposition}\label{prop cont T0}
% Let $\omega$ be the control set, assumed open, bounded and convex. 
Let $\mu^0,\mu^1\in\mathcal{P}_c^{ac}(\mathbb{R}^d)$.
%two AC-measures compactly supported with a total mass $1:=\mu^0(\mb{R}^d)=\mu^1(\mb{R}^d)$ and 
%satisfying the Geometric Condition \ref{cond1} with respect to $\omega$.
Assume that $\omega$ is a bounded open connected set, $\mu^0$, $\mu^1$
 satisfy the Geometric Condition \ref{cond1} and 
  the velocities $v$ and $\mathds{1}_{\omega}u%:\mb{R}^d\times\mb{R}^+\rightarrow\mb{R}^d
 $ satisfy the Carath\'eodory Condition \ref{cond:cata}.
Fix $\{f_n\}_{n\in\mb{N}}$ be a decreasing sequence of positive numbers converging to zero. For each $n\in\mb{N}$, let $\omega_n$ be defined by
\begin{equation}\label{omega n}
\omega_n:=\{x\in\mb{R}^d:d(x,\omega^c)>f_n\}.
\end{equation} 
Let two sequences $\{\mu^0_n\}_{n\in\mb{N}}$, $\{\mu^1_n\}_{n\in\mb{N}}$
of compactly supported measures be gi\-ven, satisfying the Geometric Condition \ref{cond1}
with respect to $\omega$.
Consider also two sequences of Borel sets $\{R^0_n\}_{n\in\mb{N}}$, $\{R^1_n\}_{n\in\mb{N}}$ of $\mb{R}^d$ 
that satisfy
\begin{equation}\begin{cases}
r_n:=\mu^0(R_n^0)=\mu^1(R_n^1)\underset{n\rightarrow\infty}{\longrightarrow} 0,\\
\mu^0_n(\mb{R}^d)=\mu^1_n(\mb{R}^d)=1-r_n,\\
W_{\infty}(\mu^0_{|(R_n^0)^c},\mu^0_n)<f_n,~
W_{\infty}(\mu^1_{|(R_n^1)^c},\mu^1_n)<f_n.
\label{e-Rn}
\end{cases}
\end{equation}
Let  
$S_{\varepsilon}(\cdot,\cdot,\tilde \omega)$ be the function defined in \eqref{e:Seps} for a given control set $\tilde\omega$.
Then for each $\varepsilon,\delta>0$, there exists $N\in\mb{N}^*$ 
such that for all $n\geq N$, 
it holds
\begin{itemize}
\item[(i)]$S_{2\varepsilon}(\mu^0_n,\mu^1_n,\omega_n)\leqslant S_{\varepsilon}(\mu^0,\mu^1,\omega)+\delta.$
\item[(ii)]$S_{2\varepsilon}(\mu^0,\mu^1,\omega)\leqslant S_{\varepsilon}(\mu^0_n,\mu^1_n,\omega_n)+\delta.$
\end{itemize}
Moreover,  there exist  two sequences of Borel sets $U_{0,n},U_{1,n}\subset\mb{R}^d$
such that \begin{equation*}
\left\{\begin{array}{l}
\mu^0_n(U_{0,n})=\mu^1_n(U_{1,n})\underset{n\rightarrow\infty}{\longrightarrow} 1,\\
\limsup\limits_{n\rightarrow\infty}S^*(\mu^0_{n|U_{0,n}},\mu^1_{n|U_{1,n}},\omega_n)
\leqslant S^*(\mu^0,\mu^1,\omega).
\end{array}\right.\end{equation*}

\end{proposition}

\begin{proof} 
We first prove \textbf{Item (i)}. Fix $\varepsilon,\delta>0$. For each $n\in\mb{N}$, define
\begin{equation*}
\begin{array}{l}
V_n^0:=\{x^0\in\supp(\mu^0):\Phi_t^v(B_{f_n}(x^0))\subset\subset\omega_n
\mbox{~for~a~}t\in[0,t^0_{\omega}(x^0)+\delta/2]\}.
\end{array}
\end{equation*}
We define 
\begin{equation}\label{e-rn0}
r_n^0:=\mu^0((V_n^0)^c).
\end{equation}
Since $\{f_n\}_{n\in\mathbb{N}^*}$ is decreasing, then  $\{\omega_n\}_{n\in\mathbb{N}^*}$
is a sequence of increasing sets, hence $V_n^0$ is a sequence of increasing sets. Thus $r_n^0$ is decreasing, hence
 $\lim_{n\to\infty}r_n^0$ exists.
We now prove that it holds 
\begin{equation}\label{pr T ep0}
\lim_{n\to\infty}r_n^0=0.
\end{equation}
We prove it by contradiction. Assume that $\lim_{n\to\infty}r_n^0 > 0$. Then it holds $\mu^0(( V_{n}^0)^c)>C>0$ for all $n\in\mb{N}^*$. 
In particular, the limit set $\left(\cap_{n=1}^\infty (V_n^0)^c\right)\cap \supp(\mu^0)$ is non-empty, hence there exists 
 $x^0\in\supp(\mu^0)$ such that
\begin{equation}\label{pr T ep1}
\Phi_t^v(B_{f_{n}}(x^0))\not\subset\omega_{n},
\end{equation}
for all $t\in[0,t^0_{\omega}(x^0)+\delta/2]$ and $n\in\mb{N}^*$. 
Due to the Geometric Condition \ref{cond1}, such time $t^0_{\omega}(x^0)$ is finite. 
Since $\omega$ is open, for a $t^*_{\omega}(x^0)\in(t^0_{\omega}(x^0),t^0_{\omega}(x^0)+\delta/2)$ and a $r(x^0)>0$, it holds
$B_{r(x^0)}(\Phi^v_{t^*_{\omega}(x^0)}(x^0))\subset\subset\omega.$
By continuity of $\Phi_t^v$, there exists $\widehat{r}(x^0)>0$ such that
\begin{equation}\label{pr T ep2}
\Phi^v_{t^*_{\omega}(x^0)}(B_{\widehat{r}(x^0)}(x^0))\subset\subset\omega.
\end{equation}
Since \eqref{pr T ep1} and \eqref{pr T ep2} are in contradiction, for $n$ large enough,
we conclude that \eqref{pr T ep0} holds. We deduce that, for all $x^0\in V_n^0$, it holds
\begin{eqnarray}
\xi^0\in \overline{B_{f_n}(x^0)}&\Rightarrow& \Phi^v_t(\xi_0)\in \omega_n\mbox{~~for some~~}t\in[0,t^0_{\omega}(x^0)+\delta/2]\nonumber\\
&\Rightarrow&  t^0_{\omega_n}(\xi^0)\leqslant t^0_{\omega}(x^0)+\delta/2. \label{lemme gamma1}
\end{eqnarray}
For each $n\in\mb{N}^*$, 
consider an optimal transference plan $\pi_n$ realizing the $\infty$-Wasserstein distance $W_\infty(\mu^0_{|(R_n^0)^c},\mu^0_n)$, see Proposition \ref{prop Wp}. We remark that \eqref{e-Rn} implies that \begin{equation}\label{lemme gamma2}
|x^0-\xi^0|< f_n \mbox{~~for $\pi_n$-almost every $(x^0,\xi^0)\in (R_n^0)^c\cap\supp(\mu^0)\times\supp(\mu^0_n)$}. 
\end{equation}
Thus, combining \eqref{lemme gamma1} and \eqref{lemme gamma2}, it holds
\begin{equation}\label{ine t0}
t^0_{\omega_n}(\xi^0)\leqslant t^0_{\omega}(x^0)+\delta/2
\end{equation}
for $\pi_n$-almost every pair $(x^0,\xi^0)$ with $x^0\in (R_n^0)^c\cap V_n^0$ and $\xi^0\in\supp(\mu^0_n)$. Using \eqref{e-rn0} and \eqref{ine t0}, 
we obtain
\begin{eqnarray*}
&&\mc{F}_{\mu^0,\omega}(t) \leqslant
\mu^0(\{x^0\in  (R_n^0)^c\cap V_n^0:t^0_{\omega}(x^0)\leqslant t\})+\mu^0(R^0_n)+\mu^0((V^0_n)^c)\\
&&=\pi_n(\{(x^0,\xi^0)\in ((R_n^0)^c\cap V_n^0)\times\supp(\mu^0_n):
t^0_{\omega}(x^0)\leqslant t\})+r_n+r_n^0\\
&&\leqslant\pi_n(\{(x^0,\xi^0)\in  ( (R_n^0)^c\cap V_n^0)\times\supp(\mu^0_n): t^0_{\omega_n}(\xi^0)\leqslant t+\delta/2\})+r_n+r_n^0\\
&&\leqslant \mu^0_n(\{\xi^0\in \supp(\mu^0_n):t^0_{\omega_n}(\xi^0)\leqslant t+\delta/2\})+r_n+r_n^0\\
&&= \mc{F}_{\mu^0_n,\omega_n}(t+\delta/2)+r_n+r_n^0,
\end{eqnarray*}
for all $t\in\mb{R}^+$.
Similarly, we have
\begin{equation*}
\mc{B}_{\mu^1,\omega}(t)\leqslant \mc{B}_{\mu^1_n,\omega_n}(t+\delta/2)+r_n+r_n^1,
\end{equation*}
for all $t\in\mb{R}^+$, where $r_n^1$ is defined similarly to \eqref{e-rn0}.
We deduce that the generalized inverse satisfies
\begin{eqnarray*}
\mc{F}^{-1}_{\mu^0_n,\omega_n}(m)&:=&\inf\{t\geqslant 0:\mc{F}_{\mu^0_n,\omega_n}(t)\geqslant m\}
\leqslant\inf\{t\geqslant \delta/2:\mc{F}_{\mu^0,\omega}(t-\delta/2)-r_n^0-r_n\geqslant m\} \\
&=&\inf\{s\geqslant 0:\mc{F}_{\mu^0,\omega}(s)\geqslant m+r_n^0+r_n\}+\delta/2=\mc{F}_{\mu^0,\omega}^{-1}(m+r_n^0+r_n)+\delta/2,
\end{eqnarray*}
for all $m\in(0,1-r_n^0-r_n)$. Similarly, we obtain
\begin{equation*}
\mc{B}^{-1}_{\mu^1_n,\omega_n}(1-r_n-2\varepsilon-m)
\leqslant \mc{B}_{\mu^1,\omega}^{-1}(1+r_n^1-2\varepsilon-m)+\delta/2,
\end{equation*}
for all $m\in(r_n^1-2\varepsilon,1-r_n-2\varepsilon)$.
For $n$ large enough, we have
\begin{eqnarray*}
S_{2\varepsilon}(\mu_n^0,\mu_n^1,\omega_n)&:=&
\sup\limits_{m\in[0,1-r_n-2\varepsilon]}\{\mc{F}^{-1}_{\mu^0_n,\omega_n}(m)+\mc{B}^{-1}_{\mu^1_n,\omega_n}(1-r_n-2\varepsilon-m)\}\\
&\leqslant&\sup\limits_{m\in[0,1-r_n-2\varepsilon]}\{\mc{F}_{\mu^0,\omega}^{-1}(m+r_n^0+r_n)+\mc{B}_{\mu^1,\omega}^{-1}(1+r_n^1-2\varepsilon-m)\}+\delta\\
&\leqslant&\sup\limits_{m\in[r_n^0+r_n,1+r_n^0-2\varepsilon]}\{\mc{F}_{\mu^0,\omega}^{-1}(m)+\mc{B}_{\mu^1,\omega}^{-1}(1+r_n^1+r_n^0+r_n-2\varepsilon-m)\}+\delta.
\end{eqnarray*}
Observe that it holds $0\leq r_n^0+r_n$. For $n$ large enough, it also holds $1 + r^0_n-2\varepsilon\leq 1-\varepsilon$, as well as $r_n^1+r_n^0+r_n\leqslant\varepsilon$. Since 
$\mc{B}_{\mu^1,\omega}^{-1}$ is increasing, this implies
\begin{equation*}
S_{2\varepsilon}(\mu_n^0,\mu_n^1,\omega_n)\leq \sup\limits_{m\in[0,1-\varepsilon]}\{\mc{F}_{\mu^0,\omega}^{-1}(m)+\mc{B}_{\mu^1,\omega}^{-1}(1-\varepsilon-m)\}+\delta=S_\varepsilon(\mu^0,\mu^1,\omega)+\delta
\end{equation*}
for $n$ sufficiently large.

\smallskip

We now prove \textbf{Item (ii)}. For all $n\in\mb{N}$, define 
$$\widetilde{V}_n^0:=\{x^0\in\supp(\mu^0_n):\Phi_t^v(B_{f_n}(x^0))\subset\subset\omega
\mbox{~~~for~a~~}t\leqslant t^0_{\omega_n}(x^0)+\delta/2\}$$
and $\widetilde{r}_n^0:=\mu^0((\widetilde{V}_n^0)^c)$. 
Using the same argument as for item (i), we can prove that
\begin{equation}\label{ine t0 (ii)}
t^0_{\omega}(\xi^0)\leqslant t^0_{\omega_n}(x^0)+\delta/2
\end{equation}
for $\pi_n$-almost every $x^0\in\widetilde{V}_n^0$ and $\xi^0\in\supp(\mu^0)\cap (R_n^0)^c$, where $\pi_n$ is a transportation plan realizing $W_\infty(\mu^0_n,\mu^0)$.
 Inequality \eqref{ine t0 (ii)} implies
\begin{eqnarray*}
&&\mc{F}_{\mu^0_n,\omega_n}(t)\leqslant
\mu^0_n(\{x^0\in \widetilde{V}_n^0:t^0_{\omega_n}(x^0)\leqslant t\})+\mu^0((\widetilde{V}_n^0)^c)\\
&&=\pi_n(\{(x^0,\xi^0)\in   \widetilde{V}_n^0\times(\supp(\mu^0)\cap (R_n^0)^c): t^0_{\omega_n}(x^0)\leqslant t\})+\widetilde{r}_n^0\\
&&\leqslant \pi_n(\{(x^0,\xi^0)\in  \widetilde{V}_n^0\times(\supp(\mu^0)\cap (R_n^0)^c): t^0_{\omega}(\xi^0)\leqslant t+\delta/2\})+\widetilde{r}_n^0\\
&&\leqslant \mu^0(\{\xi^0\in \supp(\mu^0)\cap (R_n^0)^c:t^0_{\omega}(\xi^0)\leqslant t+\delta/2\})+\widetilde{r}_n^0\\
&&\leqslant \mu^0(\{\xi^0\in \supp(\mu^0):t^0_{\omega}(\xi^0)\leqslant t+\delta/2\})+\mu^0(R^0_n)+\widetilde{r}_n^0\\
&&= \mc{F}_{\mu^0,\omega}(t+\delta/2)+r_n+\widetilde{r}_n^0,
\end{eqnarray*}
for  all $t\in\mb{R}^+$.
We also have
\begin{equation*}
\mc{B}_{\mu^1_n,\omega_n}(t)\leqslant \mc{B}_{\mu^1,\omega}(t+\delta/2)+r_n+\widetilde{r}_n^1,
\end{equation*}
for all $t\in\mb{R}^+$, where $\widetilde{r}_n^1$ is similarly defined.
We conclude as before.

\smallskip

The last statement holds for $U_{0,n}:=V^0_n$ and  $U_{1,n}:=V^1_n$, 
where $$V^1_n:=\{x^1\in\supp(\mu^1):\Phi_{-t}^v(B_{f_n}(x^1))\subset\subset\omega_n
\mbox{~for~a~}t\leqslant t^1_{\omega}(x^1)+\delta/2\}.$$
%
%\hfill\qed
\end{proof}

\subsection{Proof of Theorem \ref{th opt}}\label{sec:proof th 3}

In this section, we prove Theorem \ref{th opt}. The proof is based on the results obtained in Section \ref{sec:up mass} and \ref{sec: comp}. We highlight that the idea of the proof is to approximate the macroscopic crowds with microscopic ones with a sufficiently large number of agents, find a control steering one microscopic crowd to another, then adapt such control for the macroscopic one. For this reason, this result can be seen as the limit of Proposition \ref{prop:up to mass discret}.
\begin{proof}[Proof of   Theorem \ref{th opt}]
 We first prove \textbf{Item (i)}.
Fix $s>0$ and $\varepsilon>0$ and define the time
$$T:=S(\mu^0,\mu^1,\omega)+s.$$
To prove the theorem, we aim to show that there exists an admissible control $u$  
such that $$W_1(\mu^1,\Phi_T^{v+\mathds{1}_{\omega}u}\#\mu^0)\leqslant \varepsilon.$$

  We assume that the space dimension is $d:=2$, but the reader will see that the proof can be clearly adapted to any space dimension. The proof is divided into three steps. 

\textbf{Step 1:} In this step, we discretize the measures, in four sub-steps.
We first discretize uniformly in space the supports of $\mu^0$ and $\mu^1$.
To simplify the presentation, assume that  $\supp(\mu^0)\subset(0,1)^2$
 and $\supp(\mu^1)\subset(0,1)^2$. 
For all $k:=(k_1,k_2)\in\{0,...,n-1\}^2$, consider the set $C_{n,k}$ defined by
\begin{equation*}
C_{n,k}:=\left[\frac{k_1}{n},\frac{k_{1}+1}{n}\right]\times\left[\frac{k_2}{n},\frac{k_{2}+1}{n}\right].
\end{equation*}

Then define $\mu^l_{nk}:=\mu^l_{|_{C_{n,k}}}$ for $l=0,1$. Remark that the number of indexes $nk$ is exactly $n^2$.

We now consider only the decomposition of $\mu^0$, while the decomposition of $\mu^1$ is similar. To send a measure  to another, these measures need to have the same total mass. With this goal, we perform the second step of our decomposition: on each set $C_{n,k}$ consider a decomposition $$a_0=\frac{k_1}{n}<a_1<\ldots<a_m\leq \frac{k_1+1}n$$ such that it holds $\mu^0_{nk}([a_i,a_{i+1}]\times [\frac{k_2}{n},\frac{k_{2}+1}{n}])=\frac{1}{n^3}$ for all $i=0,\ldots,m-1$, as well as $\mu_{nk}^0([a_m,\frac{k_1+1}n]\times [\frac{k_2}{n},\frac{k_{2}+1}{n}])<\frac{1}{n^3}$. If $\mu^0(C_{n,k})<\frac{1}{n^3}$, we set $m=0$.
One can always perform such a decomposition, since $\mu^0$ is absolutely continuous.
Then define $C^0_{nki}:=[a_i,a_{i+1}]\times [\frac{k_2}{n},\frac{k_{2}+1}{n}]$ for $i=0,\ldots,m-1$ as well as $\mu^0_{nki}:=(\mu^0_{nk})_{|_{C^0_{nki}}}$. Also define $$\bar \mu^0_{nk}:=(\mu^0_{nk})_{|_{\left[a_m,\frac{k_1+1}n\right]\times \left[\frac{k_2}{n},\frac{k_{2}+1}{n}\right]}}.$$
The meaning of such decomposition is the following: each $\mu^0_{nki}$ has mass $\frac{1}{n^3}$ and is localized in $C^0_{nki}$, while $\bar \mu^0_{nk}$ has a mass strictly smaller than $\frac{1}{n^3}$ and is localized in $C^0_{n,k}$. Remark that the number of indexes $i$ depends on $n,k$. Nevertheless, the total mass not contained in the sets $C^0_{nki}$ is smaller than $\frac1n$, hence the total number of indexes $nki$ is between $n^3-n^2$ and $n^3$.

We now consider the third step of the decomposition: each $\mu^0_{nki}$ is localized in the set $C^0_{nki}:=[a_i,a_{i+1}]\times [\frac{k_2}{n},\frac{k_{2}+1}{n}]$. For each index $nki$, define a decomposition $$a_{i0}=\frac{k_2}n<a_{i1}<\ldots<a_{im}\leq  \frac{k_{2}+1}{n}$$ such that it holds $\mu^0_{nki}([a_i,a_{i+1}]\times [a_{ij},a_{i(j+1)}])=\frac{1}{n^4}$ for each $j=0,\ldots,m-1$. Then define $C_{nkij}^0:=[a_i,a_{i+1}]\times [a_{ij},a_{i(j+1)}]$ for each $j=0,\ldots,m-1$, as well as $\mu^0_{nkij}:=(\mu^0_{nki})_{|_{C^0_{nkij}}}$. Remark that the number of indexes $j$ depends on $nki$, but that the total number of indexes $nkij$ is between $n^4-n^3$ and $n^4$.

If the dimension of the space is larger than 2, one simply needs to decompose the $\mu_{nkij}$ with respect to the subsequent dimension with blocks of mass $\frac{1}{n^5}$, and so on, up to cover all dimensions.

The second and third step of the decomposition are represented in Figure \ref{fig: mesh}.

 For more details on such discretization, we refer to  \cite[Prop. 3.1]{DMR17}.

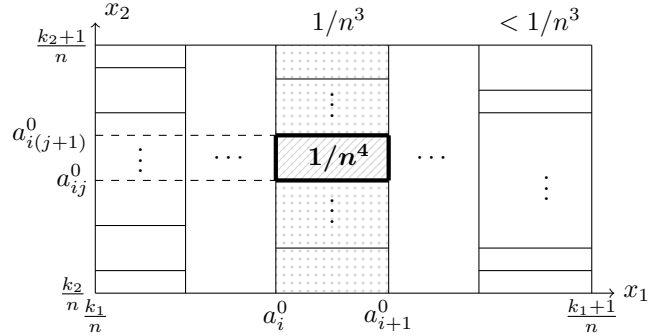
\begin{figure}[ht]
\begin{center}
\begin{tikzpicture}[scale=3]
\draw[->] (0,0) -- (0,1.2);
\draw[->] (0,0) -- (2.3,0);
\draw[-] (0,1.1) -- (2.2,1.1);
\path (0.1,1.25) node {$x_2$};
\path (2.4,0) node {$x_1$};
\draw[-] (0.4,0) -- (0.4,1.1);
\draw[-] (0,0.1) -- (0.4,0.1);
\draw[-] (0,0.3) -- (0.4,0.3);
\path (0.2,0.625) node {$\vdots$};
\draw[-] (0,0.8) -- (0.4,0.8);
\draw[-] (0,1) -- (0.4,1);
\path (0.6,0.6) node {$\cdots$};
\draw[-] (0.8,0) -- (0.8,1.1);
\draw[-] (0.8,0.2) -- (1.3,0.2);
\path (1.05,0.4) node {$\vdots$};
\draw[-,ultra thick] (0.8,0.5) -- (1.3,0.5);
\draw[-,ultra thick] (0.8,0.7) -- (1.3,0.7);
\path (1.05,0.87) node {$\vdots$};
\draw[-,ultra thick] (0.8,0.5) -- (0.8,0.7);
\draw[-,ultra thick] (1.3,0.5) -- (1.3,0.7);
\fill [opacity=0.4,pattern=north east lines] (0.8,0.5) -- (0.8,0.7) -- (1.3,0.7) -- (1.3,0.5);
\draw[dashed] (0,0.5) -- (0.8,0.5);
\draw[dashed] (0,0.7) -- (0.8,0.7);
\fill [opacity=0.4,pattern=dots] (0.8,0) -- (0.8,0.5) -- (1.3,0.5) -- (1.3,0);
\fill [opacity=0.4,pattern=dots] (0.8,0.7) -- (0.8,1.1) -- (1.3,1.1) -- (1.3,0.7);
\path (0.8,-0.1) node {$a_i^0$};
\path (1.3,-0.1) node {$a_{i+1}^0$};
\path (1.08,1.2) node {$1/n^3$};
\path (1.97,1.2) node {$<1/n^3$};
\path (-0.1,0.5) node {$a_{ij}^0$};
\path (-0.2,0.7) node {$a_{i(j+1)}^0$};
\path (1.08,0.6) node {$\boldsymbol{1/n^4}$};
\draw[-] (0.8,.95) -- (1.3,.95);
\draw[-] (1.3,0) -- (1.3,1.1);
\path (1.5,0.6) node {$\cdots$};
\draw[-] (1.7,0) -- (1.7,1.1);
\draw[-] (1.7,0.1) -- (2.2,0.1);
\draw[-] (1.7,0.2) -- (2.2,0.2);
\path (2,0.5) node {$\vdots$};
\draw[-] (1.7,0.8) -- (2.2,0.8);
\draw[-] (1.7,0.9) -- (2.2,0.9);
\draw[-] (2.2,0) -- (2.2,1.1);
\path (0,-0.1) node {$\frac{k_1}{n}$};
\path (2.2,-0.1) node {$\frac{k_1+1}{n}$};
\path (-0.1,0) node {$\frac{k_2}{n}$};
\path (-0.15,1.1) node {$\frac{k_2+1}{n}$};
\end{tikzpicture}
\caption{Example of a partition of $C_{n,k}$  in cells such as $C_{nkij}^0$ (hashed).}
\label{fig: mesh}
\end{center}
\end{figure}

We finally consider the fourth step of our decomposition: in each cell $C_{nkij}^0=[a_i,a_{i+1}]\times [a_{ij},a_{i(j+1)}]$, it is defined an absolutely continuous measure $\mu^0_{nkij}$ with mass $\frac{1}{n^4}$. Thus, it exists a constant $\epsilon>0$ such that the set 
\begin{equation}\label{e-B0}
B^0_{nkij}:=[a_i+\epsilon,a_{i+1}-\epsilon]\times [a_{ij}+\epsilon,a_{i(j+1)}-\epsilon]
\end{equation} satisfies $\mu^0_{nkij}(B^0_{nkij})=\frac{1}{n^4}-\frac{1}{n^6}$. See Figure \ref{fig:cell} for a graphical description of this decomposition. For more details of such  construction, we also refer to \cite[Prop. 3.1]{DMR17}.

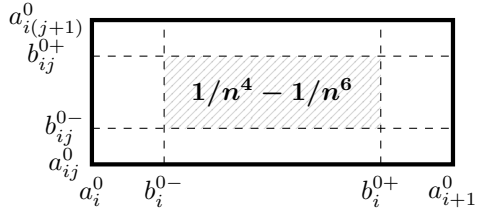
\begin{figure}[ht]
\begin{center}
\begin{tikzpicture}[scale=2.4]
\draw[color=black,ultra thick] (0 ,0) -- (0 ,0.8) -- (2,0.8) -- (2 ,0) -- cycle;
\draw[dashed] (0.4,0) -- (0.4,0.8);
\draw[dashed] (1.6,0) -- (1.6,0.8);
\draw[dashed] (0,0.2) -- (2,0.2);
\draw[dashed] (0,0.6) -- (2,0.6);
\path (0.0,-0.15) node {$a_i^0$};
\path (2,-0.15) node {$a_{i+1}^0$};
\path (-0.15,0) node {$a_{ij}^0$};
\path (-0.25,0.8) node {$a_{i(j+1)}^0$};
\path (0.4,-0.15) node {$b_i^{0-}$};
\path (1.6,-0.15) node {$b_{i}^{0+}$};
\path (-0.15,0.2) node {$b_{ij}^{0-}$};
\path (-0.25,0.6) node {$b_{ij}^{0+}$};
\fill [opacity=0.4,pattern=north east lines] (0.4,0.2) -- (0.4,0.6) -- (1.6,0.6) -- (1.6,0.2);
\path (1,0.4) node {$\boldsymbol{1/n^4-1/n^6}$};
\end{tikzpicture}
\caption{Example of  cells  $B^0_{nkij}$ (hashed).}\label{fig:cell}
\end{center}\end{figure}

We finally define $\tilde \mu^0_{nkij}:=(\mu^0_{nkij})_{|_{B^0_{nkij}}}$ and $\bar \mu^0_{nkij}:=(\mu^0_{nkij})_{|_{C^0_{nkij}\setminus B^0_{nkij}}}$.

\noindent We discretize similarly the measure $\mu^1$ on sets $C^1_{nki}, C^1_{nkij}, B^1_{nkij}$, defining $\bar \mu^1_{nk}$, $\bar \mu^1_{nkij}$, $\tilde \mu^1_{nkij}$.

\textbf{Step 2:}
In this step, we prove exact controllability of microscopic approximations $\hat \mu^0_n, \hat \mu^1_n$ of $\mu^0,\mu^1$. In a first step, we define such sequence of approximations $\hat \mu^0_n, \hat \mu^1_n$ of $\mu^0,\mu^1$ satisfying hypotheses of Proposition \ref{prop cont T0}.

\textbf{Step 2.1:} We aim to define a sequence of microscopic approximations $\hat \mu^0_n, \hat \mu^1_n$ of $\mu^0,\mu^1$, that satisfy hypotheses of Proposition \ref{prop cont T0}. We first observe the following: the number of $nkij$ indexes are between $n^4-n^3$ and $n^4$ both for $\mu^0$ and $\mu^1$. We then choose exactly a set $I^0_n$ of $n^4-n^3$ indexes $nkij$ for $\mu^0$, as well as a set $I^1_n$ of $n^4-n^3$ indexes $nkij$ for $\mu^1$.

We then define the following measures:
\begin{eqnarray}
&&\tilde \mu^0_n:=\sum_{nkij\in I^0_n} \tilde \mu^0_{nkij},\qquad \tilde \mu^1_n:=\sum_{nkij\in I^1_n} \tilde \mu^1_{nkij},\label{e-mut}\\
&&\bar\mu^0_n:=\mu^0-\tilde\mu^0_n=\sum_{nk}\bar\mu^0_{nk}+\sum_{nki}\bar\mu^0_{nki}+\sum_{nkij}\bar\mu^0_{nkij}+\sum_{nkij\not\in I^0_n}\tilde\mu^0_{nkij},\nonumber
\end{eqnarray}
and similarly $\tilde\mu^1_n$, $\bar\mu^1_n=\mu^1-\tilde\mu^1_n$. The idea of the notation is the following: the measure $\mu^0$ is decomposed into a ``relevant'' part $\tilde \mu^0_n$ and a negligible one $\bar \mu^0_n$. Such decomposition only depends on the parameter $n$ of the grid, as well as the choice of $n^4-n^3$ indexes defining the set $I^0_n$. The same holds for the decomposition of $\mu^1$.

We are ready to define the microscopic approximation $\hat \mu^0_n$. 
%Given $\tilde \mu^0_n:=\sum_{nkij\in I^0_n} \tilde \mu^0_{nkij}$, 
Observe that each $\tilde \mu^0_{nkij}$ is supported in the set $B^0_{nkij}$ defined in \eqref{e-B0} and has mass $\frac{1}{n^4}-\frac{1}{n^6}$. For each $nkij\in I^0_n$, choose a point $x^0_{nkij}$ belonging to the support of $\mu^0_{nkij}$. Then define 
\begin{equation*}% \label{e-muhat}
\hat \mu^0_n:=\sum_{nkij\in I^0_n} \left(\frac{1}{n^4}-\frac{1}{n^6}\right) \delta_{x^0_{nkij}}.
\end{equation*}
Repeat the same construction for $\mu^1$, then defining $\hat \mu^1_n$.

\textbf{Step 2.2:} We now prove that  sequences $\hat \mu^0_n, \hat \mu^1_n$ defined above, together with $\mu^0,\mu^1$, satisfy the hypotheses of Proposition \ref{prop cont T0}. Define the sequence 
$$f_n:=e^{LT}\sqrt{2}/n,$$
 where $L$ is the Lipschitz constant of the vector field $v$. Then observe that $x^0_{nkij}\in\supp(\mu^0)$ for all $nkij$, hence the fact that $\mu^0$ satisfies the Geometric Condition \ref{cond1} implies that $\hat \mu^0_n$ satisfies it too.

For each $n$, define the set $R^0_{n}$ as the support of $\bar \mu^0_n$. This implies that 
\begin{eqnarray}
r_n:=\mu^0(R^0_n)=\bar\mu^0_n(\mathbb{R}^d)=1-\tilde\mu^0_n(\mathbb{R}^d)=1-\left(n^4-n^3\right)\left(\frac{1}{n^4}-\frac{1}{n^6}\right)\to 0.\label{e-rn}
\end{eqnarray}
Moreover, by construction it holds $\hat \mu^0_n(\mathbb{R}^d)=\tilde \mu^0_n(\mathbb{R}^d)=1-r_n$. Finally, observe that each set $B^0_{nkij}$ has diameter strictly smaller than the diameter of $C^0_{nkij}$, hence smaller than\footnote{For the space $\mathbb{R}^d$, it is clearly sufficient to replace $\sqrt{2}$ with $\sqrt{d}$.} $\frac{\sqrt{2}}{n}<f_n$. Moreover, it is clear that it holds $(\mu^0)_{|_{(R^0_n)^c}}=\tilde\mu^0_n$ by construction. Then, one can estimate $W_\infty((\mu^0)_{|_{(R^0_n)^c}},\hat \mu^0_n)$ by decomposing them in each set $B^0_{nkij}$, \textit{i.e.}
\begin{eqnarray*}
&&W_\infty((\mu^0)_{|_{(R^0_n)^c}},\hat \mu^0_n)\leq \max_{nkij\in I^0_n} W_\infty(\tilde\mu^0_{nkij},\hat\mu^0_{nkij})\leq\\
 &&\max_{nkij\in I^0_n} \mathrm{diam}(\supp(\tilde\mu^0_{nkij})\cup\supp(\hat\mu^0_{nkij}))\leq \max_{nkij\in I^0_n} \mathrm{diam}(B^0_{nkij})<f_n.
\end{eqnarray*}
Repeat the same procedure for $\mu^1$, defining $R^1_n$ and proving the same estimates with $f_n,r_n$. Then, all conditions of \eqref{e-Rn} are satisfied. Summing up, all hypotheses of Proposition \ref{prop cont T0} are satisfied.

Since $v$ is uniformly bounded, there exists $R>0$ independent on $n$ and $\varepsilon$, such that 
\begin{equation}\label{def R}
%[0,1]^2
\bigcup_{t\in[0,T]} \Phi_t^v(\omega)\subset\subset B_R(0).
\end{equation}
We now estimate the value of $S_{\varepsilon/4R}(\hat \mu^0_n,\hat \mu^1_n,\omega_n)$, where $\omega_n$ is defined in \eqref{omega n}. Observe that it clearly holds  
\begin{equation*}
S_{\varepsilon/8R}(\mu^0,\mu^1,\omega)\leqslant S(\mu^0,\mu^1,\omega).
\end{equation*}
Apply Proposition \ref{prop cont T0} for $\delta:=s/2$, that gives the following estimate:
\begin{equation*}S_{\varepsilon/4R}(\hat \mu^0_n,\hat \mu^1_n,\omega_n)+\dfrac{s}{2}
\leqslant S_{\varepsilon/8R}(\mu^0,\mu^1,\omega)+s 
\leqslant S(\mu^0,\mu^1,\omega)+s =T,
\end{equation*}
for $n$ large enough.

\textbf{Step 2.3:} We now prove existence of a control exactly steering $\hat \mu^0_n$ to $\hat \mu^1_n$ up to a mass $\varepsilon/4R$. 
Using Condition \eqref{eq:cond omega}, $\omega_n$ is connected  for $n$ large enough, 
%Observe that $\omega$ convex implies $\omega_n$ convex. 
%\textcolor{red}{attention a verifier}
Using Proposition \ref{prop:up to mass discret}, it holds
\begin{equation*}
T_{e,\varepsilon/4R}(\hat \mu^0_n,\hat \mu^1_n,\omega_n)=S_{\varepsilon/4R}(\hat \mu^0_n,\hat \mu^1_n,\omega_n)<T.
\end{equation*}
Then there exists $I^l_{n,\varepsilon}\subset I_n^0$ ($l=0,1$) 
such that 
\begin{equation}\label{Iln}
|I^l_{n,\varepsilon}|=(1-\varepsilon/4R)(n^4-n^3)
\end{equation}
 (assumed integer for simplicity), 
%Then, apply Proposition \ref{prop: dim finie} and observe that there exists 
a bijection among indexes $\sigma_n:I^0_{n,\varepsilon}\rightarrow I^1_{n,\varepsilon}$ 
and a control $\mathds{1}_{\omega_n}\hat u_n$ satisfying the Carath\'eodory condition
such that
\begin{equation*}
\Phi_T^{v+\mathds{1}_{\omega}\hat u_n}(x_{nkij}^0)=x^1_{\sigma_n(nkij)}
\end{equation*}
for all $nkij\in I^0_{n,\varepsilon}$.
We define
\begin{equation*}
\tilde \mu^l_{n,\varepsilon}:=\sum_{nkij\in I^l_{n,\varepsilon}} \tilde \mu^l_{nkij},\hspace{5mm}
\bar\mu^l_{n,\varepsilon}:=\mu^l-\tilde\mu^l_{n,\varepsilon},\hspace{5mm}
\hat \mu^l_{n,\varepsilon}:=\sum_{nkij\in I^l_{n,\varepsilon}} \left(\frac{1}{n^4}-\frac{1}{n^6}\right) \delta_{x^l_{nkij}}.
\end{equation*}
 For simplicity of notation, for each $n$ we assume to re-arrange indexes $nkij$ both for $\hat\mu^0_{n,\varepsilon}$ and $\hat\mu^1_{n,\varepsilon}$, so that $I^0_{n,\varepsilon}=I^1_{n,\varepsilon}$ and the bijection $\sigma_n$ is the identity. We will then write from now on \begin{equation}\label{e-unhat}
\Phi_T^{v+\mathds{1}_{\omega}\hat u_n}(x_{nkij}^0)=x^1_{nkij}.
\end{equation}
We also denote by $\hat x_{nkij}(t)$ the corresponding trajectory, steering $x^0_{nkij}$ to $x^1_{nkij}$ in time $T$. Finally, we denote 
$$t^0_{nkij}=t^0(x^0_{nkij},\omega_n),\qquad t^1_{nkij}=t^1(x^1_{nkij},\omega_n).$$

We also make the following observation: the trajectory steering $x^0_{nkij}$ to $x^1_{nkij}$ is known to enter $\omega_n$ at time $t^0_{nkij}$ and to exit it at time $T-t^1_{nkij}$. 
Due to the construction of the control in the proof of Proposition \ref{prop: dim finie}, all trajectories $\hat  x_{nkij}(t)$ are  contained in $\omega_n$ for all times $t\in (t^0_{nkij},T-t^1_{nkij})$.

\textbf{Step 3:} We now prove approximate controllability from $\mu^0$ to $\mu^1$ in time $T=S(\mu^0,\mu^1,\omega) \allowbreak +s$. The idea is to write a control approximately steering $\mu^0$ to $\mu^1$, based on the controls $\hat u_n$ defined in \eqref{e-unhat} exactly steering $\hat\mu^0_{n,\varepsilon}$ to $\hat\mu^1_{n,\varepsilon}$ in time $T$.

Fix a given $n$. For each $nkij\in I^0_{n,\varepsilon}$ apply the following procedure:
\begin{enumerate}
\item the set $B^0_{nkij}$ evolves via the uncontrolled flow $\Phi^v_t$ in the time interval $[0,t^0_{nkij}]$; at this final time, the evolved set $B_{nkij}(t^0_{nkij}):=\Phi^v_{t^0_{nkij}}(B^0_{nkij})$ is completely contained in $\omega$.
\item we then apply a control $w_{nkij}$ concentrating the mass around $\hat  x_{nkij} (t)$ defined in \eqref{e-unhat}  for the time interval $[t^0_{nkij},T-t^1_{nkij}]$. Such mass is then moved to the corresponding evolved set $B_{nkij}(T-t^1_{nkij})$ around $\hat x (T-t^1_{nkij})$;
\item we finally let such set evolve via the uncontrolled flow $\Phi^v_t$ up to time $T$; the  set $B_{nkij}(T)$ is centered around $\hat  x(T)=x^1_{nkij}$. 
\end{enumerate}

In the following, we will prove that such strategy provides the desired result. In the first step, we will precisely define the strategy and prove estimates about the evolution of sets. In the second and final step, we will provide estimates about the Wasserstein distance, to prove approximate controllability.

\textbf{Step 3.1:} We now define a control approximately steering $\mu^0$ to $\mu^1$. We first define the control and the desired evolution for each index $nkij\in I^0_{n,\varepsilon}$. We recall that the control $\hat u_{n}$ given in \eqref{e-unhat} defines the trajectory $\hat  x_{nkij}(t)$ steering $x^0_{nkij}$ to $x^1_{nkij}$. We then define an adapted control $\tilde u_n(t,x)$ that concentrates mass around such trajectory. We choose
\begin{equation*}
\tilde u_{nkij}(t,x)=\hat u_n(t,\hat  x_{nkij}(t))+C_n(\hat  x_{nkij}(t)-x),%
\end{equation*}
where $C_n$ is a positive constant that will be chosen later. It is clearly a linear feedback stabilizing the system around the trajectory $\hat  x_{nkij}$. We also define a corresponding vector field 
\begin{equation}\label{e-wn}
w_{nkij}(t,x):=v(\hat x_{nkij}(t))+\tilde u_{nkij}(t,x).
\end{equation} Observe that in this formula the vector field $v$ is evaluated at $\hat  x_{nkij}(t)$ only, and not at the point $x$. We then define the evolution of the corresponding evolved cell $B_{nkij}(t)$ as follows:
\begin{equation}
B_{nkij}(t)=\label{e-B}\begin{cases}
\Phi^v_t(B^0_{nkij})&\mbox{~~if~} t\in[0,t^0_{nkij}],\\
\Phi^{w_{nkij}}_t(B_{nkij}(t^0_{nkij}))&\mbox{~~if~} t\in[t^0_{nkij},T-t^1_{nkij}],\\
\Phi^v_t(B_{nkij}(T-t^1_{nkij}))&\mbox{~~if~} t\in[T-t^1_{nkij},T].
\end{cases}
\end{equation}
We also give a simple estimation of the diameter of the cell $B_{nkij}(t)$ as follows:
\begin{equation}\label{e-diams}
\mbox{diam}(B_{nkij}(t))<\begin{cases}
e^{Lt} \frac{\sqrt2}{n}&\mbox{~~if~} t\in[0,t^0_{nkij}],\\
e^{-C_n(t-t^0_{nkij})}e^{Lt^0_{nkij}}\frac{\sqrt2}{n}&\mbox{~~if~} t\in[t^0_{nkij},T-t^1_{nkij}],\\
e^{-C_n(t^1_{nkij}-t^0_{nkij})}e^{LT}\frac{\sqrt2}{n}&\mbox{~~if~} t\in[T-t^1_{nkij},T],
\end{cases}
\end{equation}
where $L$ is the Lipschitz constant of the vector field $v$. The estimate in the first interval is a classical application of Gronwall lemma, recalling that $\mbox {diam}(B^0_{nkij})<\frac{\sqrt2}{n}$. For the estimate in the second interval, we use the fact that the term $v(\hat  x_{nkij}(t))+\hat u_n(t,\hat  x_{nkij}(t))$ in $w_{nkij}$ is constant with respect to the space variable, while for the term $C_n(\hat  x_{nkij}(t)-x)$ we again apply the Gronwall lemma. In the third interval, we again apply the Gronwall lemma and estimate $t\leq T$.

We now prove that $B_{nkij}(t)$ is contained in $\omega$ for all $t\in [t^0_{nkij},T-t^1_{nkij}]$. Indeed, it is sufficient to recall from Step 2.3 that one can choose the trajectory $\hat  x_{nkij}(t)$ to be contained in $\omega_n$ for all $t\in [t^0_{nkij},T-t^1_{nkij}]$. Now observe that for all $t\in [t^0_{nkij},T-t^1_{nkij}]$ it holds $\hat  x_{nkij}(t)\in B_{nkij}(t)\cap \omega_n$, that $\mbox{diam}(B_{nkij}(t))<e^{LT}\frac{\sqrt2}{n}$ and recall the definition \eqref{omega n} of $\omega_n$ with $f_n=e^{LT}\frac{\sqrt2}{n}$. Then, $B_{nkij}(t)$ is contained in $\omega$.

We now prove that, given a fixed $n$ and choosing a sufficiently large $C_n$, for each pair of distinct indexes $nkij,nk'i'j'\in I^0_{n,\varepsilon}$ the sets $B_{nkij}(t),B_{nk'i'j'}(t)$ are disjoint. Assume for simplicity that it holds $t^0_{nkij}\leq t^0_{nk'i'j'}$. First observe that at time $t=0$ the sets $B^0_{nkij},B^0_{nk'i'j'}$ are disjoint by construction. We then separate the time interval $[0,T]$ in five intervals:
\begin{enumerate}
\item Interval $[0, t^0_{nkij}]$: both sets are displaced via the flow defined by the Lipschitz vector field $v$, that keeps them disjoint.
\item Interval $[t^0_{nkij},t^0_{nk'i'j'}]$: observe that the set $B_{nk'i'j'}(t)$ evolves via the flow defined by the Lipschitz vector field $v$, hence its diameter can grow of a factor $e^{L(t^0_{nk'i'j'}-t^0_{nkij})}$. For $t\leq t^1_{nkij}$, by choosing a constant $$C_n>2e^{LT}>e^{L(t^0_{nk'i'j'}-t^0_{nkij})}$$ sufficiently large, one can reduce the diameter of the set $B_{nkij}(t)$ to have it disjoint with respect to $B_{nk'i'j'}$. For $t>t^1_{nkij}$, it is sufficient to observe that both sets are displaced by the Lipschitz vector field $v$, then apply Case 1.
\item Interval $[t^0_{nk'i'j},\min(t^1_{nkij},t^1_{nk'i'j'})]$, when non-empty: recall that the  trajectories $\hat x_{nkij}(t)$, $\hat x_{nk'i'j'}(t)$ are disjoint at each time, since the vector field $\widehat{u}_n$ defined in Step 2.3 satisfies the Carath\'eodory condition. Since in this time interval one can act on both sets with the controls $w_{nkij}$ and $w_{nk'i'j'}$, by choosing $C_n$ sufficiently large one can concentrate the sets in two sufficiently small neighborhoods around disjoint trajectories, then having disjoint sets.
\item Interval $[\min(t^1_{nkij},t^1_{nk'i'j'}),\max(t^1_{nkij},t^1_{nk'i'j'})]$ is equivalent to Case 2.
\item Interval $[\max(t^1_{nkij},t^1_{nk'i'j'}), T]$ is equivalent to Case 1.
\end{enumerate}

Then, given a fixed $n$, for each pair of distinct indexes $nkij$, $nk'i'j'$ one can choose a sufficiently large $C_n$ ensuring that the sets $B_{nkij}(t),B_{nk'i'j'}(t)$ are disjoint. By observing that the number of pairs are finite, one can choose a sufficiently large $C_n$ for which the property holds for all pairs of indexes.

We are now ready to define the control $u_n$ approximately steering $\mu^0$ to $\mu^1$ in time $T=S(\mu^0,\mu^1,\omega)+s$. For each time $t\in[0,T]$, first define
\begin{equation}\label{e-un}
u_n(t,x):=w_{nkij}(t,x)-v(x)\mbox{~~~ if~~}t\in[t^0_{nkij},T-t^1_{nkij}], x\in B_{nkij}(t),
\end{equation}
where $w_{nkij}$ is defined in \eqref{e-wn}. It is clear that such definition is well-posed, since the $B_{nkij}(t)$ are disjoint at each time. We now complete the definition of $u_n$ as follows:
\begin{equation}\label{e-un1}
u_n(t,x):=\begin{cases}
w_{nkij}(t,x)-v(x)&\mbox{for all~~}t\in[t^0_{nkij},T-t^1_{nkij}], x\in B_{nkij}(t),\\
\mbox{\begin{tabular}{l} Lip. spline\\ in the $x$ variable\end{tabular}}& \mbox{if~~}t\in [t^0_{nkij},T-t^1_{nkij}]\mbox{~for some~}nkij\in I^0_{n,\varepsilon},\\
0 &\mbox{for all~~} x\in \omega^c,\\
0 & \mbox{for all other~~~}t\in[0,T].
\end{cases}
\end{equation}
Observe that $t\in[t^0_{nkij},T-t^1_{nkij}]$ and $x\in B_{nkij}(t)$ imply that $B_{nkij}\in \omega$, \textit{i.e.} that the control is allowed to have non-zero value. %Moreover, convexity of $\omega$ allow us to choose a Lipschitz spline that is non-zero in $\omega$ only. 
Also observe that $u_n=0$ on the boundary of $\omega$, that implies that $\mathds{1}_{\omega}u_n$ is also Lipschitz for each time. Regularity with respect to time is ensured by the fact that discontinuities are allowed on a finite number of times $t^l_{nkij}$ only. Thus $\mathds{1}_{\omega}u_n$ satisfies the Carath\'eodory condition, hence it is an admissible control.

We finally observe the following key property: when applying the control $\mathds{1}_{\omega}u_n$ to the System \eqref{eq:transport}, the dynamics of the sets $B^0_{nkij}$ satisfies \eqref{e-B}. We will see in the next step that such property is the key to ensure approximate controllability of $\mu^0$ to $\mu^1$ at time $T$.

\textbf{Step 3.2:} We now prove that the control $u_n$ defined in \eqref{e-un}-\eqref{e-un1} provides approximate controllability of $\mu^0$ to $\mu^1$ at time $T$.  Recall that the solution $\mu_n(t)$ to System \eqref{eq:transport} with control $\mathds{1}_{\omega}u_n$ starting from $\mu^0$ is $$\mu_n(t):=\Phi^{v+\mathds{1}_{\omega}u_n}_t\#\mu^0.$$ We aim to prove that the distance $W_1(\mu_n(T),\mu^1)$ is less than $\varepsilon$ 
for $n$ large enough.

Recall the decomposition $\mu^0=\tilde\mu^0_{n,\varepsilon}+\bar\mu^0_{n,\varepsilon}$ introduced in \eqref{e-mut}, and similarly for $\mu^1$. One can then write 
$\mu_n(T)=\Phi^{v+\mathds{1}_{\omega}u_n}_T\#\tilde\mu^0_{n,\varepsilon}+\Phi^{v+\mathds{1}_{\omega}u_n}_T\#\bar\mu^0_{n,\varepsilon}$. We estimate
\begin{equation}\label{e-due1}
W_1(\mu_n(T),\mu^1)\leq W_1(\Phi^{v+\mathds{1}_{\omega}u_n}_T\#\tilde\mu^0_{n,\varepsilon},\tilde\mu^1_{n,\varepsilon})+W_1(\Phi^{v+\mathds{1}_{\omega}u_n}_T\#\bar\mu^0_{n,\varepsilon},\bar\mu^1_{n,\varepsilon}).
\end{equation}
For the first term, recall the decomposition $\tilde \mu^0_{n,\varepsilon}=\sum\limits_{nkij\in I^0_{n,\varepsilon}} \tilde \mu^0_{nkij}$ with the supports $\supp(\tilde \mu^0_{nkij})\allowbreak\subset B^0_{nkij}$ by construction. Then, it holds
\begin{equation}\label{e-due2}
\Phi^{v+\mathds{1}_{\omega}u_n}_T\#\tilde\mu^0_{n,\varepsilon}=\sum_{nkij\in I^0_{n,\varepsilon}} \Phi^{v+\mathds{1}_{\omega}u_n}_T\# \tilde \mu^0_{nkij}
\end{equation} and $\supp(\Phi^{v+\mathds{1}_{\omega}u_n}_T \# \tilde \mu^0_{nkij})\subset B_{nkij}(T)$. For each $nkij\in I^0_{n,\varepsilon}$, we provide the estimate 
\begin{eqnarray}
&&W_1(\Phi^{v+\mathds{1}_{\omega}u_n}_T \#\tilde \mu^0_{nkij},\tilde\mu^1_{nkij})\label{e-tri}
\leq W_1\left(\Phi^{v+\mathds{1}_{\omega}u_n}_T \#\tilde \mu^0_{nkij},\left(\frac{1}{n^4}-\frac{1}{n^6}\right)\delta_{x^1_{nkij}}\right)\\
&&\hspace*{5cm}+W_1\left(\left(\frac{1}{n^4}-\frac{1}{n^6}\right)\delta_{x^1_{nkij}},\tilde\mu^1_{nkij}\right)\nonumber
\\ \notag
&&\leq \left(\frac{1}{n^4}-\frac{1}{n^6}\right) \mbox{diam}(B_{nkij}(T))
+\left(\frac{1}{n^4}-\frac{1}{n^6}\right)\mbox{diam}(B^1_{nkij})
\\
&&< \left(\frac{1}{n^4}-\frac{1}{n^6}\right) (e^{LT}+1)\frac{\sqrt{2}}{n},\label{e-due3}
\end{eqnarray}
 with the following observation: the point $\hat x_{nkij}(T)=x^1_{nkij}$ belongs both to $B_{nkij}(T)$ by construction of such set in \eqref{e-B}, and to $B^1_{nkij}$ by construction of such cell in Step 1. Moreover, the mass of both $\tilde \mu^0_{nkij}$ and $\tilde \mu^1_{nkij}$ is $\frac{1}{n^4}-\frac{1}{n^6}$, then one can apply the triangular inequality \eqref{e-tri}, then estimate each term by observing that rays in a transference plan sending $\Phi^{v+\mathds{1}_{\omega}u_n}_T \#\tilde \mu^0_{nkij}$ to $\left(\frac{1}{n^4}-\frac{1}{n^6}\right)\delta_{x^1_{nkij}}$ have length smaller than $\mbox{diam}(B_{nkij}(T))$, and similarly for the second term. Estimates of the diameters are consequences of  \eqref{e-diams} and of the construction of $B^1_{nkij}$.

For the second term of \eqref{e-due1}, first recall that it holds
$$r_{n,\varepsilon}
:=\bar\mu^1_{n,\varepsilon}(\mb{R}^d)=1-(1-\varepsilon/4R)(n^4-n^3)(1/n^4-1/n^6)\underset{n\to \infty}{\longrightarrow} \varepsilon/4R$$ 
(see \eqref{e-rn} and \eqref{Iln}).
 Take now any point $x^0\in \supp(\mu^0)$ and study the trajectory given by the flow $\Phi^{v+\mathds{1}_{\omega}u_n}_t(x^0)$. Consider the set of times $t\in[0,T]$ for which $\Phi^{v+\mathds{1}_{\omega}u_n}_t(x^0)\in \omega$, that is nonempty due to the Geometric Condition \ref{cond1}. Take $\bar t$ the supremum of such times and observe that it holds $$\bar x:=\Phi^{v+\mathds{1}_{\omega}u_n}_{\bar t}(x^0)\in \bar\omega.$$ 
 Moreover, the control does not act on the time interval $[\bar t,T]$, \textit{i.e.} $\Phi^{v+\mathds{1}_{\omega}u_n}_t(x^0)= \Phi^{v}_t(x^0)$. 
 %One can then write 
%\begin{equation*}
%d(\omega,\Phi^{v+\mathds{1}_{\omega}u_n}_T(x^0))\leq d(\bar x,\Phi^{v}_{T-\bar t}(\bar x))\leq \|v\|_{C^0} (T-\bar t)\leq \|v\|_{C^0} T.
%\end{equation*}
%Observe that such last estimate does not depend on the initial point $x^0$. Then, the whole support of $\mu_n(T)$ is contained in a neighbourhood of $\omega$ of size  $\|v\|_{C^0} T$. 
%Observe that such set is bounded, independent on the approximation parameter $n$. 
%Also recall that $\supp(\mu^1)$ is compact, clearly independent on $n$. 
%Then, there exists a radius $R$ such that $\supp(\mu_n(T))\cup\supp(\mu^1)\subset B_R(0)$ for all $n$. 
As a consequence, using \eqref{def R}, for $n$ large enough, it holds
\begin{equation}
W_1(\Phi^{v+\mathds{1}_{\omega}u_n}_T\#\bar\mu^0_{n,\varepsilon},\bar\mu^1_{n,\varepsilon})\leq 2R \bar\mu^1_{n,\varepsilon}(\mb{R}^d)= 2Rr_{n,\varepsilon}.\label{e-due4}
\end{equation}

Summing up, from \eqref{e-due1}-\eqref{e-due2}-\eqref{e-due3}-\eqref{e-due4} it holds
\begin{equation*}
W_1(\mu_n(T),\mu^1)\leq (n^4-n^3)\left(\frac{1}{n^4}-\frac{1}{n^6}\right) (e^{LT}+1)\frac{\sqrt{2}}{n}+2Rr_{n,\varepsilon}\to \varepsilon/2.
\end{equation*}
For $n$ large enough, $W_1(\mu_n(T),\mu^1)$ is then smaller than $\varepsilon$.
%We remark that $R$ can be chosen independent on $n$ and $\varepsilon$. 
%Hence, for each $\tilde\varepsilon:=\varepsilon/2R>0$, 
%it exists $n$ such that $W_1(\mu_n(T),\mu^1)\leq 2R\varepsilon=\tilde\varepsilon$. 
Thus System \eqref{eq ODE} is approximatively controllable 
for each $T> S(\mu^0,\mu^1,\omega)$.
%Such estimate holds for any $T=S(\mu^0,\mu^1,\omega)+s$ with $s>0$, thus $S(\mu^0,\mu^1,\omega)=S(\mu^0,\mu^1)$ satisfies Item (i) of Theorem \ref{th opt}.

\bigskip

We now prove  \textbf{Item (ii)} of Theorem \ref{th opt}. To prove it, we first need to recall some results about evoluted sets and the Lebesgue measure of their boundaries, which proofs directly follow from the definitions and are then omitted.
\begin{lemma} \label{l-omegat} Let $\omega$ be an open set and $v$ a Lipschitz vector field. For $t> 0$, define the {\bf evoluted set}
\begin{equation*}
\omega^t:=\cup_{\tau\in (0,t)}\Phi^v_\tau(\omega).%\label{e-omegat}
\end{equation*}

Then, the following statements hold:
\begin{enumerate}
\item If $0< t_1<t_2$, then
\begin{equation*}
\omega^{t_1}\subset \omega^{t_2}.
%\label{e-scatole}
\end{equation*}
\item For each $t>0$ it holds $\omega\subset\omega^t$ and $\Phi^v_t(\omega)\subset\omega^t$.
\item The set $\omega^t$ is open.
\item Let $\mu^0$ be a probability measure, with compact support, absolutely continuous with respect to the Lebesgue measure and satisfying the first part of the Geometric Condition \ref{cond1}. Let $x_0\in\supp(\mu^0)$ and $t_0(x_0)$ the corresponding infimum time to enter $\omega$ as defined in \eqref{e-temps}. Let $\mathds{1}_{\omega}u$ be a control satisfying the Carath\'eodory condition. Then, $t>t_0(x_0)$ implies 
\begin{equation}
\Phi^{v+\mathds{1}_{\omega}u}_t(x_0)\in \omega^{t-t_0(x_0)}.
\label{e-l-omegat}
\end{equation}
\end{enumerate}

\end{lemma}
\begin{remark}
The key interest of estimate \eqref{e-l-omegat} is that the flow depends on the choice of the control $\mathds{1}_{\omega}u$, but that the set $\omega^{t-t_0(x_0)}$ does not depend on it.
\end{remark}

We are now ready to prove  \textbf{Item (ii)} of Theorem \ref{th opt}. Consider $$T\in(S^*(\mu^0,\mu^1),S(\mu^0,\mu^1)].$$ By definition of $S(\mu^0,\mu^1)$, it exists $m\in [0,1]$ such that 
\begin{equation}
\label{e-TT}
T<\mc{F}^{-1}_{\mu^0}(m)+\mc{B}^{-1}_{\mu^1}(1-m).\end{equation}
 Define $\bar t:=\mc{F}^{-1}_{\mu^0}(m)$ and the set 
$$A:=\left\{ x\in\supp(\mu^0)\mbox{~~s.t.~~}t_0(x_0) > \bar t\right\}.$$
By definition of $\mc{F}^{-1}_{\mu^0}$, it holds $\mu^0(A)=1-m$.

Fix now any control $\mathds{1}_{\omega}u$. By definition of the quantity $S^*(\mu^0,\mu^1)$, for each $x_0\in \supp(\mu^0)$, it holds $t_0(x_0)<T$. One can then apply Lemma \ref{l-omegat}, statements 1 and 4, in the two following cases:
\begin{itemize}
\item if $x_0\in \supp(\mu^0)\setminus A$, it holds $\Phi^{v+\mathds{1}_{\omega}u}_T(x_0)\in \omega^{T-t_0(x_0)}\subset \omega^T$;
\item if $x_0\in A$, it holds $\Phi^{v+\mathds{1}_{\omega}u}_T(x_0)\in \omega^{T-t_0(x_0)}\subset \omega^{T-\bar t}$. This implies that 
\begin{equation}
\Phi^{v+\mathds{1}_{\omega}u}_T(A)\subset \omega^{T-\bar t}.
\label{e-inclu}
\end{equation}
\end{itemize}

Consider now the solution $\mu(t)$ to \eqref{eq:transport} associated to the control $\mathds{1}_{\omega}u$. We already recalled that $\mu(t)=\Phi^{v+\mathds{1}_{\omega}u}_t\#\mu^0$. By definition of the push-forward and applying \eqref{e-inclu}, we can compute
\begin{equation}
\mu(T)(\omega^{T-\bar t})=\mu^0\left(\left(\Phi^{v+\mathds{1}_{\omega}u}_T\right)^{-1}\left(\omega^{T-\bar t}\right)\right)\geq \mu^0(A)=1-m.
\label{e-1-m}
\end{equation}

We now aim to prove that $\mu^1(\omega^{T-\bar t})<1-m$. Recall that $\omega^{T-\bar t}$ is open, by Lemma \ref{l-omegat}, statement 3. Take now $x_1\in \omega^{T-\bar t}\cap\supp(\mu^1)$, and observe that it holds  $x_1\in\Phi^v_t(\omega)$ for some $t\in(0,T-\bar t)$. This implies that $\Phi^v_{-t}(x_1)\in\omega$, hence $t_1(x_1)\leq t\leq T-\bar t$. Since such property holds for any $x_1\in \omega^{T-\bar t}\cap\supp(\mu^1)$, this implies that $\mc{B}_{\mu^1}(T-\bar t)\geq \mu^1(\omega^{T-\bar t})$, by definition of $\mc{B}_{\mu^1}$. If $\mu^1(\omega^{T-\bar t})\geq 1-m$, then $\mc{B}_{\mu^1}^{-1}(1-m)\leq T-\bar t$, hence $T\geq \bar t+\mc{B}_{\mu^1}^{-1}(1-m)$. This contradicts \eqref{e-TT}, thus it holds $\mu^1(\omega^{T-\bar t})<1-m$.

Recall now that $\partial(\omega^{T-\bar t})$ has zero Lebesgue measure by the additional hypothesis in Item (ii) of Theorem \ref{th opt}. Also recall that $\mu^1$ is absolutely continuous with respect to the Lebesgue measure. Thus, $\mu^1(\overline{\omega^{T-\bar t}})<1-m$, hence by standard approximation of measurable sets, there exists an open set $\mathcal{O}\supset \overline{\omega^{T-\bar t}}$ such that $\mu^1(\mathcal{O})<1-m$. Observe that such set depends on $\mu^1$ only, thus it does not depend on the choice of the control $u$. Also observe that, by a compactness argument, we have that
$$D:=\inf\{\|x-y\|\mbox{~~s.t.~~}x\in \overline{\omega ^{T-\bar t}}, y\not\in \mathcal{O}\}$$ is strictly positive.
%\textcolor{red}{J'enleverais tout cela: Indeed, consider the set of balls $B_\epsilon(x)$ with $x\in \overline{\omega ^{T-\bar t}}$ and $B_\epsilon(x)\subset \mathcal{O}$, that is a covering of the compact set $\overline{\omega ^{T-\bar t}}$. Then, there exists a finite covering $B_{\epsilon_n}(x_n)$ of $\overline{\omega ^{T-\bar t}}$, thus $D\geq \inf\epsilon_n>0$.}

We now prove that for each control $\mathds{1}_{\omega}u$, the associated solution $\mu(t)$ to \eqref{eq:transport}  satisfies
\begin{equation}W_1(\mu(T),\mu^1)\geq D  (1-m-\mu^1(\mathcal{O})).\label{e-item2}\end{equation}
The idea is to observe that the $1-m$ mass of $\mu(T)$ in $\omega^{T-\bar t}$ cannot be completely transferred to the set $\mathcal{O}$, since $\mu^1(\mathcal{O})<1-m$. Thus, a part of it needs to be transferred to $\mathbb{R}^d\setminus \mathcal{O}$, for which the distance of transfer is larger than $D $. More formally, consider a transference plan $\pi\in\Pi(\mu(T),\mu^1)$: recall that $\pi(\mathbb{R}^d\times \mathcal{O})=\mu^1(\mathcal{O})<1-m$, thus $\pi(\omega^{T-\bar t}\times \mathcal{O})\leq \mu^1(\mathcal{O})<1-m$, while $\pi(\omega^{T-\bar t}\times \mathbb{R}^d)=\mu(T)(\omega^{T-\bar t})\geq 1-m$ by \eqref{e-1-m}. Thus it holds 
$$\pi \big( \omega^{T-\bar t}\times( \mathbb{R}^d\setminus \mathcal{O}) \big) \geq 1-m-\mu^1(\mathcal{O}).$$
 Observe that $(x,y)\in \omega^{T-\bar t}\times( \mathbb{R}^d\setminus \mathcal{O})$ implies $\|x-y\|\geq D $. As a consequence, it holds
$$\int_{\mathbb{R}^d\times \mathbb{R}^d} \|x-y\|\,d\pi(x,y)\geq \int_{\omega^{T-\bar t} \times( \mathbb{R}^d\setminus \mathcal{O})} D \,d\pi(x,y)=D  (1-m-\mu^1(\mathcal{O})).$$
Passing to the infimum among all transference plans $\pi\in\Pi(\mu(T),\mu^1)$, it holds \eqref{e-item2}. Observe that neither $D$ nor $1-m-\mu^1(\mathcal{O})$ depend on the choice of the control $\mathds{1}_{\omega}u$, thus there exists no sequence $\mathds{1}_{\omega}u_n$ such that the associated solution $\mu_n$ satisfies $W_1(\mu_n(T),\mu^1)\to 0$. Hence, the system is not approximately controllable at time $T$.
%\hfill\qed
\end{proof}

\begin{remark}\label{rmq:T2* cont}
We give now two examples in which System \eqref{eq ODE} is never exactly controllable on 
 $[0,S^*(\mu^0,\mu^1))$  and another where System \eqref{eq ODE} is exactly controllable
 at each time $T\in[0,S^*(\mu^0,\mu^1))$.
 
 \noindent {\bf Figure \ref{fig:ex (0,T*) cont} (left).} Consider $\omega:=(-1,1)\times(-1.5,1.5)$.
The vector field $v$ is $(1,0)$, thus uncontrolled trajectories are right translations. Define %
\begin{equation*}
\left\{\begin{array}{l}
\mu^0:=\mathds{1}_{(-2.5,-2)\times(-1,1)}dx,\\
\mu^1:=\mathds{1}_{(2,2.5)\times(-1,1)}dx.
\end{array}\right.
\end{equation*}
The time $S^*(\mu^0,\mu^1)$ 
at which we can act on the particles and the minimal time $S(\mu^0,\mu^1)$ 
are respectively equal to $1.5$ and $2.5$.
We observe that for each time $T\in [0,S^*(\mu^0,\mu^1)]$
System \eqref{eq:transport} is not  approximately controllable. 
Indeed, each point takes a time $t>2$ to go from $\supp(\mu^0)$ to $\supp(\mu^1)$, hence one cannot expect approximate controllability for smaller times.

 \noindent {\bf Figure \ref{fig:ex (0,T*) cont} (right).} Consider $\omega:=(-1,1)\times(-1.5,1.5)$.
The vector field $v$ is $(-y,x)$, thus uncontrolled trajectories are rotations with constant angular velocity. 
Define
$$\mu^0=\mu^1:=\mathds{1}_{B_1(1,0)\backslash B_{0.5}(1,0)}dx. $$ 
In this case, both quantities $S^*(\mu^0,\mu^1)$ and $S(\mu^0,\mu^1)$ are  equal to $\pi$.
Since $\Phi^v_t\#\mu^0=\mu^1$ for all $t\geqslant 0$,
we remark that System \eqref{eq ODE} is exactly controllable for all $T\in [0,S^*(\mu^0,\mu^1))$.

\begin{figure}[ht]
\begin{center}
\hfill
\begin{tikzpicture}[scale=1]
\fill[pattern=dots,opacity = 0.5] (-2,-1.5) -- (0,-1.5) -- (0,1.5) -- (-2,1.5) -- cycle;
\draw (-2,-1.5) -- (0,-1.5) -- (0,1.5) -- (-2,1.5) -- cycle;
\fill[pattern=north east lines,opacity = 0.5] (-3,-1) -- (-3.5,-1) -- (-3.5,1) -- (-3,1) -- cycle;
\draw (-3,-1) -- (-3.5,-1) -- (-3.5,1) -- (-3,1) -- cycle;
\fill[pattern=north east lines,opacity = 0.5] (1,-1) -- (1.5,-1) -- (1.5,1) -- (1,1) -- cycle;
\draw (1,-1) -- (1.5,-1) -- (1.5,1) -- (1,1) -- cycle;
\path (-3.25,0) node {$\mu^0$};
\path (1.25,0) node {$\mu^1$};
\path (-1,-1) node {$\omega$};
\path (-1,1.95) node {$v$};
\draw (-1.5,1.75) -- (-0.5,1.75);
\draw (-0.5,1.75) -- (-0.6,1.85);
\draw (-0.5,1.75) -- (-0.6,1.65);
\draw (0.5,1.5) -- (1.5,1.5);
\draw (0.5,1.6) -- (0.5,1.4);
\draw (1.5,1.6) -- (1.5,1.4);
\path (1,1.75) node {\scriptsize 1};
\end{tikzpicture}
\hspace*{2cm}
\begin{tikzpicture}[scale=1]
\fill[pattern=dots,opacity = 0.5] (-2,-1.5) -- (0,-1.5) -- (0,1.5) -- (-2,1.5) -- cycle;
\draw(-2,-1.5) -- (0,-1.5) -- (0,1.5) -- (-2,1.5) -- cycle;
\draw  (0,-1) arc (90:450:-1) -- (0,-0.5) arc (90:450:-0.5) -- cycle;
\fill[pattern=north east lines,opacity = 0.5] (0,-1) arc (90:270:-1) -- (0,0.5) arc (-90:-270:-0.5) -- cycle;
\fill[pattern=north east lines,opacity = 0.5] (0,1) arc (90:270:1) -- (0,-0.5) arc (-90:-270:0.5) -- cycle;
\path (0.6,0.5) node {$\mu^1$};
\path (0.6,-0.5) node {$\mu^0$};
\path (-1,-1) node {$\omega$};
\path (1.15,-1.5) node {$v$};
\draw  (0.25,-1.25) arc (90:180:-1);
\draw (1.25,-0.25) -- (1.15,-0.35) ;
\draw (1.25,-0.25) -- (1.35,-0.35) ;
\end{tikzpicture}\hfill~
\caption{Left : The macroscopic system is not approximately controllable for each $T\in[0,S^*(\mu^0,\mu^1))$.
Right : The macroscopic system is approximately controllable for each $T\in[0,S^*(\mu^0,\mu^1))$.}\label{fig:ex (0,T*) cont}
\end{center}\end{figure}
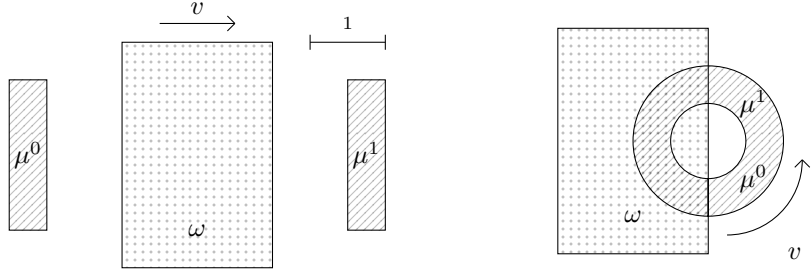

\end{remark}

\section{Numerical simulations}\label{sec:num sim}
In this section, we give some numerical examples of the algorithm 
developed in the proof of Theorem \ref{th opt} to compute the infimum time and the solution associated
to the minimal time problem to approximately steer an AC measure to another. We use a Lagrangian scheme, introduced in \cite{PR13,PT}, for simulations of transport equations.

We construct the mesh in Algorithm \ref{algo 0}. We then compute the minimal time in Algorithm \ref{algo 1}.

We first give Algorithm \ref{algo 0}. To simplify the notations, we assume that the space dimension is $d=2$.

\begin{center}\captionsetup{type=figure}
  \captionof{algorithm}{Construction of the meshes $\mathcal{T}_n^0$ 
  and $\mathcal{T}_n^1$}\label{algo 0}
  \begin{algorithmic}[0]
Let $n\in\mb{N}^*$ and two AC measures $\mu^0$, $\mu^1$ be given. Choose $\alpha_0,\alpha_1$ such that 
$\supp(\mu^0),~\supp(\mu^1)\subset(\alpha_0,\alpha_1)^2$.\\
\noindent\textbf{Step 1:} Construction, for all $k:=(k_1,k_2)\in\{0,...,n-1\}^2$, of the sets 
\begin{equation*}
\begin{array}{l}C_{n,k}:=
\left[\alpha_0+(\alpha_1-\alpha_0)\frac{k_1}{n},\alpha_0+(\alpha_1-\alpha_0)\frac{k_{1}+1}{n}\right)\\
\hspace*{2cm}\times\left[\alpha_0+(\alpha_1-\alpha_0)\frac{k_2}{n},\alpha_0+(\alpha_1-\alpha_0)\frac{k_{2}+1}{n}\right).
\end{array}
\end{equation*}\\
\textbf{Step 2:} 
Partition of $C_{n,k}$ into subsets $\{C_{nki}^0\}_{i}$ 
with $C_{nki}^0=[a_i^0,a_{i+1}^0)\times(\alpha_0,\alpha_1)$
such that  $\mu^0_{|C_{n,k}}(C_{nki}^0) =1/n^3$
(if $\mu^0_{|C_{n,k}}(C_{nki}^0) <1/n^3$, then we do not partition $C_{n,k}$)
and, for each $i$, partition of $C_{nki}^0$ into some subsets $\{C_{nkij}^0\}_{j}$ 
with $C_{nkij}^0=[a_i^0,a_{i+1}^0)\times[a_{ij}^0,a_{i(j+1)}^0)$
such that  $$\mu^0(C_{nkij}^0) =1/n^4.$$
 See Figure \ref{fig: mesh}. We define similarly the cells $C_{nkij}^1$.
\smallskip\\
\textbf{Step 3:} Construction of
$B^0_{nkij}:=[b_i^{0-},b_{i}^{0+})\times[b_{ij}^{0-},b_{ij}^{0+})\subset\subset C^0_{nkij}$ 
and  $B^1_{nkij}:=[b_i^{1-},b_{i}^{1+})\times[b_{ij}^{1-},b_{ij}^{1+})\subset\subset C^1_{nkij}$ 
such that 
\begin{equation*}
\mu^0(B^0_{nkij})=\mu^1(B^1_{nkij})=\frac{1}{n^4}-\frac{1}{n^6}.
\end{equation*}
See Figure \ref{fig:cell}.
\smallskip\\
\textbf{Step 4:} Definition of  $\mc{T}^0_n:=\bigcup_{kij\in I_n^0} B_{nkij}^0$ and  $\mc{T}^1_n:=\bigcup_{kij\in I_n^1} B_{nkij}^1$,
where, for $l=0,1$, $I_n^l$ is the set of $kij$ such that $B_{nkij}^l$ is well defined. 
  \end{algorithmic}
 \smallskip \hrule
\end{center}

\medskip

We now recall the algorithm to compute the minimal time. 
We will assume that $\omega$ is convex in order to use the constructive proof of
Proposition \ref{prop: dim finie}, \textit{i.e.}  the agents cross the control area following straight trajectories.

\begin{center}\captionsetup{type=figure}
  \captionof{algorithm}{Approximate controllability}\label{algo 1}
  \begin{algorithmic}[0]
Let $\mu^0$ and $\mu^1$ be two AC measures  satisfying the Geometric Condition \ref{cond1}.
\smallskip\\
\noindent\textbf{Step 1:} Construction of the meshes $\mc{T}^0_n:=\bigcup_{kij\in I_n^0} B_{nkij}^0$ and  $\mc{T}^1_n:=\bigcup_{kij\in I_n^1} B_{nkij}^1$
following Algorithm \ref{algo 0}.
\smallskip\\
\textbf{Step 2:}
Definition  of $X^l:=\{x^l_{kij}:kij\in I_n^l\}$ ($l=0,1$)  with $x^l_{nkij}$ being a point of $B_{kij}^l\cap\Supp(\mu^l)$
\smallskip\\
\textbf{Step 3:} Definition of
\begin{equation*}%
\omega_n:=\{x\in\mb{R}^d:d(x,\omega^c)>e^{LT}\sqrt{2}/n\}.
\end{equation*} 
 For $t_{nkij}^l:=t^l(x^l_{nkij},\omega_n)$ with $l=0,1$,
computation of the minimal time to steer $X^0$ to $X^1$ up to a mass $\varepsilon/4R$
\begin{equation*}
M_{e,\varepsilon/4R}(X^0,X^1,\gamma):=\max\limits_{1\leq i\leq n-R} \{t^0_{nkij}+t^1_{nkij+R}\},
\end{equation*}
where $R:=\lfloor n\varepsilon \gamma/4R \rfloor$
and the sequences $\{t_{nkij}^0\}_{kij}$, $\{t_{nkij}^1\}_{kij}$ are increasingly and decreasingly ordered, respectively\smallskip\\
\textbf{Step 4:} Computation of the optimal permutations $\sigma_0$ and $\sigma_1$ 
minimizing \eqref{eq:inf2} to steer $X^0$ to $X^1$ up to a mass $\varepsilon/4R$.\smallskip\\
\textbf{Step 5:} Concentration of the mass of $B_{nkij}^0$ around $x_{nkij}^0$  in order to obtain no intersection of the cells
when they follow the trajectories of $\delta_{x^0_{nkij}}$ \smallskip\\
\textbf{Step 6:} Computation of the control $u_n$ and the solution $\mu_n$ to \eqref{eq:transport} on $(0,T)$.
  \end{algorithmic}  \smallskip\hrule
\end{center}

\subsection{Example 1: the 1-D case}
Consider the initial data $\mu^0$ and the target $\mu^1$ defined by
\begin{equation*}
\left\{\begin{array}{l}
\mu^0:=0.5\times\mathds{1}_{(0,2)}dx,\\\noalign{\smallskip}
\mu^1:=0.5\times\mathds{1}_{(7,8)\cup(10,11)}dx.
\end{array}\right.
\end{equation*}
Let the velocity field $v:=1$ and the control region $\omega:=(5,6)$ be given.
This situation is illustrated in Figure \ref{fig:situation1D}.
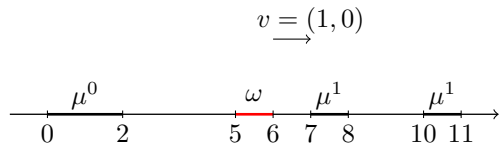
\begin{figure}[ht]
\begin{center}
\begin{tikzpicture}[scale=0.5]
\path (5.5,0.5) node {$\omega$};
\path (1,0.5) node {$\mu^0$};
\path (7.5,0.5) node {$\mu^1$};
\path (10.5,0.5) node {$\mu^1$};
\draw[->,line width = 0.5pt] (-1,0) -- (12,0);
\draw[line width = 1pt] (0,0) -- (2,0);
\draw[line width = 1pt] (7,0) -- (8,0);
\draw[line width = 1pt] (10,0) -- (11,0);
\draw[line width = 1pt,color=red] (5,0) -- (6,0);
\path (0,-0.5) node {$0$};
\path (2,-0.5) node {$2$};
\path (5,-0.5) node {$5$};
\path (6,-0.5) node {$6$};
\path (7,-0.5) node {$7$};
\path (8,-0.5) node {$8$};
\path (10,-0.5) node {$10$};
\path (11,-0.5) node {$11$};
\draw[-] (0,-0.1) -- (0,0.1);
\draw[-] (2,-0.1) -- (2,0.1);
\draw[-] (5,-0.1) -- (5,0.1);
\draw[-] (6,-0.1) -- (6,0.1);
\draw[-] (7,-0.1) -- (7,0.1);
\draw[-] (8,-0.1) -- (8,0.1);
\draw[-] (10,-0.1) -- (10,0.1);
\draw[-] (11,-0.1) -- (11,0.1);
\path (7,2.5) node {$v=(1,0)$};
\draw[->] (6,2) -- (7,2);
\end{tikzpicture}
\end{center}
\caption{Control set $\omega$, initial configuration $\mu^0$ and final one $\mu^1$ for Example 1.}
\label{fig:situation1D}
\end{figure}
In this case, the infimum time $T_a(\mu^0,\mu^1)$ is 8, which is computed in Step 3 of Algorithm \ref{algo 1}. One cannot achieve approximate control at such time, but we aim to control the system at time $T=T_a(\mu^0,\mu^1)+\delta$, with
$\delta:=0.1$, hence $T=8.1$.
Following  Algorithm \ref{algo 1}, we obtain the solution presented in Figure \ref{fig:simu1D}.
The maximal density for this solution is equal to $1.1$. It is due to the fact that we concentrate a part of the mass coming from the set $\{x<5\}$ in the control set, to slow it down. This increase of the maximal density can be seen as a drawback of the method for several key applications, namely for egress problems. Indeed, high concentrations need to be carefully avoided in such settings, since they might induce death by suffocation, that is among the main causes of fatalities in stampedes/crushes (see for instance \cite{helbing2002crowd}).

For this reason, in the future we plan to study new control strategies for minimal time problems, in which a constraint on the maximal density is added. Alternatively, we aim to estimate the maximal density value that is reached with optimal strategies.

\begin{figure}[ht]\begin{center}
\hspace*{-6mm}\includegraphics[scale=0.5]{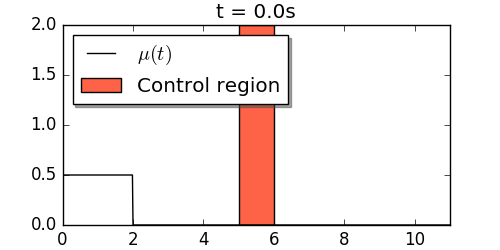}
\hspace*{-6mm}\includegraphics[scale=0.5]{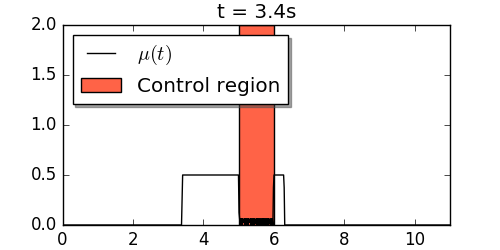}\\
\hspace*{-6mm}\includegraphics[scale=0.5]{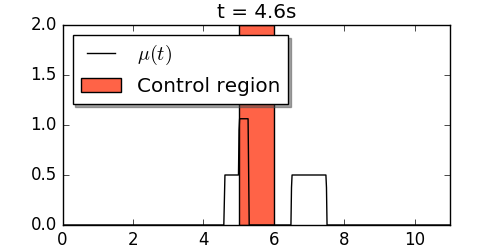}
\hspace*{-6mm}\includegraphics[scale=0.5]{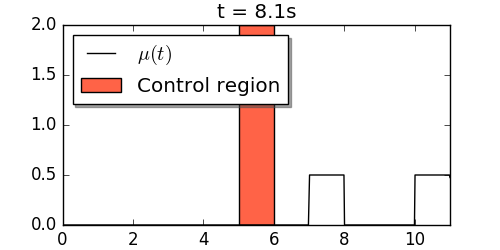}\end{center}
\caption{Example 1: solution at time $t=0$, $t=3.4$, $t=4.6$ and $t=T=8.1$.}
\label{fig:simu1D}\end{figure}

\subsection{Example 2: the 2D case}

We now give an example in the 2D case. Consider the initial data $\mu^0$ and the target $\mu^1$ defined by\begin{equation*}
\left\{\begin{array}{l}
\mu^0:=\frac{1}{8}\times\mathds{1}_{(0,4)\times (1,3)}dx,\\\noalign{\smallskip}
\mu^1:=\frac{1}{16}\times\mathds{1}_{(8,14)\times (0,4)\backslash (9,13)\times(1,3)}dx.
\end{array}\right.
\end{equation*}
We fix the velocity field $v:=(1,0)$ and the control region 
$\omega:=(5,7)\times (0,4)$.
This situation is illustrated in Figure \ref{fig:situation2D}. 
Again, in this case $T_a(\mu^0,\mu^1)=8$. Since it is not possible to approximatively steer $\mu^0$ to $\mu^1$ at such time, we control the system at time $T=T_a(\mu^0,\mu^1)+\delta$, with $\delta:=1$, hence $T=9$. 
Following Algorithm \ref{algo 1}, we present the solution in Figure \ref{fig:simu2D}.
As in the previous example, we observe a high concentration of the crowd in the control region, in this case
with a maximal density equal to $8$.

\begin{figure}
\begin{center}
\begin{tikzpicture}[scale=0.5]
\draw (0,1) -- (4,1) -- (4,3) -- (0,3) -- cycle;
\draw[pattern=north east lines,opacity=0.5] (0,1) -- (4,1) -- (4,3) -- (0,3) -- cycle;
\path (2,2) node {$\mu^0$};

\fill[pattern=dots,opacity=0.5] (5,0) -- (7,0) -- (7,4) -- (5,4) -- cycle;
\path (6,2) node {$\omega$};
\draw (8,0) -- (14,0) -- (14,4) -- (8,4) -- cycle;
\draw (9,1) -- (13,1) -- (13,3) -- (9,3) -- cycle;
\fill[pattern=north east lines,opacity=0.5] (8,0) -- (14,0) -- (14,1) -- (8,1) -- cycle;
\fill[pattern=north east lines,opacity=0.5] (8,3) -- (14,3) -- (14,4) -- (8,4) -- cycle;
\fill[pattern=north east lines,opacity=0.5] (8,1) -- (9,1) -- (9,3) -- (8,3) -- cycle;
\fill[pattern=north east lines,opacity=0.5] (13,1) -- (14,1) -- (14,3) -- (13,3) -- cycle;
\path (11,3.5) node {$\mu^1$};

\draw[->,line width = 1pt] (0,0) -- (15,0);
\draw[->] (0,0) -- (0,5);
\draw(5,0) -- (7,0) -- (7,4) -- (5,4) -- cycle;

\path (0,-0.5) node {$0$};
\path (4,-0.5) node {$4$};
\path (5,-0.5) node {$5$};
\path (7,-0.5) node {$7$};
\path (8,-0.5) node {$8$};
\path (14,-0.5) node {$14$};

\path (-0.5,0) node {$0$};
\path (-0.5,1) node {$1$};
\path (-0.5,3) node {$3$};
\path (-0.5,4) node {$4$};

\path (7,5.5) node {$v=(1,0)$};

\draw[->] (6.5,5) -- (7.5,5);

\end{tikzpicture}
\end{center}
\caption{Control set $\omega$, initial configuration $\mu^0$ and final one $\mu^1$ for Example 2.}
\label{fig:situation2D}
\end{figure}
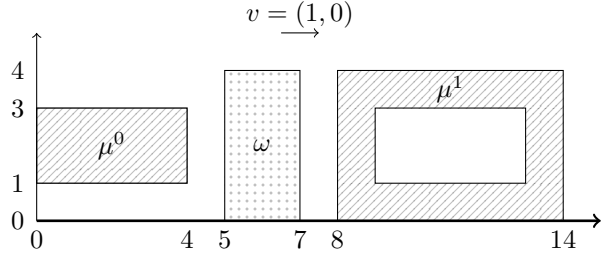

\begin{figure}[ht]\begin{center}
\hspace*{-6mm}\includegraphics[scale=0.45]{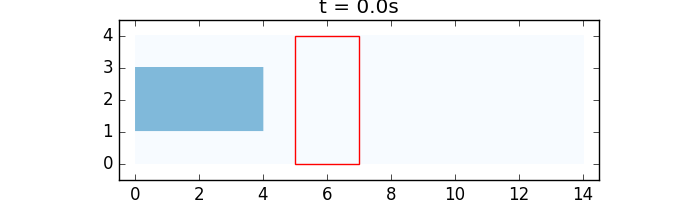}
\hspace*{-6mm}\includegraphics[scale=0.45]{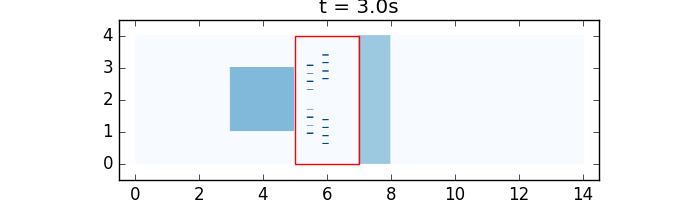}

\smallskip

\hspace*{-6mm}\includegraphics[scale=0.45]{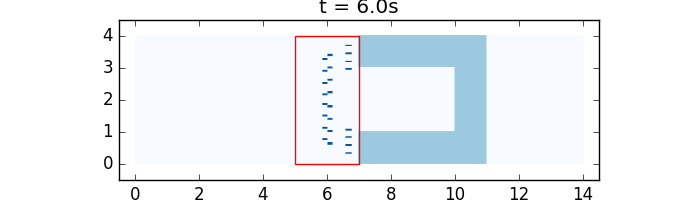}
\hspace*{-6mm}\includegraphics[scale=0.45]{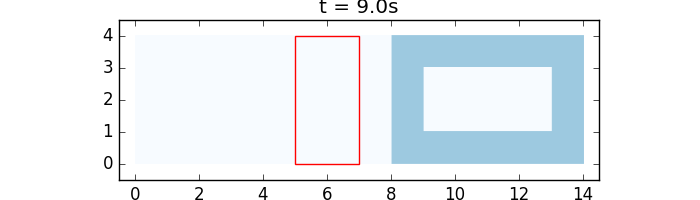}\end{center}
\caption{Example 2: solution at time $t=0$, $t=3$, $t=6$ and $t=T=9$.}
\label{fig:simu2D}\end{figure}

\appendix

\section{Proof of Lemma \ref{l-omegatreg}}  %\label{a-1}
  \setcounter{theorem}{0}
     \renewcommand{\thetheorem}{\Alph{section}.\arabic{theorem}}
     
In this section, we prove Lemma \ref{l-omegatreg}.  We first give the definition of the uniform interior cone condition and state a related result.
%{\bf Statement 1} directly follows from the definition.

\begin{definition} \label{d-intcone} Let $$C(x_0,n,\alpha,r):=\left\{ \begin{array}{l}
x\in B_r(x_0)\mbox{~~s.t.~~}n\cdot (x-x_0) >0,\\
~~~\|x-x_0-n(n\cdot (x-x_0))\|<\alpha\, n\cdot (x-x_0)\end{array}\right\}$$ be a (open) cone centered in $x_0$, with normal unit vector $n$, tangent $\alpha$ and radius $r$.

We say that a set $A$ satisfies the uniform interior cone condition if there exist uniform $\alpha,r$ such that for all $x_0\in \partial A$ there exists $n$ such that $$C(x_0,n,\alpha,r)\subset A.$$
\end{definition}

\begin{lemma}
\label{l-misura}
Let $v$ be a $C^1$ vector field, and $A$ an open bounded set satisfying the uniform interior cone condition. Then, the following statements hold.
\begin{enumerate}
\item The set $\Phi^v_t(A)$ satisfies the uniform interior cone condition;
\item For each $t>0$, the set $A^t = \cup_{ \tau \in (0,t)}\Phi_{\tau}^v(A)$ satisfies the uniform interior cone condition.
\end{enumerate}

\end{lemma}
The proof is omitted. It can be easily recovered by a first-order expansion of the flow, in which parameters depend continuously on the point, since $v$ is chosen to be $C^1$.

We are now ready to prove Lemma \ref{l-omegatreg}. 
Since $\omega$ satisfies the uniform interior cone condition, by Lemma \ref{l-misura}, Statement 2, it holds that $\omega^t$ satisfies the uniform interior cone condition, with uniform parameters $\alpha',r'$. We now prove that this implies that $\partial(\omega^t)$ has zero Lebesgue measure. Since $\omega^t$ is open and bounded, then it is measurable, thus its characteristic function $\mathds{1}_{\omega^t}$ is in $L^1(\mathbb{R}^d)$. It is now sufficient to prove that no point of the boundary is a Lebesgue point with respect to the Lebesgue measure. Indeed, observe from one side that for each $x\in \partial(\omega^t)$ it holds $\mathds{1}_{\omega^t}(x)=0$ since $\omega^t$ is open. On the other side, for $r<r'$ it holds $$\frac{1}{\mathcal{L}(B_r(x))}\int_{B_r(x)}\mathds{1}_{\omega^t}(y)d\mathcal{L}(y)\geq \frac{\mathcal{L}(C(x,m,\alpha',r')\cap B_r(x))}{\mathcal{L}(B_r(x))}.$$ By a simple dilation and rototranslation argument, we remark that such term coincides with $\frac{\mathcal{L}(C(0,m,\alpha',1))}{\mathcal{L}(B_1(0))}$, that is a strictly positive constant, independent on $r$. As a consequence, it holds $$0=\mathds{1}_{\omega^t}(x)\neq \lim_{r\to 0}\frac{1}{\mathcal{L}(B_r(x))}\int_{B_r(x)}\mathds{1}_{\omega^t}(y)d\mathcal{L}(y),$$ hence $x$ is not a Lebesgue point. Since $x\in \partial(\omega^t)$ is generic, then no point in $\partial(\omega^t)$ is a Lebesgue point. Since the set of points that are not Lebesgue points for a measurable function has zero Lebesgue measure (Lebesgue-Besicovitch differentiation theorem, see \textit{e.g.} \cite[Sec. 1.7]{evgar}), then $\partial(\omega^t)$ has zero Lebesgue measure.

%\bibliographystyle{AIMS}
%\bibliography{references}

\end{document}